\documentclass[10pt]{article}
\usepackage{amsmath,amsthm,mathrsfs,bm}

\usepackage[initials,alphabetic]{amsrefs}
\usepackage[utf8]{inputenc}
\usepackage[T1]{fontenc}

\usepackage{amsfonts,amstext,amssymb,verbatim,epsfig,psfrag,stmaryrd,extarrows}
\usepackage{pgf,tikz}
\usepackage{float}
\usetikzlibrary{arrows}
\usepackage{hyperref}
\usepackage{color}
\usepackage{transparent}
\usepackage{subfigure}
\usepackage{prettyref}
\newrefformat{pr}{Proposition \ref{#1}}
\newrefformat{co}{Corollary \ref{#1}}
\newrefformat{le}{Lemma \ref{#1}}
\newrefformat{th}{Theorem \ref{#1}}
\newrefformat{ex}{Example \ref{#1}}
\newrefformat{re}{Remark \ref{#1}}
\newrefformat{se}{Section \ref{#1}}

\usepackage{esint}

\usepackage{bbm,mathrsfs}

\usepackage{xcolor}
\usepackage[normalem]{ulem}

\hypersetup{colorlinks=true,pdfborder=000,  citecolor  = blue, citebordercolor= {magenta}}
\hypersetup{linkcolor=purple}
\makeatletter
\let\reftagform@=\tagform@
\def\tagform@#1{\maketag@@@{(\ignorespaces\textcolor{purple}{#1}\unskip\@@italiccorr)}}
\renewcommand{\eqref}[1]{\textup{\reftagform@{\ref{#1}}}}
\makeatother

\makeatletter

\allowdisplaybreaks

\DeclareUrlCommand\ULurl@@{%
  \def\UrlLeft{\uline\bgroup}%
  \def\UrlRight{\egroup}}
\def\ULurl@#1{\hyper@linkurl{\ULurl@@{#1}}{#1}}
\DeclareRobustCommand*\ULurl{\hyper@normalise\ULurl@}
\makeatother

\def\lessim{\ \lower4pt\hbox{$
		\buildrel{\displaystyle <}\over\sim$}\ }
\def\gessim{\ \lower4pt\hbox{$\buildrel{\displaystyle >}
		\over\sim$}\ }

\def\si{\sigma}

\def\eps{{\varepsilon}}

\newcommand{\indi}{\ensuremath{\boldsymbol 1}}

\newcommand{\bt}{\boldsymbol{t}}

\newcommand{\jx}{\mathcal{J}}

\newcommand{\pref}{\prettyref}

\newtheorem{lemma}{\bf Lemma}[section]

\newtheorem{theorem}[lemma]{\bf Theorem}

\newtheorem{example}{\bf Example}

\newtheorem{proposition}[lemma]{\bf Proposition}

\theoremstyle{remark}

\newtheorem{remark}{Remark}[section]

\numberwithin{equation}{section}

\newcommand{\8}{\infty}

\newcommand{\ix}{\mathcal{I}}

\newcommand{\px}{\mathcal{P}}

\newcommand{\nz}{\mathbb{N}}
\newcommand{\rz}{\mathbb{R}}

\newcommand{\ez}{\mathbb{E}}

\newcommand{\zz}{\mathbb{Z}}
\newcommand{\pz}{\mathbb{P}}

\newcommand{\Ga}{\Gamma}

\newcommand{\sfT}{\mathsf T}

\newcommand{\sfa}{\mathsf a}
\newcommand{\sfb}{\mathsf b}
\newcommand{\sfm}{\mathsf m}

\newcommand{\al}{\alpha}
\newcommand{\de}{\delta}
\renewcommand{\si}{\sigma}
\newcommand{\ga}{\gamma}
\newcommand{\la}{\lambda}
\renewcommand{\bt}{\beta}
\newcommand{\ta}{\theta}

\newcommand{\Crt}{\mathrm{Crt}}

\newcommand{\Cov}{\mathrm{Cov}}
\newcommand{\GOE}{\mathrm{GOE}}
\newcommand{\sgn}{\mathrm{sgn}}

\newcommand{\dd}{\mathrm{d}}

\newenvironment{Proof of lemma}{\noindent{\bf Proof of Lemma}}{\hfill$\Box$\newline}
\newenvironment{Proof of theorem}{\noindent{\bf Proof of Theorem}}{\hfill{\footnotesize${\square}$}\newline}
\newenvironment{Proof of theorems}{\noindent{\bf Proof of Theorems}}{\hfill$\Box$\newline}
\newenvironment{Proof of proposition}{\noindent{\bf Proof of Proposition}}{\hfill$\Box$\newline}
\newenvironment{Proof of propositions}{\noindent{\bf Proof of Propositions}}{\hfill$\Box$\newline}
\newenvironment{Proof of exercise}{\noindent{\it Proof of Exercise:}}{\hfill$\Box$}
\begin{document}
\title{Complexity of Gaussian random fields with isotropic increments: critical points with given indices}
\author{Antonio Auffinger \thanks{Department of Mathematics, Northwestern University, tuca@northwestern.edu, research partially supported by NSF Grant CAREER DMS-1653552 and NSF Grant DMS-1517894.} \\
	\small{Northwestern University}\and Qiang Zeng \thanks{Department of Mathematics, University of Macau, qzeng.math@gmail.com, research partially supported by SRG 2020-00029-FST and FDCT 0132/2020/A3.}\\
	\small{University of Macau}}
\date{\today}

\maketitle
\abstract{We study the landscape complexity of the Hamiltonian $X_N(x) +\frac\mu2 \|x\|^2,$ where $X_{N}$ is a smooth Gaussian process with isotropic increments on $\mathbb R^{N}$. This model describes a single particle on a random potential in statistical physics. We derive asymptotic formulas for the mean number of critical points of index $k$ with critical values in an open set as the dimension $N$ goes to infinity. In a companion paper, we provide the same analysis without the index constraint.}

\tableofcontents

\section{Introduction}

This is the second part of two papers on the landscape complexity of locally isotropic Gaussian random fields (a.k.a.~Gaussian random fields with isotropic increments). In this second article, we investigate the complexity of critical values with a given index.

The model is defined as follows. Let $B_N \subset \mathbb R^N$ be a sequence of subsets and  let $H_N : B_N \subset \mathbb R^N \to \mathbb R$ be given by
\begin{align}\label{hamil}
  H_N(x) = X_N(x) +\frac\mu2 \|x\|^2,
\end{align}
where $\mu \in \rz$, $\|x\|$ is the Euclidean norm of $x$, and $X_N$ is a Gaussian random field that satisfies
\begin{align*}
  \ez[(X_N(x)-X_N(y))^2]=N D\Big(\frac1N \|x-y\|^2\Big), \ \ x,y\in \rz^N.
\end{align*}
Here the function $D:\mathbb R_+ \to \mathbb R_+$ is called the correlator (or structure) function and $\rz_+=[0,\8)$.
The correlator $D$ has a representation
      \begin{align}\label{eq:drep}
        D(r) = \int_{(0,\8)} (1-e^{-rt^2})\nu(\dd t) + Ar, \ \ r\in \rz_+,
      \end{align}
      where $A\in \rz_+$ is a constant and $\nu$ is a $\si$-finite measure with
      \begin{align*}
        \int_{(0,\8)} \frac{t^2}{1+t^2}\nu(\dd t) <\8.
      \end{align*}
The case where $\nu(0,\8)=\8$ is called a long range correlated (LRC) Gaussian field and is the focus of this paper. We refer the reader to \cite{Ya87}*{Section 25.3} for more details on locally isotropic Gaussian fields.


In \cite{AZ20}, we considered the mean number of critical points of LRC fields. Here, we will study critical points with given indices. To state our main results, let $E\subset \rz$ be a Borel set. We define
\begin{align*}
\mathrm{Crt}_{N,k}(E,B_N) &= \#\{x\in B_N: \nabla H_N(x)=0, \frac1NH_N(x)\in E, i(\nabla^2 H_N(x))=k\},\\
	\mathrm{Crt}_N(E,B_N) &= \#\{x\in B_N: \nabla H_N(x)=0, \frac1N H_N(x)\in E \},
\end{align*}
where $i(\nabla^2 H_N(x))$ is the index (or number of negative eigenvalues) of the Hessian $\nabla^2 H_N(x)$.

 Throughout the paper we will consider the following extra assumptions on $X_N$.

\textbf{Assumption I} (Smoothness). The function $D$ is four times differentiable at $0$ and it satisfies
\begin{align*}
	0<|D^{(4)}(0)|<\8.
\end{align*}


\textbf{Assumption II} (Pinning). We have $$X_N (0) = 0.$$


These two assumptions are natural and necessary for our study on complexity; see \cite{AZ20} for a discussion. We state our main results in two separate batches. The first set counts the average of the {\it total} number of local minima and saddles of $H_N$. The second set counts the average number of local minima and saddles of $H_N$ with a given fixed critical value. Although the results of the first set can be essentially recovered from the second, formulas and proofs for the first set are much clearer thus we state them separately. We hope this organization provides a gentle introduction to the reader to appreciate the second set of results, where most of the novelty (and difficulty of the paper) resides.


The following condition is only needed when the critical value is not restricted, which simply provides the correct scaling for domains of the random fields.

\textbf{Assumption III} (Domain growth). Let $z_N$ be a standard $N$ dimensional Gaussian random variable. There exist $\Xi$ or $\Theta$ such that the sequence of sets $B_N$ satisfies
\begin{align}
\lim_{N\to \8} \frac1N \log \pz(  z_N \in  |\mu| B_N/\sqrt{D'(0)} ) &= -\Xi\le 0, &\mu\neq 0,\label{eq:bnasp1}\\
\lim_{N\to \8} \frac1N \log |B_N|& = \Theta, &\mu=0.\label{eq:bnasp2}
\end{align}

Our first main result counts the total number of critical points of a fixed index $k$.

\begin{theorem}\label{th:fixktotal} 
Assume Assumptions I, II, and III. Let $k\in\zz_+=\{0,1,2,...\}$. Then we have
\[
\lim_{N\to\8}\frac1N\log\ez \Crt_{N,k}(\rz,B_N)= \begin{cases}
\frac{\mu^2}{4D''(0)} -\log\frac{|\mu|}{\sqrt{-2D''(0)}}-\frac12-\Xi +I_k, &\mu\neq0,  \\
\log\sqrt{-2D''(0)}-\frac32-\frac12\log(2\pi)- \frac12\log[D'(0)]+\Theta,& \mu=0,
\end{cases}
\]
where the positive constant $I_k$ is given in \eqref{ikcase}.
\end{theorem}
\begin{remark}\label{re:25}
We will see in the proof that the sequence of  constants $(I_k)_{k\geq0}$ is strictly decreasing and $I_k\ge1$ when $\mu>\sqrt{-2 D''(0)}$. When $k=0$ and $\mu\neq0$, we have
\[
I_0=\frac{\mu^2}{-4D''(0)}+\frac12+\log\frac{|\mu|}{\sqrt{-2D''(0)}},
\]
and it follows that $\lim_{N\to\8}\frac1N\log\ez \Crt_{N,0}(\rz,B_N)=-\Xi$, which coincides with the complexity of total number of critical points for $\mu>\sqrt{-2D''(0)}$ obtained in \cite{AZ20}*{Theorem 1.1}. When $\mu\le\sqrt{-2 D''(0)}$, $I_k$ does not depend on $k$, which also includes the case $\mu=0$.

The cases $k\ge1$ and $k=0$ do not have a unified expression for $I_k$.  The reason for this unusual phenomenon will be more clear when we recover this theorem from the general Theorems \ref{th:criticalfix1} and \ref{th:fixk}.
\end{remark}

\begin{remark}
By symmetry, one can get the complexity of critical points with index $N-k$ for fixed $k\in\zz_+$.
\end{remark}

Next, we consider the total number of critical points of diverging index $k=k_N$. Let $\ga\in(0,1)$ and $k=k_N$ a sequence of integers such that
\[
\lim_{N\to\8} \frac{k_N}N =\ga.
\]
Let $s_\ga$ be the $\ga$th quantile of the semicircle law, i.e.
\begin{align}\label{eq:sgasc}
\frac1\pi\int_{-\sqrt{2}}^{s_\ga} \sqrt{2-x^2} \dd x = \ga.
\end{align}

\begin{theorem}\label{th:divtotal}
Assume Assumptions I, II, and III. Let $k_N,\ga$ and $s_\ga$ be as above. Then we have
\begin{align*}
\lim_{N\to\8}&\frac1N\log\ez \Crt_{N,k_N}(\rz,B_N)\\
&= \begin{cases}
\frac{\mu^2}{4D''(0)} -\log\frac{|\mu|}{\sqrt{-2D''(0)}}-\frac12 -\frac12 s_\ga^2-\frac{\mu s_\ga}{\sqrt{-D''(0)}} -\Xi, &\mu\neq0,\\
\log\sqrt{-2D''(0)}-\frac12-\frac12\log(2\pi)- \frac12\log[D'(0)]  -\frac12 s_\ga^2 +\Theta,& \mu=0.
\end{cases}
\end{align*}
\end{theorem}
\begin{remark}
    In particular, the complexity of total number of critical points obtained in \cite{AZ20}*{Theorem 1.1} when $|\mu|\le \sqrt{-2D''(0)}$ can be recovered by considering the supremum over all $\ga\in(0,1)$.

\end{remark}


In the following, we refine previous results by imposing restrictions on the location of the critical values of $H_N$. More precisely, we concern the number of minima/saddles with critical values in an open set $E\subset \mathbb R$ and confined to a shell $B_{N}(R_{1},R_{2})=\{ x\in \mathbb R^{N}: R_{1}< \frac{\| x\|}{\sqrt N} < R_{2} \}$. The shell is a natural choice, as the isotropy assumption implies rotational invariance. To emphasize the dependence on $R_1$ and $R_2$, we also write
\begin{align}
  \Crt_{N}(E, (R_1,R_2)) & =\Crt_{N}(E, B_{N}(R_1,R_2)), \notag\\
  \Crt_{N,k}(E, (R_1,R_2)) & =\Crt_{N,k}(E, B_{N}(R_1,R_2)). \label{eq:crter12}
\end{align}
We assume Assumptions I, II, and the following technical assumption for the general theorems:

\textbf{Assumption IV}.
For any $x\in \rz^N\setminus\{0\}$, we have
\begin{align}
  -2D''(0)&> \left(\frac{\al\|x\|^2}{N}+\bt\right)\bt,\label{eq:asmp1} \\
  -4D''(0)&> \left(\frac{\al\|x\|^2}{N}+\bt\right)\frac{\al \|x\|^2}N,\label{eq:asmp2}
\end{align}
where $\al$ and $\bt$ are defined in \pref{eq:albt0}.

This assumption is rather mild, and is satisfied by e.g.~the so called Thorin--Bernstein functions; see \cite{AZ20}*{Section 3} and \cite{SSV} for more details. The next result provides the asymptotic behavior of the number of local minima.

\begin{theorem}\label{th:criticalfix1}
  Let $0\le R_1<R _2\le \8$ and $E$ be an open set of $\rz$.  	Assume Assumptions I, II, IV, and $|\mu|+\frac1{R_2} > 0$.  Then
  \begin{align*}
   \lim_{N\to\8}\frac1N\log \ez\Crt_{N,0}(E,(R_1,R_2))= & \frac12\log[-4D''(0)]-\frac12\log D'(0)+\frac12 \\
      &+\sup_{(\rho,u,y)\in F} [\psi_*(\rho,u,y)- \ix^-(\rho,u,y)],
 \end{align*}
 where $F=\{(\rho,u,y): \rho\in (R_1,  R_2), u\in  E, y\le -\sqrt2\}$ and the functions $\psi_*$ and $\ix^-$ are defined in \eqref{eq:psifunction} and \eqref{eq:ixdef} respectively.
\end{theorem}
The condition $|\mu|+\frac1{R_2} > 0$ merely says $R_2<\8$ if $\mu=0$, which is necessary to get non-trivial asymptotics as we saw in Assumption III. In Example \ref{ex:3} at the end of \pref{se:k=0}, we show how to recover the $k=0$ case of \pref{th:fixktotal} from \pref{th:criticalfix1}. The next result establishes the complexity of saddles with fixed index.

\begin{theorem}\label{th:fixk}
  Let $0\le R_1<R _2\le \8$ and $E$ be an open set of $\rz$.  	Assume Assumptions I, II, IV, and $|\mu|+\frac1{R_2} > 0$.   Then for any fixed $k\in\nz$,
  \begin{align*}
   &\lim_{N\to\8}\frac1N\log \ez\Crt_{N,k}(E,(R_1,R_2))=   \frac12\log[-4D''(0)]-\frac12\log D'(0)+\frac12 \\
      &+\max\Big\{\sup_{(\rho,u,y)\in F} [\psi_*(\rho,u,y)-k J_1(y)],  ~  \sup_{(\rho,u,y)\in F} [\psi_*(\rho,u,y)- \ix^+(\rho,u,y)- (k-1) J_1(y)] \Big\},
 \end{align*}
 where $F=\{(\rho,u,y): \rho\in (R_1,  R_2), u\in  E, y\le -\sqrt2\}$ and the functions $\psi_*$, $\ix^+$, and $J_1$ are given by \eqref{eq:psifunction}, \eqref{eq:ix+} and \eqref{eq:rf1ev}, respectively.
\end{theorem}

In Example \ref{ex:31} at the end of \pref{se:kge1}, we provide details on how to recover the $k\ge1$ cases of Theorem \ref{th:fixktotal} from \pref{th:fixk}, which amounts to solving the involved optimization problem in \pref{th:fixk}. Theorems \ref{th:criticalfix1} and \ref{th:fixk} also show the evident difference between the cases $k=0$ and $k\ge1$, which explains the inconsistent formulas for $I_k$ in \pref{th:fixktotal}.  Our last result concerns the complexity of saddles with diverging index. 

\begin{theorem}\label{th:kdiv}
  Let $0\le R_1<R _2\le \8$ and $E$ be an open set of $\rz$.  	Assume Assumptions I, II, IV, $|\mu|+\frac1{R_2} > 0$, and $\lim_{N\to\8}\frac{k_N}{N}=\ga\in(0,1)$. Then
  \begin{align*}
   &\lim_{N\to\8}\frac1N\log \ez\Crt_{N,k_N}(E,(R_1,R_2))\\
   &=    \frac12\log[-4D''(0)]-\frac12\log D'(0)+\frac12 + \sup_{R_1<\rho<R_2, u\in  E  }\psi_*(\rho,u,s_\ga),
 \end{align*}
 where the function $\psi_*$ is given in \pref{eq:psifunction}.
\end{theorem}

Let us briefly explain the novelty of this paper and sketch our plan of attack. In part one \cite{AZ20}, we found the conditional distribution of the Hessian as a deformation of a matrix from the Gaussian Orthogonal Ensemble (GOE) and introduced bounds to control its dependency on the field location, the source of the main difficulty and difference between LRC and invariant models.   Here, we develop techniques to handle the extra condition on indices, which turns out to be quite subtle. In previous papers that dealt with complexity of saddles and local minima \cites{ABC13, ABA13}, fixing the index translated directly to a large deviation estimate on fixed eigenvalues of GOE. Here, the format of the conditional Hessian (and in particular its dependence on the field location) forces us to follow a different route. First, we take advantage of  a result of Lazutkin that relates the signature of a matrix to its principal minor. Then the most technical part of this paper is to deal with quadratic forms of Gaussian random variables, where the coefficients are given by the Stieltjes transform of GOE matrices with singularities on the real line. For a fixed index, these singularities lie outside the support of semicircle distribution; this allows us to derive various large deviation estimates. For the diverging index, since the singularity is in the support of semicircle distribution, we rely on an identity to transfer the expectation of the singular sum to the expectation of a single eigenvalue combined with a convexity argument.

The rest of the paper is organized as follows. In Section \ref{se:whole}, we fix the notation and provide the proofs of Theorems \ref{th:fixktotal} and \ref{th:divtotal}. We deduce a variety of large deviation estimates for Gaussian quadratic forms with singular coefficients in Section \ref{se:5}. This is the major technical input for analyzing the complexity of local minima in \pref{se:k=0} and the complexity of saddles with fixed index in \pref{se:kge1}. Finally, we derive the complexity of saddles with diverging index in \pref{se:div}. In the Appendix, we list some facts proved in \cite{AZ20} and needed in this paper for the reader's convenience and to make this version as self-contained as possible.

\textbf{Acknowledgments.}
Both authors would like to thank Yan Fyodorov for suggesting the study of fields with isotropic increments and providing several references, Julian Gold for providing the reference \cite{lazutkin1988signature}. A.A. thanks the hospitality of the International Institute of Physics where part of this work was developed. Q.Z. thanks Jiaoyang Huang for showing him the identity \pref{eq:elak}.

\section{Notation and proofs for Theorems \ref{th:fixktotal} and \ref{th:divtotal}}\label{se:whole}
We basically follow the notations used in \cite{AZ20}. Throughout, we regard a vector to be a column vector. We write e.g.~$C_{\mu,D}$  for a constant depending on $\mu$ and $D$ which may vary from line to line. For $N\in \nz$, let us denote $[N]=\{1,2,...,N\}$. Given $a,b\in \rz$, we write $a\vee b = \max\{a,b\}$ and $a\wedge b=\min\{a,b\}$. Also, $a_+=a\vee 0$ and $a_-=a\wedge 0$. For a vector $(y_1,...,y_N)^\sfT\in \rz^N$, we write $L(y_1^N)=\frac1N \sum_{i=1}^N \de_{y_i}$ for its empirical measure. Recall that an $N \times N$ matrix $M$ in the Gaussian Orthogonal Ensemble (GOE) is a symmetric matrix with centered Gaussian entries that satisfy
\begin{align}\label{eq:goemat}
  \ez(M_{ij})=0,\ \ \ez(M_{ij}^2) =\frac{1+\de_{ij}}{2N},\ \ i,j\in[N].
\end{align}
We will simply write $\GOE_N$ or $\GOE(N)$ for the matrix $M$. Denoting by $\lambda_1 \leq \dots \leq \lambda_N$ the eigenvalues of $M$, we write $L_N =L(\la_1^N)= \frac{1}{N} \sum_{k=1}^N \delta_{\lambda_k}$ for its empirical spectral measure. From time to time, we may also use $\la_k$ to denote the $k$th smallest eigenvalue of $\GOE_{N+1}$  or $\GOE_{N-1}$. This should be clear from context. For a closed set $F\subset \rz$, we denote by  $\px(F)$ the set of probability measures with support contained in $F$. We equip the space $\px(\rz)$ with the weak topology, which is compatible with the metric
\begin{align}\label{eq:measd}
d(\mu,\nu):=\sup\Big\{\Big|\int f \dd \mu-\int f \dd\nu \Big|: \|f\|_\8\le1, \|f\|_L\le 1\Big\},\quad \mu,\nu \in\px(\rz),
\end{align}
where $\|f\|_\8$ and $\|f\|_L$ denote the $L^\8$ norm and Lipschitz constant of $f$, respectively. Let $B(\nu,\de)$ denote the open ball in the space $\px(\rz)$ with center $\nu$ and radius $\de$ w.r.t.~to the metric $d$ given in \pref{eq:measd}. Similarly, we write $B_K(\nu,\de)=B_K(\nu,\de) \cap \px([-K,K])$ for some constant $K>0$. We denote by $\si_{{\rm sc}}$ the semicircle law scaled to have support $[-\sqrt2,\sqrt2]$.

We will frequently use the following facts which are consequences of large deviations. Using the large deviation principle (LDP) of empirical measures of GOE matrices \cite{BG97}, for any $\de>0$, there exists $c=c(\de)>0$ and $N_\de>0$ such that for all $N>N_\de$,
\begin{align}\label{eq:conineq}
\pz(L(\la_1^N)\notin B(\si_{\rm sc}, \de) )\le e^{-cN^2}.
\end{align}
In the same spirit, by \cite{MMS14}*{Theorem 5}, there exist $c_1, c_2>0$ such that for any $t>c_1\sqrt{\frac{\log(1+N)}N}$,
\begin{align}\label{eq:mms}
\pz\Big(\sup_{\|h\|_L\le 1} \Big|\frac1N\sum_{i=1}^N h(\la_i)- \int h(x)\si_{\rm sc}(\dd x)\Big|>t \Big)\le e^{- c_2 N^2 t^2}.
\end{align}
On the other hand, by the LDP of extreme eigenvalues of GOE matrices \cites{BDG01,ABC13}, for any fixed $k=1,2,3,...,$ the $k$th smallest eigenvalue $\la_k$ satisfies an LDP with speed $N$ and  a good rate function
\begin{align}\label{eq:jk1x}
J_k(x)=kJ_{1}(x)=
\begin{cases}
k\int_{x}^{-\sqrt{2}}\sqrt{z^2-2}\dd z, & x\le -\sqrt 2,\\
\8, & \text{otherwise},
\end{cases}
\end{align}
where the function $J_1(x)$ can be explicitly computed as
\begin{align}\label{eq:rf1ev}
J_1(x)=\begin{cases}
\frac12\log2 -\frac12 x\sqrt{x^2-2}-\log(-x+\sqrt{x^2-2}),& x\le -\sqrt2,\\
\8,& x>-\sqrt2.
\end{cases}
\end{align}
In particular, writing $\la_N^*=\max_{i\in [N]} |\la_i|$ for the operator norm of $\GOE_N$, by \cite{BDG01}*{Lemma 6.3}, there exists $N_0>0$ and $K_0>0$ such that for $K>K_0$ and $N>N_0$,
\begin{align}\label{eq:ladein}
\pz(\la_N^*>K)\le e^{-NK^2/9}.
\end{align}
This can also be seen directly from the LDP of $\la_1$, even though it was originally proved as a technical input for the LDP of $\la_1$. For a probability measure $\nu$ on $\rz$, let us define
\begin{align}\label{eq:psidef0}
  \Psi(\nu,x)=\int_\rz \log |x-t| \nu(\dd t), \qquad \Psi_*(x)=\Psi(\si_{\rm sc}, x).
\end{align}
By calculation,
\begin{align}\label{eq:phi*}
\Psi_*(x) &= \frac12 x^2-\frac12 -\frac12\log2-\int_{\sqrt2}^{|x|} \sqrt{y^2-2}\dd y \indi\{|x|\ge\sqrt2\} \notag \\
&=
\begin{cases}
\frac12 x^2-\frac12 -\frac12\log2,& |x|\le \sqrt2,\\
\frac12x^2 -\frac12-\log2 - \frac12 |x| \sqrt{x^2-2}+\log(|x|+\sqrt{x^2-2}) ,& |x|>\sqrt2.
\end{cases}
\end{align}
Note that $\Psi_*(x)-\frac{x^2}{2}\le -\frac12-\frac12\log2$. We define the major part of complexity functions $\psi_*(\rho,u,y)$ for $u,y\in \rz$ and $\rho>0$,
\begin{align}
 		&\psi_*(\rho,u,y)=\psi(\si_{\rm sc},\rho,u,y), \notag\\
    &\psi(\nu,\rho,u,y)= \Psi(\nu, y)-\frac{(u-\frac{\mu\rho^2}{2}+\frac{\mu D'(\rho^2) \rho^2}{D'(0) } )^2}{ 2 (D(\rho^2)-\frac{D'(\rho^2)^2 \rho^2}{D'(0)} )} -\frac{\mu^2\rho^2}{2D'(0)}+\log \rho-\frac{-2D''(0)}{-2D''(0)-\frac{[D'(\rho^2 )-D'(0)]^2}{{ D(\rho^2 )-\frac{D'(\rho^2)^2 \rho^2}{D'(0)}}}}
 		\notag \\
 		& \times\Big(y+\frac{1} {\sqrt{-4D''(0)}}\Big[ \mu + \frac{(u-\frac{\mu \rho^2}{2} +\frac{\mu D'(\rho^2) \rho^2}{D'(0)}) (D'(\rho^2) -D'(0) )}{D( \rho^2) -\frac{D'(\rho^2)^2 \rho^2}{D'(0) }}\Big]\Big)^2.
\label{eq:psifunction} 	
\end{align}
By \cite{AZ20}*{Lemma 5.7}, with other variables fixed we have
\begin{align}\label{eq:psi*lim}
  \lim_{\rho\to0+}\psi_*(\rho,u,y) = -\8, \ \ \lim_{|y|\to \8}\psi_*(\rho,u,y) = -\8.
\end{align}

Let $z$ be a standard Gaussian r.v.~and $\Phi$ the c.d.f.~of $z$. For $a\in\rz, b>0$, we have
\begin{align}\label{eq:gau2}
  \ez[(a+bz)\indi\{a+bz>0\}]=a\Phi(\frac{a}b) +\frac{b}{\sqrt{2\pi}} e^{-\frac{a^2}{2b^2}},
\end{align}
which is strictly increasing as $a$ increases, and
\begin{align}\label{eq:gau3}
  \ez[-(a+bz)\indi\{a+bz<0\}]= -a\Phi(-\frac{a}b) +\frac{b}{\sqrt{2\pi}} e^{-\frac{a^2}{2b^2}},
\end{align}
which is strictly decreasing as $a$ increases. These two functions are 1-Lipschitz convex in $a$. Also,
\begin{align}\label{eq:absgau}
\sqrt{\frac2\pi}b\le \ez|a+bz|=\frac{\sqrt2 b }{\sqrt{\pi}}e^{-\frac{a^2}{2b^2}}
 +a(2\Phi(\frac{a}{b})-1)\le \sqrt{\frac2\pi}b+|a|.
\end{align}
Unless specified otherwise, we always assume Assumptions I and II throughout. Moreover, Assumption IV is assumed after this section.

In the rest of this section, we prove the results for the total number of critical points with given indices. The starting point is the Kac--Rice formula (see e.g. \cite{AT07}*{Theorem 11.2.1}),
\begin{align} \label{eq:startingpoint}
  \ez \Crt_{N,k}(E,B_N)&=\int_{B_N} \ez[|\det \nabla^2 H_N(x)| \indi\{\frac1N H_N(x)\in E,i(\nabla^2 H_N(x))=k\}| \nabla H_N(x)=0]\notag\\
  & \quad \times p_{\nabla H_N(x)}(0)\dd x,
\end{align}
where $p_{\nabla H_N(x)}(t)$ is the p.d.f.~of ${\nabla H_N(x)}$ at $t$. By Lemma \ref{le:cov}, we have in distribution
\[
  \nabla^2 H_N(x) \stackrel{d}= \sqrt{-4D''(0)} \GOE_N- (\sqrt{\frac{-2D''(0)}N}Z-\mu)I_N
\]
where $Z$ is a standard Gaussian variable independent of $\GOE_N$.  When $E=\rz$, there is no restriction on the critical value. And thanks to the independence of $\nabla H_N$ and $\nabla^2 H_N$, the above representation of $\nabla^2 H_N(x)$ as a GOE matrix shifted by an independent scalar matrix is good enough for asymptotic analysis. Indeed, together with large deviation principles, it is sufficient to result in the high dimensional limits. This strategy is by far standard after it was developed in \cite{ABC13}.

\begin{proof}[Proof of Theorem \ref{th:fixktotal}]

The Kac--Rice formula in the current setting simplifies to
\[
\ez\Crt_{N,k}(\rz,B_N)=\int_{B_N} \ez[|\det \nabla^2 H_N(x)| \indi\{ i(\nabla^2 H_N(x))=k\}] p_{\nabla H_N(x)}(0)\dd x.
\]
A direct application of Lemma 3.3 in \cite{ABC13} gives us
\begin{align}\label{eq:ecnk8}
  \ez\Crt_{N,k}(\rz,B_N) & = \frac{C_N}{N+1} \ez_{\GOE(N+1)} e^{-\frac12(N+1)\la_{k+1}^2-\frac{\sqrt{N(N+1)}\mu \la_{k+1}}{\sqrt{-D''(0)}}},
\end{align}
where
\begin{align}\label{eq:stir1}
  C_N = \begin{cases}
  \frac{\sqrt{2}[-4D''(0)]^{N/2} \Ga(\frac{N+1}2)  (N+1) }{\sqrt{\pi}N^{N/2} e^{Nm^2 }|\mu|^N }  \pz(  z_N \in  |\mu| B_N/\sqrt{D'(0)} ), & \mu\neq 0,\\
 \frac{\sqrt{2}[-4D''(0)]^{N/2} \Ga(\frac{N+1}2)  (N+1)|B_N| }{\sqrt{\pi}N^{N/2} (2\pi)^{N/2}D'(0)^{N/2}} , & \mu=0,
  \end{cases}
\end{align}
whose asymptotics is given by the Stirling's formula and Assumption III. Let
$$
\phi(x)=-\frac12x^2-\frac{\mu x}{\sqrt{-D''(0)} }.
$$
Since $\phi$ is bounded above, by \pref{eq:jk1x} and Varadhan's Lemma,
\begin{align*}
\sup_{x\in \rz} \phi(x)-J_{k+1}(x)&\le \liminf_{N\to\8} \frac1N\log \ez_{\GOE(N+1)} e^{(N+1)\phi(\la_{k+1})} \\
&\le \limsup_{N\to\8} \frac1N\log \ez_{\GOE(N+1)} e^{(N+1)\phi(\la_{k+1})}\le \sup_{x\in \rz} \phi(x)-J_{k+1}(x).
\end{align*}
Let $I_{k}=\sup_{x\in \rz} [\phi(x)-J_{k+1}(x)]$. Then
\begin{align}\label{ikcase}
I_k=
\begin{cases}
\phi(-\sqrt2)=-1+\frac{\sqrt2 \mu}{\sqrt{-D''(0)}},& \mu\le \sqrt{-2D''(0)},\\
\phi(x_0)-J_1(x_0),& \mu>  \sqrt{-2D''(0)}, k=0,\\
\phi(x_k)-J_{k+1}(x_k), & \mu>  \sqrt{-2D''(0)}, k\ge1,
\end{cases}
\end{align}
where
\[
x_0=-\frac{\mu}{2\sqrt{-D''(0)}}-\frac{\sqrt{-D''(0)}}\mu,\ \ x_k= \frac{\mu- (k+1)\sqrt{\mu^2 -D''(0)k(k+2)} }{k(k+2)\sqrt{-D''(0)}}.
\]
A combination of these observations with \pref{le:repl} is enough to obtain the result.
\end{proof}

\begin{remark}\rm
  Clearly $x_0<-\sqrt2$. One can check that $x_k<-\sqrt2$ if and only if
\[
\Big[k(k+2)\frac{\mu^2}{-2D''(0)}-k(k+2)\Big]^2>0,
\]
which indeed holds for $k\ge1$.
Since $\phi(x)-J_{k+1}(x)< \phi(x)-J_{k}(x)$ for any $x<-\sqrt2$, we see that $I_k$ is strictly decreasing in $k$ and $I_k\ge 1$ when $\mu>\sqrt{-2 D''(0)}$. Observe also that $x_0$ and $x_k$ in \eqref{ikcase} have different expressions. The reason for different formats of local minima and saddles will become clear in Sections \ref{se:k=0} and \ref{se:kge1}.
\end{remark}

\begin{proof}[Proof of Theorem \ref{th:divtotal}]
The main ingredient of the proof is the LDP for the empirical spectral measure of GOE matrices \pref{eq:conineq}.
By \pref{le:repl}, it suffices to obtain the asymptotics of
$$\ez_{\GOE(N+1)} e^{(N+1)\phi(\la_{k+1})}.$$
As a consequence of \pref{eq:conineq}, we have for each $\eps>0$,
\begin{align}\label{eq:lakn2}
\pz(|\la_{k+1} -s_\ga|>\eps)\le e^{-c (N+1)^2}
\end{align}
for some $c=c(\ga,\eps)$. Therefore,
\begin{align*}
&\ez_{\GOE(N+1)} e^{(N+1)\phi(\la_{k+1})} \\
&\le \ez_{\GOE(N+1)} e^{(N+1)\phi(\la_{k+1})} \indi\{\la_{k+1}\in [s_\ga-\eps, s_\ga+\eps]\} + e^{\frac{(N+1)\mu^2}{-2D''(0)} -c(N+1)^2}\\
&\le e^{(N+1)\sup_{s_\ga-\eps\le x\le s_\ga +\eps}\phi(x)} + e^{\frac{(N+1)\mu^2}{-2D''(0)} -c(N+1)^2}.
\end{align*}
On the other hand,
\begin{align*}
\ez_{\GOE(N+1)} e^{(N+1)\phi(\la_{k+1})} &\ge \ez_{\GOE(N+1)} e^{(N+1)\phi(\la_{k+1})} \indi\{\la_{k+1}\in [s_\ga-\eps, s_\ga+\eps]\} \\
&\ge e^{(N+1)\inf_{s_\ga-\eps\le x\le s_\ga +\eps} \phi(x)}\pz(|\la_{k+1}-s_\ga|\le \eps).
\end{align*}
It follows that
\begin{align*}
\inf_{s_\ga-\eps\le x\le s_\ga +\eps} \phi(x) &\le \liminf_{N\to\8}  \frac1N\log \ez_{\GOE(N+1)} e^{(N+1)\phi(\la_{k+1})}\\
& \le \limsup_{N\to\8}  \frac1N\log \ez_{\GOE(N+1)} e^{(N+1)\phi(\la_{k+1})} \le \sup_{s_\ga-\eps\le x\le s_\ga +\eps} \phi(x).
\end{align*}
By the continuity of $\phi$, sending $\eps\to 0+$ we see that
\[
\lim_{N\to\8}  \frac1N\log \ez_{\GOE(N+1)} e^{(N+1)\phi(\la_{k+1})} = \phi(s_\ga),
\]
which ends the proof of the theorem.
\end{proof}

\section{Large deviations for singular Gaussian quadratic forms}\label{se:5}

\subsection{Reduction to large deviation estimates}
From now on, we consider the domain $B_N$ to be a shell $B_{N}(R_{1},R_{2})=\{ x\in \mathbb R^{N}: R_{1}< \frac{\| x\|}{\sqrt N} < R_{2} \}$, where $0\le R_1< R_2\le \8$. By \pref{le:cov}, $Y:= \frac{H_N(x)}N- \frac{D'(\frac{\|x\|^2}N)\sum_{i=1}^{N} x_i \partial_i H_N(x)}{N D'(0)}$ is independent of $\nabla H_N(x)$.
Thanks to \pref{eq:ayg} and \pref{eq:martin}, by conditioning and recalling the notation \pref{eq:crter12}, we can rewrite the Kac--Rice representation \pref{eq:startingpoint} as
\begin{align}
  &\ez\Crt_{N,k}(E,(R_1,R_2)) \notag\\
  &=\int_{B_N} \int_E \ez[|\det \nabla^2 H_N(x)|\indi\{i(\nabla^2 H_N(x))\}=k |Y=u] p_{\nabla H_N(x)}(0) \pz(Y\in \dd u) \dd x \notag\\
  &\stackrel{\rho=\frac{\|x\|}{\sqrt N}}{\scalebox{3}[1]{=}} S_{N-1} N^{(N-1)/2} \int_{R_1}^{R_2}\int_E \ez[|\det G|\indi\{i(G)=k\} ] \frac1{\sqrt{2\pi}\si_Y} e^{-\frac{(u -m_Y)^2}{2\si_Y^2}} \notag\\
  &\quad   \frac1{(2\pi)^{N/2} D'(0)^{N/2}} e^{-\frac{N \mu^2 \rho^2}{2D'(0)}} \rho^{N-1}   \dd u  \dd\rho.\label{eq:krerr}
  \end{align}
  Here  $S_{N-1} = \frac{2\pi^{N/2}}{\Gamma(N/2)}$ is the area of $N-1$ dimensional unit sphere, $G$ depends on $u$ implicitly.
  Note that
  \begin{align}\label{eq:snlim}
  \lim_{N\to\8} \frac1N \log (S_{N-1}N^{\frac{N-1}{2}})= \frac12\log(2\pi)+\frac12.
  \end{align}
Enumerate the eigenvalues of $G$ and $G_{**}$ as $( \lambda_j(G))_{1 \leq j \leq N}$ and $(\lambda_j(G_{**}))_{1 \leq j \leq N-1}$ in ascending manner. By the interlacement property, for all $j \in \{1,\dots, N-1\}$,
\begin{align}
\lambda_j (G) \leq \lambda_j(G_{**}) \leq \lambda_{j+1}(G)\,.
\label{eq:interlace}
\end{align}
For $k \geq 1$, we have
\begin{align*}
\{ i(G) = k \}  \subset \{i(G_{**}) = k-1 \}\cup\{  i(G_{**})=k\} .
\end{align*}
The following result of Lazutkin \cite{lazutkin1988signature} relates the signature of a matrix and that of a principal minor. Recall that the signature of a symmetric matrix is the number of positive eigenvalues minus that of negative eigenvalues. 
\begin{lemma}[{\cite[Equation 2]{lazutkin1988signature}}]\label{le:lazutkin}
 Let $S$ be a symmetric block matrix, and write its inverse $S^{-1}$ in block form with the same block structure:
	\begin{align*}
	S = \begin{pmatrix} A & B \\ B^\sfT & C \end{pmatrix}\,,\quad\quad S^{-1} = \begin{pmatrix} A' & B' \\ (B')^\sfT & C' \end{pmatrix}\,.
	\end{align*}
	Then, $\sgn(S) = \sgn(A) + \sgn(C')$, with $\sgn(M)$ denoting the signature of the matrix $M$.
\end{lemma}

Recall $z_1'=\si_1 z_1 -\si_2 z_2 + m_{1}$, $z_3'=(\si_2 z_2 + \frac{\rho\sqrt{\al \bt}z_3}{\sqrt N}-m_{2}) /\sqrt{-4D''(0)}$ as in \pref{eq:ayg}. To save space, let us write the Schur complement of $G_{**}$ as
\begin{align}\label{eq:lazutkin_variable}
  \zeta=\zeta(z_1', z_3')=z_1' - \xi^\sfT G_{**}^{-1} \xi,
\end{align}
where
\begin{equation}\label{eq:ztdef}
\zeta(r,s)= r- [-4D''(0)]^{-1/2} \Bigl\langle \xi, \Bigl(\sqrt{\frac{N-1}{N}}\GOE_{N-1}-s I_{N-1}\Bigr)^{-1} \xi\Bigr\rangle.
\end{equation}

\begin{remark} \label{rmk:lazutkin}
The variable $\zeta$ will be pivotal in the following analysis. Recalling the block matrix representation \eqref{eq:ayg} of $G$ and applying Lemma~\ref{le:lazutkin} in our setting, $G^{-1}$ can be expressed in block form, with $(G^{-1})_{11}$ given by $\frac1\zeta$
so that the index of $G$ is equal to the index of $G_{**}$ exactly when
$\zeta= z_1' - \langle \xi, (G_{**})^{-1} \xi\rangle > 0.$
In other words, for $k\ge 1$,
\begin{align}\label{eq:kdcmp}
\{i(G)=k\}=\{i(G_{**})=k, \zeta>0\}\cup\{i(G_{**})=k-1, \zeta <0\},
\end{align}
and for $k=0$,
\begin{align}\label{eq:0dcmp}
 \{i(G)=0\}= \{i(G_{**})=0,\zeta>0\}.
\end{align}
\end{remark}


Substituting \pref{eq:kdcmp} into the Kac--Rice formula \pref{eq:krerr}, we have
\begin{align}
	&\ez\Crt_{N,k}(E,(R_1,R_2) ) = S_{N-1} N^{(N-1)/2} \int_{R_1}^{R_2} \int_{E} \ez[|\det G| ( \indi\{i(G_{**})=k,\zeta>0\} \notag\\
	&\qquad  +  \indi\{i(G_{**})=k-1, \zeta<0 \} )]\frac1{\sqrt{2\pi}\si_Y} e^{-\frac{(u-m_Y)^2}{2\si_Y^2}}  \frac1{(2\pi)^{N/2} D'(0)^{N/2}} e^{-\frac{N \mu^2 \rho^2}{2D'(0)}} \rho^{N-1}   \dd u  \dd\rho.\label{eq:krk}
\end{align}
Since $\Crt_{N,k}(E,B_N) \le \Crt_{N}(E,B_N)$, by Lemmas \ref{le:exptt} and \ref{le:exptt2} we may assume $\bar E$ to be compact and $R_2<\8$. Let us define
\begin{align}
  I^k(E,(R_1,R_2)) &= \int_{R_1}^{R_2} \int_{E} \ez[|\det G|  \indi\{i(G_{**})=k,\zeta>0\}] \notag\\
  &\quad \frac{ e^{-\frac{(u-m_Y)^2}{2\si_Y^2}}}{\sqrt{2\pi}\si_Y} \frac{e^{-\frac{N \mu^2 \rho^2}{2D'(0)}}}{(2\pi)^{N/2} D'(0)^{N/2}}  \rho^{N-1}   \dd u  \dd\rho,\notag\\
  II^k(E,(R_1,R_2)) &=\int_{R_1}^{R_2} \int_{E} \ez[|\det G|  \indi\{i(G_{**})=k-1,\zeta<0\}] \notag\\
  &\quad \frac{e^{-\frac{(u-m_Y)^2}{2\si_Y^2}}}{\sqrt{2\pi}\si_Y} \frac{e^{-\frac{N \mu^2 \rho^2}{2D'(0)}} }{(2\pi)^{N/2} D'(0)^{N/2}}   \rho^{N-1} \dd u  \dd\rho. \label{eq:ikdef}
\end{align}
If $k=0$, we understand $\{i(G_{**})=k-1,\zeta<0\}=\emptyset$ and $II^0=0$. Using the Schur complement formula, we have
\begin{align}
  \det G= \det (G_{**})(z_1' - \xi^\sfT G_{**}^{-1} \xi ). \label{eq:schur}
\end{align}
Let us write
$$Q=Q(z_3')=\xi^\sfT G_{**}^{-1}\xi
$$
so that $\zeta=z_1'-Q(z_3')$. Recall that $\la_1\le \la_2\le \cdots \le \la_{N-1}$ denote the ordered eigenvalues of $\GOE_{N-1}$. Using the representation \pref{eq:goedc},
\begin{align*}
  Q(z_3')
  =[-4D''(0)]^{-1/2} \sum_{i=1}^{N-1}\frac{{\xi'}_i^2}{(\frac{N-1}{N})^{1/2}\tilde \la_i-z_3'} \stackrel{d}=\frac{\sqrt{-D''(0)}}{N}\sum_{i=1}^{N-1} \frac{Z_i^2}{(\frac{N-1}{N})^{1/2}\la_i-z_3'},
\end{align*}
where $\xi'=V\xi$ and $Z_i$'s are independent standard Gaussian random variables. Since $\xi'$ is independent of $\tilde\la_i$'s, the last $\stackrel{d}=$ follows by applying a random permutation (depending on $\tilde\la_i$'s) to  $\xi'$.
For convenience, we regard $\sqrt{\frac{-2D''(0)}{N}}Z_i=\xi_i'$.
By convention, $ \la_0=-\8$. A direct calculation yields that the conditional distribution of $z_1'$ given $z_3'=y$ is given by
\begin{align}\label{eq:z13con0}
z_1'| z_3'=y \sim N \Big (\bar \sfa, \frac{\sfb^2}{N}\Big),
\end{align}
where
\begin{align*}
 \bar \sfa&=  m_1-\frac{\si_2^2(\sqrt{-4D''(0)}y + m_2)}{\si_2^2+\frac{\al \bt \rho^2}{N}}\\
 &=\frac{-2D''(0)\al\rho^2( u-\frac{\mu\rho^2}{2}+\frac{\mu D'(\rho^2) \rho^2}{D'(0) } )}{(-2D''(0)-\bt^2) \sqrt{D(\rho^2)-\frac{D'(\rho^2)^2 \rho^2}{D'(0)} }}
   +\frac{\al\bt\rho^2 \mu }{-2D''(0)-\bt^2}\notag \\
   &\quad  -\frac{(-2D''(0)-\bt^2-\al\bt\rho^2)\sqrt{-4D''(0)} y}{-2D''(0)-\bt^2},\\
 \frac{\sfb^2}{N}& =  \si_1^2+\si_2^2-\frac{\si_2^4}{\si_2^2+\frac{\al\bt \rho^2}{N}} =\frac{-4D''(0)}{N} +\frac{2D''(0)\al^2\rho^4}{N(-2D''(0)-\bt^2)}.
\end{align*}
Using \pref{eq:schur}, \pref{eq:z13con0}, \pref{eq:gau2} and by conditioning, we have
\begin{align}
  &\ez[|\det G|  \indi\{i(G_{**})=k,\zeta>0\}]=\ez(|\det G_{**}| | z_1'-\xi^\sfT G_{**}^{-1} \xi |  \indi\{i(G_{**})=k , \zeta>0\}) \notag\\
&= [-4D''(0)]^{\frac{N-1}{2}} \ez[ |\det ((\frac{N-1}{N})^{1/2}\GOE_{N-1}-z_3' I_{N-1}) | \indi\{\la_k<(\frac{N}{N-1})^{1/2}z_3'<\la_{k+1} \} \notag\\
&\quad \ez(  | z_1'- Q(z_3') |\indi\{z_1'- Q(z_3') >0\}  | \la_1^{N-1}, z_3',\xi')] \notag\\
&=[-4D''(0)]^{\frac{N-1}{2}}\int_{\rz} \ez\Big[ |\det ((\frac{N-1}{N})^{1/2}\GOE_{N-1}- y I_{N-1}) | \indi\{\la_k<(\frac{N}{N-1})^{1/2} y<\la_{k+1} \}\notag\\
& \quad(\mathsf{a}_N\Phi(\frac{\sqrt N \sfa_N}\sfb) +\frac{\sfb}{\sqrt{2\pi N}} e^{-\frac{N\sfa_N^2}{2\sfb^2}})\Big]  \frac{\sqrt {-4N D''(0)} \exp\{-\frac{N(\sqrt{ -4D''(0)}y+m_2)^2}{2(-2D''(0)-\bt^2)}\}} {\sqrt{2\pi(-2D''(0)-\bt^2)} } \dd y, \label{eq:detgk}
\end{align}
where
\begin{align}\label{eq:abdef}
  \sfa_N & =\sfa_N(\rho,u,y)=\bar\sfa -Q(y).
\end{align}
For any $x\in \rz$ and $b>0$, using l'Hospital's rule,
\begin{align}\label{eq:phiblim}
\lim_{N\to\8}\frac1N \log \Big(x\Phi(\frac{\sqrt N x}b)+\frac{b}{\sqrt{2\pi N}} e^{-\frac{N x^2}{2b^2}}\Big) = -\frac{(x_-)^2}{2b^2},
\end{align}
where we recall the notation $x_-=x\wedge 0$. We want to replace $Q(y)$ with a deterministic quantity in the large $N$ limit.



Let $\bar\rz=\rz\cup\{-\8,+\8\}$ be the extended real numbers.
\begin{lemma}\label{le:sclaw}
  Let $h: \rz\to \bar\rz$ be a measurable function that is bounded and continuous on an open set $G\supset [-\sqrt2,\sqrt2]$. Then for any fixed $\ell=1,2,...$ and any deterministic sequence $\ta_N\to 1$, we have almost surely
  \[
  \lim_{N\to\8} \frac{1}{N-1}\sum_{i=\ell}^{N-1} h(\ta_N\la_i) = \int h(x) \si_{\rm sc}(\dd x),
  \]
  and
  \[
  \lim_{N\to\8} \frac{1}{N-1}\sum_{i=\ell}^{N-1} h(\ta_N\la_i)Z_i^2 = \int h(x) \si_{\rm sc}(\dd x).
  \]
\end{lemma}
\begin{proof}
  Let $[-\sqrt2,\sqrt2]\subset (-K,K')$ be a neighborhood such that $h(x)$ is bounded and continuous for $x\in[-K,K']$. Let
  \[
  \hat h (x)= \begin{cases}
            h(x), & \mbox{if } -K\le x\le K', \\
            h(K'), & \mbox{if } x>K', \\
            h(-K), & \mbox{otherwise}.
          \end{cases}
   \]
Since $\hat h(x)$ is a bounded continuous function on $\rz$, by the semicircle law, we have almost surely
  \[
  \lim_{N\to\8} \frac{1}{N-1}\sum_{i=\ell}^{N-1} \hat h(\la_i)=\lim_{N\to\8} \frac{1}{N-1}\sum_{i=1}^{N-1} \hat h(\la_i)=\int \hat h(x) \si_{\rm sc}(\dd x)=\int h(x) \si_{\rm sc}(\dd x).
  \]
  On the other hand, using the LDP of the extreme eigenvalues of GOE \pref{eq:jk1x}, for any $\eps>0$, by continuity we have for all $N$ large enough,
  \begin{align*}
    &\pz \Big(\Big|\frac{1}{N-1}\sum_{i=\ell}^{N-1} [\hat h(\la_i) - h(\ta_N\la_i)]\Big|>\eps\Big)\\
    &\le \pz (\la_\ell\le -\ta_N^{-1} K) +\pz(\la_{N-1}\ge \ta_N^{-1} K')\le  e^{-\ell N J_K}+e^{-N J_{K'}}
  \end{align*}
  where $J_K$ is an absolute constant depending on $K$. By the Borel--Cantelli lemma, we have almost surely
  \[
  \lim_{N\to\8} \frac{1}{N-1}\sum_{i=\ell}^{N-1} [\hat h(\la_i) - h(\ta_N\la_i)] = 0,
  \]
  and the first assertion follows. For the second one, using the Bernstein inequality for the subexponential variables $Z_i^2-1$ and by conditioning,
  \begin{align*}
    &\pz\Big(\Big|\frac{1}{N-1}\sum_{i=\ell}^{N-1} h(\ta_N \la_i) (Z_i^2-1)\Big|>\eps, -K\le \ta_N\la_\ell<\ta_N\la_N\le K' \Big)\\
    & \le 2\ez_{\GOE(N-1)}\Big[\indi\{ -K\le\ta_N \la_\ell<\ta_N\la_{N-1}\le K'\} \\
    &\qquad \times \exp\Big(-c\min\Big\{\frac{(N-1)^2\eps^2}{\sum_{i=\ell}^{N-1} h(\ta_N\la_i)^2}, \frac{(N-1)\eps}{\max_{\ell\le j\le N-1} |h(\ta_N\la_j)|}\Big\} \Big)\Big]\\
    &\le 2 \exp\Big(-c\min\Big\{\frac{(N-1)\eps^2}{ \|\hat h\|_\8^2}, \frac{(N-1)\eps}{ \|\hat h\|_\8}\Big\} \Big).
  \end{align*}
  It follows that for $N$ large enough
  \begin{align*}
    &\pz\Big(\Big|\frac{1}{N-1}\sum_{i=\ell}^{N-1} h(\ta_N\la_i)(Z_i^2-1)\Big|>\eps \Big)    \le \pz(\la_\ell<-\ta_N^{-1}K) +\pz(\la_{N-1}>\ta_N^{-1}K')\\
    &\qquad +\pz\Big(\Big|\frac{1}{N-1}\sum_{i=\ell}^{N-1} h(\ta_N\la_i) (Z_i^2-1)\Big|>\eps, -K\le \ta_N\la_\ell<\ta_N\la_{N-1}\le K' \Big)\\
    &\le 3 e^{-c N},
  \end{align*}
  for some constant $c=c(\eps, \ell, K, K',\|\hat h\|_\8)$. The proof is completed by using the Borel--Cantelli lemma again.
\end{proof}

Let us fix $y<-\sqrt2$. Then $h_y(x)=\frac1{x-y}$ is a bounded continuous function on $[\frac12(y-\sqrt2),2]$ whose interior contains $[-\sqrt2,\sqrt2]$. Recall that the Stieltjes transform of the semicircle law is defined as
\[
\sfm(z)= \int_{-\sqrt2}^{\sqrt2} \frac1{x-z}\si_{\rm sc}(\dd x)=-z-\sqrt{z^2-2}, \ \ z<-\sqrt2.
\]
Here we have chosen the branch of square root $\sqrt w>0$ for any $w>0$ so that $\sfm(z)\sim -\frac1z$ as $z\to -\8$. By \pref{le:sclaw} with $\ta_N=(\frac{N-1}{N})^{1/2}$, we have $\lim_{N\to\8} Q(y) =\sqrt{-D''(0)}\sfm(y)$ almost surely. Recalling $\bar \sfa$ as in \pref{eq:z13con0}, let us define
\begin{align}\label{eq:abdef1}
\sfa=\sfa(\rho,u,y)&=\bar \sfa-\sqrt{-D''(0)}\sfm(y)=  \frac1{-2D''(0)-\bt^2 } \Big( \frac{-2D''(0)\al\rho^2 (u-\frac{\mu\rho^2}{2}+\frac{\mu D'(\rho^2) \rho^2}{D'(0) } )}{\sqrt{D(\rho^2)-\frac{D'(\rho^2)^2 \rho^2}{D'(0)} } }  \notag\\
  &\quad  +\al\bt\rho^2 \mu - (-2D''(0)-\bt^2-\al\bt\rho^2)\sqrt{-4D''(0)} y \Big) -\sqrt{-D''(0)}  \sfm(y) .
\end{align}
Observe that the coefficient of $u$ in $\sfa$ is negative, thus $ \sfa<0$ for large $u>0$ and $\sfa>0$ for small $u<0$. However, we cannot replace $Q(y)$ with $\sqrt{-D''(0)}\sfm(y)$ directly due to the lack of good concentration. In fact, $\ez[Q(y)]$ is not even finite. In some sense, all the argument below is meant to resolve this issue. The correct quantity to replace $Q(y)$ turns out to involve the semicircle and chi-square distributions in a twisted way. Intuitively, $\sfa$ is approximately the mean of $z_1'$ conditioning on $z_3'$, and the sign of $\sfa$ determines whether we should use the left or the right tail of the conditional distribution. We also remark that the factor $(\frac{N-1}{N})^{1/2}$ barely affects our analysis below.

We consider a Lipschitz approximation of $f(x)=\frac1{x-y}$. Fix $\de\in(0, -\sqrt2-y]$ and define a Lipschitz function
\[
f_{y,\de}(x)=
\begin{cases}
\frac1{x-y}, & \mbox{if } x\ge y+\de, \\
\frac{x-y}{\de^2}, & \mbox{if } y\le x < y+\de,\\
0, & \mbox{otherwise}.
\end{cases}
\]
We have for $x>y$,
\begin{align}\label{eq:fyde}
 \frac{1}{x-y}\indi\{x\ge y+\de\} \le f_{y,\de}(x)\le \frac{1}{x-y},
\end{align}
and $\int f_{y,\de}(x)\si_{\rm sc}(\dd x) = \sfm(y)$.

Let $h:\rz\to\rz_+$ be a nonnegative bounded Lipschitz function which is strictly decreasing on $[-\sqrt2,\sqrt2]$. By \pref{le:sclaw}, for every $t>-\frac{1}{2\|h\|_\8}$, we have almost surely
\begin{align*}
  \lim_{N\to\8}& \frac{t}{N-1} \sum_{i=\ell}^{N-1} h(\la_i)-\frac1{N-1}\sum_{i=\ell}^{N-1}\frac12\log(1+2 h(\la_i)t) \\ &= t \int h(x)\si_{\rm sc}(\dd x) - \frac12\int \log[1+ 2 h(x)t] \si_{\rm sc} (\dd x).
\end{align*}
Let us define
\begin{align}\label{eq:ladef}
 \Lambda(t;h)&
    =\begin{cases}
  t \int h(x)\si_{\rm sc}(\dd x) - \frac12\int \log[1+ 2 h(x)t] \si_{\rm sc} (\dd x), & \mbox{if } t\ge -\frac{1}{2\|h\|_\8}, \\
  +\8, & \mbox{otherwise}.
                \end{cases}
\end{align}
If $0>t>-\frac{1}{2\|h\|_\8}$, there exists $c\in(0,1)$ such that
\[
\int \log[1+ 2 h(x)t] \si_{\rm sc} (\dd x)\ge \int \frac{2h(x)t}{1+2h(x)t} \si_{\rm sc}(\dd x) =-\int \frac{ch(x)}{\|h\|_\8- ch(x)}\si_{\rm sc}(\dd x)>-\8,
\]
 we see $\Lambda(t;h)<\8$ for all $t>-\frac{1}{2\|h\|_\8}$. Let
$$
\Lambda^*(s;h) =\sup_{t\in\rz}[ts-\Lambda(t;h)]
$$
be the Fenchel--Legendre transform of $\Lambda$. 
Fix $\ell\in \nz$. For $t\in\rz$, we define
  \[
  \Lambda_N(t;h,\ell)= t\int h(x)\si_{\rm sc} (\dd x) +\log \ez  e^{ -\frac{t}N\sum_{i=\ell}^{N-1}h(\la_i)Z_i^2}.
  \]
\begin{lemma}\label{le:gepre}
  Let $h:\rz\to\rz_+$ be a nonnegative bounded Lipschitz function which is strictly decreasing on $[-\sqrt2,\sqrt2]$. 
  Then
  \begin{enumerate}
    \item
    For any real number $t\neq -\frac1{2\|h\|_\8}$, $\Lambda(t;h)=\lim_{N\to\8}\frac1{N-1}\Lambda_N(Nt;h,\ell)$ as extended real number. 
    \item $\Lambda(t;h)$ is differentiable for $t>-\frac{1}{2\|h\|_\8}$.
    \item For any $\Lambda'(-\frac1{2\|h\|_\8}+;h)< s < \int h(x)\si_{\rm sc}(\dd x)$ there is a unique $\tau_s>-\frac1{2\|h\|_\8}$ such that $\Lambda'(\tau_s;h)=s$ and $\Lambda^*(s;h)=s\tau_s-\Lambda(\tau_s;h).$  Here $\Lambda'(-\frac1{2\|h\|_\8}+;h)$ is the right limit of $\Lambda'(t;h)$ at $t=-\frac1{2\|h\|_\8}$. Moreover, for any $t\neq s$, $\tau_s(t-s)+\Lambda^*(s;h)< \Lambda^*(t;h)$, i.e., any such $s$ is an exposed point. For $s_\8=\int h(x)\si_{\rm sc}(\dd x)$, $\Lambda^*(s_\8;h)=\lim_{t\to\8}s_\8 t-\Lambda(t;h)$.
  \end{enumerate}
\end{lemma}
\begin{proof}
  For item 1, using the moment generating function of chi-square distribution with conditioning, we find for $t>-\frac1{2\|h\|_{\8}}$,
  \[
  \frac1{N-1}\Lambda_N(Nt;h,\ell)= \frac{Nt}{N-1}\int h(x)\si_{\rm sc}(\dd x) +\frac1{N-1}\log \ez \exp\Big(-\frac12 \sum_{i=\ell}^{N-1} \log[1+ 2 h(\la_i) t] \Big).
  \]
  Note that $\log[1+h(x)t]$ is a bounded Lipschitz function in $x$. By \pref{eq:mms}, there exists a constant $c=c(h,t,\ell)>0$ such that for $N$ large,
  \begin{align*}\label{}
    \pz\Big(\Big|\frac{1}{N-1} \sum_{i=\ell}^{N-1} \log[1+2 h(\la_i)t]- \int \log[1+2h(x)t]\si_{\rm sc}(\dd x)\Big|>N^{-1/4}\Big) & \le e^{-c N^{3/2}}.
  \end{align*}
  Let
  $$\Omega_N=\{|\frac{1}{N-1} \sum_{i=\ell}^{N-1} \log[1+2h(\la_i)t]- \int \log[1+2h(x)t]\si_{\rm sc}(\dd x)|\le N^{-1/4}\}$$
  be the event with overwhelming probability. Since
  \begin{align*}
   & \limsup_{N\to\8}   \frac1{N-1}\log \ez \indi_{\Omega_N^c}\exp\Big(-\frac12 \sum_{i=\ell}^{N-1} \log[1+ 2h(\la_i) t] \Big)\\
   &\le \limsup_{N\to\8}   \frac1{N-1}\log \exp\Big[\frac{(N-1)\|h\|_\8 t_-}{1+ 2\|h\|_\8 t_-} -c N^{3/2}\Big]=-\8,
  \end{align*}
  it follows that
  \begin{align*}
    \lim_{N\to\8} & \frac1{N-1}\log \ez \exp\Big(-\frac12 \sum_{i=\ell}^{N-1} \log[1+ 2h(\la_i) t] \Big)\\
    &=\lim_{N\to\8}  \frac1{N-1}\log \Big( e^{-\frac{N-1}{2}[\int \log[1+2 h(x)t]\si_{\rm sc}(\dd x)+O(N^{-1/4})]}\pz (\Omega_N) \Big)\\
    &=-\frac{1}{2}  \int \log[1+2 h(x)t]\si_{\rm sc}(\dd x)
  \end{align*}
  and we have $\lim_{N\to\8}\frac1{N-1}\Lambda_N(Nt;h,\ell)=\Lambda(t;h)$ for $t>-\frac1{2\|h\|_{\8}}$. For $t< -\frac1{2\|h\|_{\8}}$, by continuity we may find $y_0\in \rz$ such that $1+2h(y_0)t< 0$. It follows that
  \[
  \ez\prod_{i=\ell}^{N-1} \frac1{[1+2h(\la_i)t]^{1/2}} = \8
  \]
  because the joint density of eigenvalues does not vanish in a neighborhood of $(y_0,y_0,...,y_0)\in \rz^{N-\ell}$.

  For items 2 and 3, by the dominate convergence theorem,
  \[
  \Lambda'(t;h)= \int h(x)\si_{\rm sc}(\dd x) - \int \frac{h(x)}{1+2 h(x)t} \si_{\rm sc} (\dd x)
  \]
  is strictly increasing and continuous in $t$ such that $\Lambda'(t;h)<0$ for $t<0$, $\Lambda'(0;h)=0$ and $\lim_{t\to+\8} \Lambda'(t;h)=\int h(x)\si_{\rm sc}(\dd x)$. It follows that for any $\Lambda'(-\frac1{2\|h\|_\8}+;h)< s < \int h(x)\si_{\rm sc}(\dd x)$ there is a unique $\tau_s>-\frac1{2\|h\|_\8}$ such that $\Lambda'(\tau_s;h)=s$. Therefore,
  \[
  \Lambda^*(s;h)=s\tau_s-\Lambda(\tau_s;h).
  \]
  If $s=\int h(x)\si_{\rm sc}(\dd x)$, the function $t\mapsto st-\Lambda(t;h)$ is strictly increasing and we have
  \[
  \Lambda^*(s;h)=\lim_{t\to\8} st-\Lambda(t;h).
  \]
  The assertion on exposed points is just \cite{DZ98}*{Lemma 2.3.9(b)}.
\end{proof}

Let us define
\begin{align*}
  \bar\Lambda(t;h)&=\limsup_{N\to\8} \frac1N\log\ez  \exp\Big[t\int h(x)\si_{\rm sc}(\dd x)- \frac{t}N\sum_{i=\ell}^{N-1}h(\la_i)Z_i^2\Big],\\
  \bar\Lambda^*(s;h) &=\sup_{t\in\rz} \{st-\bar\Lambda(t;h)\}.
\end{align*}
By \pref{le:gepre}, $\Lambda(t;h)=\bar\Lambda(t;h)$ for $t\neq -\frac1{2\|h\|_\8}$. Since for $t_1<t_2$,
  \begin{align*}
    &t_1\int h(x)\si_{\rm sc}(\dd x) +\limsup_{N\to\8}\frac1N\log\ez [ e^{-{t_1}\sum_{i=\ell}^{N-1}h (\la_i){Z_i^2} }]\\
    &\ge t_1\int h(x)\si_{\rm sc}(\dd x) + \limsup_{N\to\8} \frac1N\log\ez [ e^{-{t_2}\sum_{i=\ell}^{N-1} h (\la_i){Z_i^2}}],
  \end{align*}
  sending $t_2\to t_1+$, we deduce that for any $t$, $\bar\Lambda(t;h)  \ge \bar\Lambda(t+;h)=\Lambda(t+;h)$.
  It follows together with continuity that for any $s\in \rz$,
  \begin{align}\label{eq:la**}
    \bar\Lambda^*(s;h) =\sup_{t>-\frac1{2\|h\|_\8}} \{st-\Lambda(t;h)\}=\Lambda^*(s;h),
  \end{align}
  and thus
  \begin{align}\label{eq:lasy}
   \Lambda^*(s;h)=
   \begin{cases}
   +\8,& s> \int h(x) \si_{\rm sc}(\dd x),\\
    \lim_{t\to\8} st -\Lambda(t;h),& s= \int h(x) \si_{\rm sc}(\dd x), \\
      s\tau_s-\Lambda(\tau_s;h),& \Lambda'(-\frac1{2\|h\|_\8}+;h)<s< \int h(x) \si_{\rm sc}(\dd x),\\
      -\frac{s}{2\|h\|_\8}  - \Lambda(-\frac1{2\|h\|_\8}+;h), &s\le \Lambda'(-\frac1{2\|h\|_\8}+;h).
   \end{cases}
  \end{align}
  We will mainly consider the case $h=f_{y,\de}$ for $y<-\sqrt2$ and $\de\in(0,-\sqrt2-y]$. In particular,
  \[
  -\frac1{2\|f_{y,\de}\|_\8}=-\frac\de2,\  \ \Lambda(y+\de;f_{y,\de})<\8, \ \ -\frac1{2\|f_{y,-\sqrt2-y}\|_\8}=\frac{\sqrt2+y}2.
  \]
  Note also that $\Lambda^*(s;f_{y,\de})$ does not depend on $\de>0$ if $s\ge 0$. For simplicity, we will write $\Lambda^*(s;y)=\Lambda^*(s;f_{y,-\sqrt2-y})$.

\subsection{Left tail of $Q(y)$}

We first consider the right tail of $\sfa_N$, or the left tail of $Q(y)$. Sometimes the right tail of $Q(y)$ will also be obtained as a by-product.

\begin{proposition}\label{pr:ge}
  For any fixed $\ell\in\nz$, $y<-\sqrt2$, $0<\de\le -\sqrt2-y$, and $\Lambda(-\frac\de2+;f_{y,\de})\le s<s'\le\sfm(y)$,
  \[
  \lim_{N\to\8}\frac1N\log \pz\Big(\sfm(y) -s'\le \frac{1}{N} \sum_{i=\ell}^{N-1} f_{y,\de}(\la_i)Z_i^2 \le \sfm(y) -s\Big) = -\inf_{x\in[s,s']}\Lambda^*(x;f_{y,\de}).
  \]
  In particular, for $0\le s<s' \le \sfm(y)$,
  \[
  \lim_{N\to\8}\frac1N\log \pz\Big(\sfm(y)-s'\le \frac{1}{N} \sum_{i=\ell}^{N-1} f_{y,\de}(\la_i)Z_i^2 \le \sfm(y)-s\Big) = -\Lambda^*(s;f_{y,\de}).
  \]
\end{proposition}
\begin{proof}
  This is a consequence of the G\"artner--Ellis theorem applying to the random variables $\sfm(y)-\frac1N\sum_{i=\ell}^{N-1} f_{y,\de}(\la_i)(Z_i^2)$. Indeed, with \pref{le:gepre} and \pref{eq:la**} we deduce from \cite{DZ98}*{Theorem 4.5.3} the upper bound for compact sets
  \[
  \limsup_{N\to\8} \frac1N\log\pz\Big(\sfm(y)-\frac1N\sum_{i=\ell}^{N-1} f_{y,\de}(\la_i)Z_i^2 \in [s,s']\Big)\le- \inf_{x\in [s,s'] }\Lambda^*(x;f_{y,\de}).
  \]
  This also implies that this sequence of random variables is exponentially tight as shown in the proof of \cite{DZ98}*{Theorem 2.3.6}. By \pref{le:gepre}, every $x\in(s,s')$ is an exposed point of $\Lambda^*$ with an exposing hyperplane $\tau_x$ for which $\Lambda(\tau_x;f_{y,\de})=\lim_{N\to\8}\frac1N \Lambda_N(N\tau_x;f_{y,\de})$ and $\Lambda(\ga_x \tau_x;f_{y,\de})<\8$ for some $\ga_x>1$,
  it follows from \cite{DZ98}*{Theorem 4.5.20} that
  \[
  \liminf_{N\to\8} \frac1N\log\pz\Big(\sfm(y)-\frac1N\sum_{i=\ell}^{N-1} f_{y,\de}(\la_i)Z_i^2 \in [s,s']\Big)\ge - \inf_{x\in (s,s')}\Lambda^*(x;f_{y,\de}).
  \]
  Since $\Lambda'(0;f_{y,\de})=0$ and $\Lambda'(t;f_{y,\de})$ is strictly increasing, we see that $\Lambda^*(0;f_{y,\de})=0$. But $\Lambda^*(\cdot;f_{y,\de})$ is nonnegative and convex, we have
  \[
  \inf_{x\in (s,s')}\Lambda^*(x;f_{y,\de}) = \inf_{x\in [s,s']}\Lambda^*(x;f_{y,\de})
  \]
  and the first assertion follows. If $s\ge0$, then $\inf_{x\in [s,s']}\Lambda^*(x;f_{y,\de})=\Lambda^*(s;f_{y,\de})$.
\end{proof}

Fix $k,\ell\in\nz$. We will need to deal with the joint probability involving $\la_j$ and $\frac1N\sum_{i=\ell}^{N-1}\frac{Z_i^2}{\la_i-y}$. 
Let us write
 \begin{align*}
 \tilde\Lambda_N(t_1,t_2;\de)&= \log \ez\exp\Big\{t_1\la_j -\frac{t_2}N \sum_{i=\ell}^{N-1} f_{y,\de}(\la_i){Z_i^2}+t_2\sfm(y) \Big\},\ \ t_1,t_2\in\rz,\\
 \tilde \Lambda(t_1,t_2;\de)&= \limsup_{N\to\8}\frac1N \tilde \Lambda_N(N t_1, Nt_2;\de),\\
 \tilde \Lambda^*(s_1,s_2;\de) &=\sup_{(t_1,t_2)\in\rz^2} [s_1t_2+s_2t_2-\tilde \Lambda(t_1,t_2;\de)].
 \end{align*}

\begin{lemma}\label{le:exppt}
For $s_1,s_2\in\rz$,
\[
 \tilde \Lambda^*(s_1,s_2;\de)=j J_1(s_1)+  \Lambda^*(s_2;f_{y,\de}).
 \]
  If $y'<-\sqrt2, \Lambda'(-\frac\de2+;f_{y,\de})< s<\sfm(y)$, then $(y',s)$ is an exposed point of $\tilde \Lambda^*$.
\end{lemma}
\begin{proof}
 Using the tail probability estimate of $\la_1$ \pref{eq:ladein},
 \[
 \pz(\la_j\le -K)\le\pz(\la_1\le -K)\le e^{-\frac19(N-1)K^2}
  \]
  for $K$ large enough. If $t_1\le0$, we may take in particular $K>10|t_1|$ so that
 \begin{align*}
   \ez e^{t_1 N \la_j}\indi\{\la_j\le -K\} &\le \sum_{i=1}^\8 e^{-t_1N K(i+1)}\pz(\la_j \in[-(i+1)K, -iK])\\
   &\le \sum_{i=1}^\8 e^{-t_1N K(i+1)-\frac19(N-1)i^2K^2}\le 2.
 \end{align*}
 It follows that
 \begin{align*}
   \limsup_{N\to\8}\frac1N\log \ez e^{t_1 N \la_j}\le \limsup_{N\to\8}  \frac1N\log [e^{-t_1NK}\pz(\la_j>-K)+\ez e^{t_1 N \la_j}\indi\{\la_j\le -K\} ]\le |t_1|K.
 \end{align*}
 Similar estimates hold for $t_1>0$ as well. This verifies Varadhan's lemma \cite{DZ98}*{Theorem 4.3.1} and thus
 \[
 \lim_{N\to\8} \frac1N\log \ez e^{N t_1 \la_j} = \sup_{x\in \rz}[ t_1x-j J_1(x)].
 \]
 For $t_1\in\rz, t_2>-\frac\de2$, by conditioning and \pref{eq:mms} as before,
 \begin{align*}
   \tilde\Lambda_N(t_1,t_2;\de)&=t_2\sfm(y)+\log\ez \exp \Big\{ t_1\la_j - \frac1{2N}\sum_{i=\ell}^{N-1} \log[1+2t_2f_{y,\de}(\la_i)]\Big\}\\
   &= t_2\sfm(y)+\log\Big[\ez \indi_{\Omega_N} \exp \Big\{ t_1\la_j-\frac12 \int \log[1+2t_2f_{y,\de}(x)]\si_{\rm sc}(\dd x) +O(N^{-1/4})\Big\}\\
   &\quad +\ez\indi_{\Omega_N^c}\exp \Big\{ t_1\la_j - \frac1{2N}\sum_{i=\ell}^{N-1} \log[1+2t_2f_{y,\de}(\la_i)]\Big\}\Big]
 \end{align*}
 where $\Omega_N=\{|\frac{1}{N} \sum_{i=k}^{N-1} \log[1+2f_{y,\de}(\la_i)t_2]- \int \log[1+2f_{y,\de}(x)t_2]\si_{\rm sc}(\dd x)|\le N^{-1/4}\}$ and $\pz(\Omega_N^c)\le e^{-c(t_2,\de,y,k) N^{3/2}}$. From here we deduce that for $t_1\in\rz, t_2>-\frac\de2$,
 \begin{align*}
   \tilde \Lambda(t_1,t_2;\de) &=  t_2\sfm(y) -\frac12 \int \log[1+2t_2f_{y,\de}(x)]\si_{\rm sc}(\dd x) +\sup_{x\in \rz}[ t_1x-j J_1(x)],
 \end{align*}
 and $\tilde \Lambda(t_1,t_2;\de) =+\8$ for $t_2<-\frac\de2$. By the duality lemma \cite{DZ98}*{Lemma 4.5.8},
 \[
 \tilde \Lambda^*(s_1,s_2;\de)=j J_1(s_1)+  \Lambda^*(s_2;f_{y,\de}).
 \]
 Here is a subtlety that is worth more explanation following the same idea as for \pref{eq:lasy} in one variable. For general $0<\de\le -\sqrt2-y$, one has $\tilde \Lambda(t_1,-\frac\de2;\de) \ge \tilde \Lambda(t_1,-\frac\de2+;\de)$ and the two sides may not be equal. Hence $\tilde \Lambda^*(s_1,s_2;\de)$ does not depend on $\Lambda(t_1,-\frac\de2;\de)$ and is instead determined by $\Lambda(t_2;f_{y,\de})$ for $t_2>-\frac\de2$. Since $\Lambda(t_2;f_{y,\de})$ is right continuous at $t_2=-\frac\de2$, $\Lambda^*(s_2;f_{y,\de})$ yields the correct formula for $\tilde \Lambda^*(s_1,s_2;\de)$.

 By \pref{le:gepre} and \cite{DZ98}*{Lemma 2.3.9}, if $y'<-\sqrt2, \Lambda'(-\frac{\de}{2}+;f_{y,\de})< s<\sfm(y)$, then $(y',s)$ is an exposed points for $\tilde \Lambda^*$.
\end{proof}

\begin{proposition}\label{pr:geup}
  Let $y<-\sqrt2$ and $0<\de\le -\sqrt2-y$. Fix $k,\ell\in\nz$. Then for any closed set $B$ and $x<-\sqrt2$,
  \begin{align*}
    \limsup_{N\to\8}\frac1N\log \pz \Big(\la_k\le x,\ \sfm(y)-\frac1N\sum_{i=\ell}^{N-1} f_{y,\de}(\la_i)Z_i^2 \in B\Big) & \le -k J_1(x) -\inf_{s\in B}\Lambda^*(s;f_{y,\de}).
  \end{align*}
  Moreover, if $0\le s<\sfm(y)$,
  \[
  \liminf_{N\to\8}\frac1N\log \pz \Big(\la_k\le x,\ \sfm(y)-\frac1N\sum_{i=\ell}^{N-1} f_{y,\de}(\la_i)Z_i^2 \ge s\Big)  \ge -k J_1(x) -\Lambda^*(s;f_{y,\de}).
  \]
\end{proposition}
\begin{proof}
  1. For $t>-\frac\de2$, let $\Omega_N=\{|\frac{1}{N} \sum_{i=\ell}^{N-1} \log[1+2f_{y,\de}(\la_i)t]- \int \log[1+2f_{y,\de}(x)t]\si_{\rm sc}(\dd x)|\le N^{-1/4}\}.$
  By \pref{eq:mms}, there exists a constant $c=c(y,\de,\ell,t)>0$ such that for $N$ large, we have $\pz(\Omega_N^c)\le e^{-c N^{3/2}}$.  Since $Z_i$'s are independent of $\la_i$'s, by an argument similar to that of \pref{le:gepre} with conditioning, we have for all $w$ and any $-\frac\de2<t\le0$,
  \begin{align*}
    &\pz \Big(\la_k\le x,\ \sfm(y) -\frac1N\sum_{i=\ell}^{N-1} f_{y,\de}(\la_i)Z_i^2\le w\Big) \\
    &\le e^{-Nt[w-\sfm(y)]} \ez\Big(\indi\{\la_k\le x\} \prod_{i=\ell}^{N-1}(1+2tf_{y,\de}(\la_i))^{-1/2}\Big)\\
    &\le  e^{-Nt[w-\sfm(y)]}\Big( e^{-\frac{N}2 (\int \log[1+2f_{y,\de}(x)t]\si_{\rm sc}(\dd x) -N^{-1/4})}\pz(\la_k\le x, \Omega_N) + e^{\frac{N t}{\de+2 t}} \pz(\Omega_N^c)\Big)\\
    &\le e^{-N[t(w-\sfm(y))+\frac12\int \log[1+2f_{y,\de}(x)t]\si_{\rm sc}(\dd x) -N^{-1/4}]}\pz(\la_k\le x) +e^{\frac{N t}{\de+2 t}-Nt[w-\sfm(y)] -cN^{3/2}}.
  \end{align*}
  By \pref{le:gepre}, $\Lambda(t;f_{y,\de})\ge 0$ for all $t$, if $w\le0$, then
  \[
  \Lambda^*(w;f_{y,\de})=\sup_{-\frac\de2 <t\le0} \Big\{wt-\sfm(y)t+\frac12\int \log[1+2f_{y,\de}(x)t]\si_{\rm sc}(\dd x)\Big\}.
  \]
  Using the LDP of $\la_k$, sending $N\to\8$ and then optimizing in $t$, for $w\le0$ we find
  \begin{align*}
    \limsup_{N\to\8}\frac1N\log\pz \Big(\la_k\le x, \sfm(y)-\frac1N\sum_{i=\ell}^{N-1} f_{y,\de}(\la_i)Z_i^2\le w\Big) &\le -k J_1(x) -\sup_{-\frac\de2 <t\le0} \{wt-\Lambda(t;f_{y,\de})\}\\
    &= -k J_1(x) -\Lambda^*(w;f_{y,\de}).
  \end{align*}
  Since $\Lambda^*(w;f_{y,\de})\ge0$ is convex and $\Lambda^*(0;f_{y,\de})=0$, $\Lambda^*(w;f_{y,\de})=\inf_{s\le w} \Lambda^*(s;f_{y,\de})$. We have proved the lemma for $B=(-\8,w]$ and $w\le0$. The case $B=[w,\8)$ and $w\ge0$ can be proved in a similar way with $t\ge0$.
  The argument for general closed $B$ is standard.

2. By \pref{le:exppt}, every $(x',s')$ is an exposed point of $\tilde \Lambda^*$ for $x'<x, \Lambda(-\frac\de2+;f_{y,\de})< s<s'<\sfm(y)$. Using the abstract G\"artner--Ellis theorem \cite{DZ98}*{Theorem 4.5.20} again,
\begin{align*}
  &\liminf_{N\to\8}   \frac1N\log \pz \Big(\la_k\le x,\ \sfm(y)-\frac1N\sum_{i=\ell}^{N-1} f_{y,\de}(\la_i)Z_i^2 \ge s\Big)\\
    &\ge -\inf_{x'<x, s<s'<\sfm(y)} [k J_1(x') +\Lambda^*(s';f_{y,\de})]
\end{align*}
and the second claim follows from the monotonicity and continuity of $J_1$ and $\Lambda^*$.
\end{proof}

\subsection{Right tail of $Q(y)$}
The right tail of $Q(y)$ is more involved compared to the left tail as $\bar\Lambda$ is not steep and $s$ is not an exposed point for $s\le \Lambda'(-\frac\de2+;f_{y,\de})$.  We also need some sort of uniform estimates on the (upper bound of) tail probability. Recall that we write $\Lambda^*(s;y)=\Lambda^*(s;f_{y,-\sqrt2-y})$.

\begin{proposition}\label{pr:ge2}
  Fix $k,\ell\in\nz$, $y<-\sqrt2$ and $\de>0$ such that $y+\de\le -\sqrt2$. For any $y_1<-\sqrt2$ and $s<0$, we have
  \[
  \limsup_{N\to\8}\frac1N\log \sup_{y'\le y}\pz\Big(\la_{\ell}\le y_1, \frac1N\sum_{i=k}^{N-1}f_{y', \de} (\la_i){Z_i^2} -\sfm(y)\ge -s \Big)\le -\ell J_1(y_1) -\Lambda^*(s;y).
  \]
  Moreover, if $y\le y_2<-\sqrt2$ and $ \ell J_1(y_1)<k J_1(y_2)$, we have for $s<\sfm(y)$,
  \[
  \liminf_{N\to\8}\frac1N\log \pz\Big(\la_{\ell}\le y_1, \la_k>y_2, \frac1N\sum_{i=k}^{N-1}\frac{Z_i^2}{\la_i-y} -\sfm(y)\ge -s \Big)\ge -\ell J_1(y_1) -\Lambda^*(s_-;y).
  \]
\end{proposition}

\begin{proof}
1. To get the uniform estimate, let
\begin{align*}
  \bar f_{y',\de, y}(x)=\begin{cases}
  \frac{1}{x-y}, & \mbox{if } x\ge -\sqrt2, \\
  \frac{1}{\sqrt2+y'+\de}\Big(\frac1\de+\frac1{\sqrt2+y} \Big) (x+\sqrt2)+\frac1{-\sqrt2-y}, & \mbox{if } y'+\de\le x<-\sqrt2, \\
  f_{y',\de}(x), & x<y'+\de,
                \end{cases}
\end{align*}
and $g(x)=\bar f_{y',\de, y}(x)-f_{y,-\sqrt2-y}(x), x\in\rz$. Morally speaking, $\bar f_{y',\de, y}$ is a bridge to go from the function $x\mapsto \frac1{x-y'}$ to $x\mapsto \frac1{x-y}$ by connecting the points $(y'+\de,\frac1\de)$ and $(-\sqrt2,\frac1{-\sqrt2-y})$ with straight line.
Clearly, $\bar f_{y',\de, y}(x)\ge f_{y',\de}(x)$ for all $x\in\rz$ and $g$ is supported in $[y',-\sqrt2]$ such that $0\le g(x)\le \frac1\de$.
 For any $\eps>0$,
 \begin{align*}
   &\pz\Big(\la_\ell\le y_1, \frac1N\sum_{i=k}^{N-1}f_{y', \de} (\la_i){Z_i^2} \ge \sfm(y) -s \Big) \\
   &\le \pz\Big(\la_\ell\le y_1, \frac1N\sum_{i=k}^{N-1}[  f_{y, -\sqrt2-y} (\la_i)+g(\la_i) ]{Z_i^2} \ge \sfm(y) -s \Big)\\
   &\le \pz\Big(\la_\ell\le y_1, \frac1N\sum_{i=k}^{N-1} f_{y, -\sqrt2-y} (\la_i){Z_i^2} \ge \sfm(y) -s -\eps \Big)+ \pz\Big(\la_\ell\le y_1, \frac1N\sum_{i=k}^{N-1}g(\la_i) {Z_i^2} \ge \eps\Big).
 \end{align*}
 By \pref{pr:geup} and continuity,
 \begin{align*}
   \limsup_{\eps\to0+}\limsup_{N\to\8} \frac1N\log \pz\Big(\la_\ell\le y_1, \frac1N\sum_{i=k}^{N-1} f_{y, -\sqrt2-y} (\la_i){Z_i^2} \ge \sfm(y) -s -\eps \Big) \le -\ell J_1(y_1) -\Lambda^*(s;y).
 \end{align*}
 For $\eps>0$, we take $\ga=\ga(\eps,\de)>0$ such that
 $$
 \frac12\Big(\frac{\de\eps}{\ga}-\log \frac{\de\eps}{\ga}-1\Big)> 10\Lambda^*(s;y)+10.
 $$
 Since $\lim_{N\to\8}\frac{[\ga N]}{N}=\ga$, by LDP of empirical measures of GOE as in \pref{eq:lakn2}, we may find $c=c(\ga)>0$ such that $\pz(\la_{[\ga N]}<\frac{s_\ga-\sqrt2}2)\le e^{-c N^2}$ where $s_\ga$ is the quantile with $\si_{\rm sc}([-\sqrt2,s_\ga])=\ga$. It follows that uniformly for $y'\le y$,
 \begin{align*}
   &\pz\Big(\la_\ell\le y_1, \frac1N\sum_{i=k}^{N-1}g(\la_i) {Z_i^2} \ge \eps\Big)\le \pz\Big(\la_\ell\le y_1, \frac{1}{N}\sum_{i=k}^{N-1} Z_i^2 \indi\{\la_i\le-\sqrt2 \}\ge \de\eps, \la_{[\ga N]}<\frac{s_\ga-\sqrt2}2 \Big) \\
   &\quad  +\pz\Big(\la_\ell\le y_1,\frac{1}{N}\sum_{i=k}^{N-1} Z_i^2 \indi\{\la_i\le-\sqrt2 \} \ge \de\eps ,\la_{[\ga N]}\ge \frac{s_\ga-\sqrt2}2\Big)\\
   &\le \pz\Big(\la_{[\ga N]}<\frac{s_\ga-\sqrt2}2 \Big)+\pz \Big(\la_\ell\le y_1, \frac{1}{N}\sum_{i=k}^{[\ga N]} Z_i^2 \ge \de\eps \Big).
 \end{align*}
 By Cramer's Theorem and the LDP for $\la_\ell$,
 \begin{align*}
   \limsup_{N\to\8}\frac1N\log \pz \Big(\la_\ell\le y_1, \frac{1}{N}\sum_{i=k}^{[\ga N]} Z_i^2 \ge \de\eps \Big)\le -\ell J_1(y_1) -\frac12\Big(\frac{\de\eps}{\ga}-\log \frac{\de\eps}{\ga}-1\Big).
 \end{align*}
 From here the upper bound follows.

 2. For the lower bound, we take advantage of an idea in \cite{BGR97}. If $\sfm(y)> s>s_0:=\Lambda'(\frac{\sqrt2+y}{2}+;f_{y,-\sqrt2-y})$, by \pref{le:exppt},  $(y_1,s)$ and $(y_2,s)$ are exposed points of $\tilde\Lambda^*$. Using \cite{DZ98}*{Theorem 4.5.20},
 \begin{align*}
   \liminf_{N\to\8}\frac1N\log   \pz\Big(\la_{\ell}\le y_1, \frac1N\sum_{i=k}^{N-1}f_{y,-\sqrt2-y}(\la_i){Z_i^2} -\sfm(y)\ge -s \Big) \ge -\ell J_1(y_1)- \inf_{s_0<s'<s} \Lambda^*(s';y),\\
   \limsup_{N\to\8}\frac1N\log   \pz\Big(\la_{k}\le y_2, \frac1N\sum_{i=k}^{N-1}f_{y,-\sqrt2-y}(\la_i){Z_i^2} -\sfm(y)\ge -s \Big) \le -k J_1(y_2)- \inf_{s'\le s}  \Lambda^*(s';y).
 \end{align*}
 Since $ \ell J_1(y_1)<k J_1(y_2)$ and observing that
 \begin{align*}
  & \pz\Big(\la_{\ell}\le y_1, \la_k>y_2,\frac1N\sum_{i=k}^{N-1} \frac{Z_i^2}{\la_i-y} -\sfm(y)\ge -s \Big)\\
  &\ge \pz\Big(\la_{\ell}\le y_1, \frac1N\sum_{i=k}^{N-1}f_{y,-\sqrt2-y}(\la_i){Z_i^2} -\sfm(y)\ge -s \Big)\\
   &\quad -\pz\Big(\la_k\le y_2, \frac1N\sum_{i=k}^{N-1}f_{y,-\sqrt2-y}(\la_i){Z_i^2} -\sfm(y)\ge -s\Big),
 \end{align*}
 the assertion follows.
 Now assume $s\le s_0$. For $\eps>0$, by independence,
 \begin{align*}
   &\pz\Big(\la_{\ell}\le y_1, \la_k>y_2, \frac1N\sum_{i=k}^{N-1}\frac{Z_i^2}{\la_i-y} -\sfm(y)\ge -s \Big)\\
   &\ge \pz\Big(\frac1N\sum_{i=k+1}^{N-1} \frac{Z_i^2}{\la_i-y} -\sfm(y)\ge -s_0-\eps, \ \frac{Z_k^2}{N(\la_k-y)} \ge -s+s_0+\eps, \ y_2<\la_k\le -\sqrt2+\eps,\la_{\ell}\le y_1 \Big)\\
   &\ge \pz\Big(\frac1N\sum_{i=k+1}^{N-1} \frac{Z_i^2}{\la_i-y} -\sfm(y)\ge -s_0-\eps, \  y_2 <\la_k\le -\sqrt2+\eps,\la_{\ell}\le y_1\Big) \\
   &\quad \pz\Big(\frac{Z_k^2}{N(-\sqrt2+\eps -y)}\ge -s+s_0+\eps  \Big)\\
   &\ge \Big[\pz\Big(\la_{\ell}\le y_1, \frac1N\sum_{i=k+1}^{N-1} f_{y,-\sqrt2-y}(\la_i){Z_i^2} -\sfm(y)\ge -s_0-\eps\Big) -\pz\Big( \la_k > -\sqrt2+\eps\Big)\\
   &\ \ -\pz\Big(\la_k\le y_2, \frac1N\sum_{i=k+1}^{N-1} f_{y,-\sqrt2-y}(\la_i){Z_i^2} -\sfm(y)\ge -s_0-\eps\Big) \Big]  \pz\Big(\frac{Z_k^2}{N(-\sqrt2+\eps -y)}\ge -s+s_0+\eps  \Big).
 \end{align*}
 Since $s_0+\eps$ is an exposed point, using the LDP for $\la_k$ and the G\"artner--Ellis theorem again,
 \begin{align*}
   \liminf_{N\to\8}\frac1N\log \pz\Big(\frac1N\sum_{i=k+1}^{N-1} \frac{Z_i^2}{\la_i-y} -\sfm(y)\ge -s_0-\eps, \  y_2<\la_k\le -\sqrt2+\eps, \la_{\ell}\le y_1\Big)\\
   \ge -\ell J_1(y_1)-\Lambda^*(s_0+\eps;y).
 \end{align*}
 Using the tail probability asymptotic estimates for standard Gaussian,
 \[
 \liminf_{N\to\8}\frac1N\log\pz\Big(\frac{Z_k^2}{N(-\sqrt2+\eps -y)}\ge -s+s_0+\eps  \Big)\ge -\frac12(-s+s_0+\eps )(-\sqrt2+\eps -y).
 \]
 It follows from continuity that
 \begin{align*}
 &\liminf_{\eps\to0+} \liminf_{N\to\8}\frac1N\log\pz\Big(\la_{\ell}\le y_1, \la_k>y_2, \frac1N\sum_{i=k}^{N-1}\frac{Z_i^2}{\la_i-y} -\sfm(y)\ge -s \Big)\\
 &\ge \frac12(-s+s_0 )(\sqrt2 +y) -\ell J_1(y_1) -\Lambda^*(s_0;y)=-\ell J_1(y_1)-\frac{(\sqrt2+y)s}{2} +\Lambda\Big(\frac{\sqrt2+y}{2};y\Big) \\
 & =-\ell J_1(y_1)-\Lambda^*(s;y).
 \end{align*}
 Here we used the fact that $\Lambda^*(s;y)=\frac{(\sqrt2+y)s}{2} -\Lambda(\frac{\sqrt2+y}{2};y)$ for $s\le s_0$.
\end{proof}

Let us compute $\Lambda^*(s;f_{y,\de})$ for $y<-\sqrt2$ and $\de\in(0,-\sqrt2-y]$. 
In this case, we have $\|f_{y,\de}\|_\8 = \frac1\de$ and for $t>-\frac\de2$, $\Lambda'(-\frac\de2+;f_{y,\de})< s<\sfm(y)$,
\begin{align*}
 \Lambda(t;f_{y,\de}) &= \sfm(y)t - \frac12[\frac12(y-2t)^2-\frac12 y^2 -\frac12y\sqrt{y^2 -2}+\frac12(y-2t)\sqrt{(y-2t)^2-2} \\ &\qquad +\log(2t-y+\sqrt{(y-2t)^2-2})-\log(-y+\sqrt{y^2-2})], \\
 \Lambda'(t;f_{y,\de})  &=\sfm(y)-\sfm(y-2t) =-2t+\sqrt{(y-2t)^2-2}-\sqrt{y^2-2},\\
 \tau_s &= \frac{s^2+2s\sqrt{y^2-2}}{-4(s+y+\sqrt{y^2-2})} =\frac{s^2+2s\sqrt{y^2-2}}{4(\sfm(y)-s)},\\
 \Lambda^*(s;f_{y,\de})& =-\frac18s^2-\frac{s}{2\sfm(y)} -\frac12\log\Big(1-\frac{s}{\sfm(y)}\Big) =\Big(\frac1{4\sfm(y)^2}-\frac18\Big) s^2 +\sum_{n=3}^\8\frac{s^n}{2n\sfm(y)^n}.
\end{align*}
From \pref{pr:ge2}, it should be transparent that the case $\de=-\sqrt2-y$ will play a special role in the following. 
 In particular, $\Lambda'(\frac{\sqrt2+y}{2}+;y)=-\sqrt2-y-\sqrt{y^2-2}$.

\begin{remark}\label{re:ylim}
  All these calculations are made for fixed $y<-\sqrt2$ and $0<\de\le -\sqrt2-y$. Later on, we may need the boundary behavior as $\de\downarrow 0$ or as $y\uparrow -\sqrt2$ which forces $\de\downarrow 0$. We observe that the functions are unchanged whenever the variables remain fixed and well-defined as $\de\downarrow 0$. By continuity and \pref{eq:lasy},
  \begin{align*}
    \lim_{y\to-\sqrt2-} \Lambda'(t;f_{y,\de})& = -2t+\sqrt{(\sqrt2+2t)^2-2}, \quad   \lim_{y\to-\sqrt2-,t\to0+} \Lambda'(t;f_{y,\de})=0,\\
    \lim_{y\to-\sqrt2-}\tau_s & =\frac{s^2}{4(\sqrt2-s)}, \quad    \lim_{y\to-\sqrt2-,s\to0+}\tau_s =0,\\
    \lim_{y\to-\sqrt2-}\Lambda^*(s;f_{y,\de}) & =
    \begin{cases}
    -\frac18s^2 -\frac{s}{2\sqrt2} -\frac12\log\Big (1-\frac{s}{\sqrt2}\Big), &s\ge 0,\\
     0,& \text{otherwise}.
    \end{cases}
  \end{align*}
  The bottom line is that those functions extend continuously to the boundary.
\end{remark}

In the following, the above large deviation estimates are typically applied for \emph{finitely} many points so that the convergence in $N$  is uniform. To give a rough idea on how to use these estimates, if $\sfa<0$, then the right tail of $\sfa_N$ (and the left tail of $Q(y)$) will potentially contribute in \pref{eq:detgk}. It turns out that this tail probability only affects the complexity function for local minima. For saddles, we will see that it suffices to simply drop $z_1'-Q(z_3')>0$ in \pref{eq:detgk}. If $\sfa\ge0$, then the left tail of $\sfa_N$ (and the right tail of $Q(y)$) will potentially contribute when $z_1'-Q(z_3')<0$.

\section{Local minima}\label{se:k=0}
For local minima, we only need to consider $I^0$ defined in \pref{eq:ikdef}. 
For $\rho,u$ and $y<-\sqrt2$ fixed, let us define
  \[
  \ix_-(\rho,u,y;x)=\Lambda^*(x;f_{y,\de})+\frac1{2\sfb^2} [(\sfa+\sqrt{-D''(0)}x)_-]^2,\quad x\ge0,
\]
which is independent of $\de\in(0,-\sqrt2-y]$.

\begin{lemma}\label{le:ixmin}
  Let
  \begin{align}\label{eq:ixmin}
    \hat x&= \hat x (\rho,u,y)= \frac{C+B\sfm(y)-\sqrt{(C-B\sfm(y))^2+4B}}{2B},\notag\\
  B&=B(\rho)=\frac{-D''(0)\al^2\rho^4}{(-2D''(0)-\bt^2)\sfb^2}, \quad C=C(\rho,u,y)=\frac1{\sfm(y)} -\frac{2\sqrt{-D''(0)}\sfa}{\sfb^2}.
\end{align}
  Then
  $$\inf_{x\in[0,\sfm(y)]}\ix_-(\rho,u,y;x) = \begin{cases}
    \ix_-(\rho,u,y;0)=0, & \mbox{if } \sfa\ge0, \\
    \ix_-(\rho,u,y;\hat x)>0, & \mbox{otherwise}.
  \end{cases}
  $$
  Moreover, if $\sfa<0$,  the minimum is attained at $0<\hat x<\sfm(y)$ with $\sfa+\sqrt{-D''(0)}\hat x < 0$.
\end{lemma}
\begin{proof}
  For any fixed $\rho,u$ and $y$, $\ix_-(\rho,u,y;\cdot)$ is a continuous function of $x\in[0,\sfm(y)]$ and attains minimum. Since $\Lambda^*(x;f_{y,\de})$ is strictly increasing for $x\ge0$, if $\sfa+\sqrt{-D''(0)}x\ge 0$, then
  $$\ix_-(\rho,u,y;x)\ge \Lambda^*(-\sfa [-D''(0)]^{-1/2} \vee 0;f_{y,\de}).$$
  In particular, if $\sfa\ge0$,
  $$\inf_{x\in[0,\sfm(y)]} \ix_-(\rho,u,y;x) =\Lambda^*(0;f_{y,\de})=0.$$
  Suppose $\sfa<0$. 
  A calculation yields
  \begin{align}
    &\partial_x \ix_-(\rho,u,y;x)\notag\\
    &=-\frac14x+\frac14(y-\sqrt{y^2-2}) +\frac1{2(-y-\sqrt{y^2-2}-x)} +\frac{\sqrt{-D''(0)}(\sfa+\sqrt{-D''(0)}x)}{\sfb^2} \label{eq:ixder1}\\
    &=\frac{-D''(0)\al^2\rho^4 x}{2(-2D''(0)-\bt^2)\sfb^2}+\frac1{2(-y-\sqrt{y^2-2}-x)} +\frac14(y-\sqrt{y^2-2})+\frac{\sqrt{-D''(0)}\sfa}{\sfb^2}. \label{eq:ixder2}
  \end{align}
  Since in \pref{eq:ixder1},
  \[
  -\frac14x+\frac14(y-\sqrt{y^2-2}) +\frac1{2(-y-\sqrt{y^2-2}-x)}>0
  \]
  for $x>0$ and tends to $+\8$ as $x\to \sfm(y)$, we see that the minimum of $\ix_-(\rho,u,y;\cdot)$ can only be attained for $\sfa+\sqrt{-D''(0)}x< 0$.
  From \pref{eq:ixder2}, it is clear that there is a unique solution to $\partial_x \ix_-(\rho,u,y;x)=0$ and this solution is in  $(0,\sfm(y))$.
  It follows that $\ix_-(\rho,u,y;\cdot)$ has a unique minimizer $\hat x$, which is in $(0,\sfm(y))$. We claim that $\hat x$ is given by \pref{eq:ixmin}. Indeed, since $C\sfm(y)>1$, we see
  $$\frac{C+B\sfm(y)-\sqrt{(C-B\sfm(y))^2+4B}}{2B}>0.$$
  If $B\sfm(y)\le C$ we have
  $$
  \frac{C+B\sfm(y)-\sqrt{(C-B\sfm(y))^2+4B}}{2B}<\frac{C+ B\sfm(y)-|C-B\sfm(y)|}{2B}=\sfm(y).
  $$
  If $B\sfm(y)>C$, since the other possible solution to $\partial_x \ix_-(\rho,u,y;x)=0$ is
  \[
  \frac{C+B\sfm(y)+\sqrt{(C-B\sfm(y))^2+4B}}{2B}> \frac{C+ B\sfm(y)+|B\sfm(y)-C|}{2B}=\sfm(y),
  \]
  the expression given in \pref{eq:ixmin} has to be the minimizer in $(0,\sfm(y))$.
\end{proof}
Thanks to this lemma, we know $\hat x>0$ when $\sfa<0$. Let us define
\begin{align}\label{eq:ixdef}
  \ix^-(\rho,u,y)= 
  \begin{cases}
    \ix_-(\rho,u,y;0)=0, & \mbox{if } \sfa\ge0, \\
    \ix_-(\rho,u,y;\hat x), & \mbox{otherwise}.
  \end{cases}
\end{align}
If $\sfa=0$, $C\sfm(y)=1$ and thus $\hat x=0$. Clearly, $\ix^-$ is a continuous function of $(\rho,u,y)$.
Note that $\ix^-$ extends continuously to $y=-\sqrt2$. Since $\Lambda^*(x;f_{y,\de})$ is independent of $\de$ for $x\ge0$, we will write $\Lambda^*(x)=\Lambda^*(x;f_{y,\de})$ in what follows when there is no ambiguity.

%
%

\subsection{Upper bound}
We need a covering argument for the quantity $I^0$ defined in \pref{eq:ikdef} in the following.  To this end, let $K>0$ be a large constant to be determined later. For $\rho'\in[R_1, R_2], u'\in \bar E, y'\in[-K,-\sqrt2],\de>0$ and $\de_1>0$, let
\begin{align*}
&I_N(\rho',u',y';\de;\de_1)= 
\int_{\rho'-\de}^{\rho'+\de} \int_{u'-\de}^{u'+\de} \int_{y'-\de}^{y'+\de} \frac{e^{-\frac{(u-m_Y)^2}{2\si_Y^2}-\frac{N \mu^2 \rho^2}{2D'(0)} -\frac{N(\sqrt{ -4D''(0)}y+m_2)^2}{2(-2D''(0)-\bt^2)} }} {\si_Y\sqrt{-2D''(0)-\bt^2}}
 \\
  & \quad \ez\Big[ e^{(N-1)\Psi(L((\frac{N-1}{N})^{1/2}\la_1^{N-1}),y)} \indi\{L((\frac{N-1}{N})^{1/2}\la_1^{N-1})\in B_K(\si_{\rm sc}, \de_1),(\frac{N-1}{N})^{1/2}\la_1> y\} \\
  & \quad \Big(\mathsf{a}_N\Phi\Big(\frac{\sqrt N\sfa_N}\sfb\Big) +\frac{\sfb}{\sqrt{2\pi N}} e^{-\frac{N\sfa_N^2}{2\sfb^2}}\Big)\Big]  \rho^{N-1}  \dd y \dd u  \dd\rho.
\end{align*}
Here and in what follows we always replace the integration limits with the boundary of $(R_1, R_2)\times \bar E \times(-K,-\sqrt2)$ if they exceed the latter set; e.g.~we always replace $y'+\de$ with $-\sqrt2$ if $y'+\de>-\sqrt2$. From \pref{le:albtd} and \pref{eq:abdef1}, we know
  \begin{align}\label{eq:supa}
  \sup\{|\sfa(\rho,u,y)|: (\rho,u,y)\in [R_1,R_2]\times \bar E\times [-K,-\sqrt2]\}<\8.
  \end{align}
\begin{lemma}\label{le:covlet}
  For any $K>0$ and $0<\de<1$, we have
  \begin{align*}
    &\limsup_{\de_1\to0+}\limsup_{N\to\8} \frac1N\log I_N(\rho',u',y';\de;\de_1)\le 
    \sup_{\rho'-\de < \rho< \rho'+\de, u'-\de< u< u'+\de,\atop y'-\de< y< y'+\de}\psi_*(\rho,u,y)\\
    &\quad -\inf_{0<x<\sfm((y'+\de)\wedge -\sqrt2)} \Lambda^*((x-3\sqrt\de )_+;f_{y'-2\de,\de})+ \frac1{2\sfb_m^2} [(\sfa_m+\sqrt{-D''(0)}x)_-]^2.
  \end{align*}
  Here $\sfb_m=\sup_{\rho\in (\rho'-\de, \rho'+\de)\cap [R_1,R_2]} \sfb(\rho)$ and
  \[
  \sfa_m= \sup_{(\rho,u,y)\in (\rho'-\de, \rho'+\de)\times (u'-\de, u'+\de)\times( y'-\de, y'+\de)\cap [R_1,R_2]\times \bar E\times [-K,-\sqrt2]} \sfa(\rho,u,y).
  \]
  Moreover,
  \begin{align*}
    \liminf_{\de\to0+}  \inf_{0<x<\sfm((y'+\de)\wedge -\sqrt2)} \Lambda^*((x-3\sqrt\de )_+;f_{y'-2\de,\de})+ \frac1{2\sfb_m^2} [(\sfa_m+\sqrt{-D''(0)}x)_-]^2& =\ix^-(\rho',u',y').
  \end{align*}
\end{lemma}
\begin{proof}
  Let $y\in (y'-\de, y'+\de)$ and $\eps_1>0$ be a small number that will be sent to zero later.  Since $\sfm(y)$ is convex and increasing, we have
\begin{align*}
  \sfm((y'+\de) \wedge -\sqrt2)-\sfm(y'-2\de)&\le \sfm(-\sqrt2)-\sfm(-\sqrt2-3\de)\le 3\sqrt\de.
\end{align*}
By \pref{pr:ge}, for any $\eps_2>0$ there exists $N(\eps_2,y',\de)>0$ such that
\begin{align*}
  &\frac1N\log\pz \Big(\frac{1}{N}\sum_{i=1}^{N-1} f_{y'-2\de,\de}(\la_i)Z_i^2\in[0, \sfm((y'+\de)\wedge -\sqrt2)-j\eps_1 ] \Big)\\
  &\le -\Lambda^*[(j\eps_1 -[\sfm((y'+\de) \wedge -\sqrt2)-\sfm(y'-2\de)])_+]+\eps_2\le -\Lambda^*[(j\eps_1 -3\sqrt\de)_+]+\eps_2
\end{align*}
for all $N>N(\eps_2,y',\de)$ and all $j=0,1,...,[\frac{\sfm(y)}{\eps_1}]$.
Note that the functions
  \[
  x\mapsto x\Phi(\frac{\sqrt N x}\sfb) +\frac{\sfb}{\sqrt{2\pi N}} e^{-\frac{Nx^2}{2\sfb^2}},\quad \sfb \mapsto x\Phi(\frac{\sqrt N x}\sfb) +\frac{\sfb}{\sqrt{2\pi N}} e^{-\frac{Nx^2}{2\sfb^2}}, \quad y\mapsto \sfm(y)
  \]
  are increasing.
  For $(\frac{N-1}{N})^{1/2}\la_i>y$ and $N$ large enough, $\frac1{(\frac{N-1}{N})^{1/2}\la_i-y}>\frac{1}{\la_i-y'+2\de}\ge f_{y'-2\de,\de}(\la_i)$, we deduce
  \begin{align*}
    &\ez\Big[\indi\{(\frac{N-1}{N})^{1/2}\la_1> y\}  \Big(\mathsf{a}_N\Phi(\frac{\sqrt N\sfa_N}\sfb) +\frac{\sfb}{\sqrt{2\pi N}} e^{-\frac{N\sfa_N^2}{2\sfb^2}}\Big)\Big] \\
    &\le \sum_{j=0}^{[ \frac{\sfm(y)}{\eps_1}] } \ez \Big[\Big(\mathsf{a}_N\Phi(\frac{\sqrt N\sfa_N}\sfb) +\frac{\sfb}{\sqrt{2\pi N}} e^{-\frac{N{\sfa_N}^2}{2\sfb^2}}\Big)\\
    &\quad \indi\Big\{(\frac{N-1}{N})^{1/2}\la_1> y, \frac{1}{N}\sum_{i=1}^{N-1} \frac{Z_i^2}{(\frac{N-1}{N})^{1/2}\la_i-y}\in[\sfm(y)-(j\eps_1+\eps_1), \sfm(y)-j\eps_1 ]\Big\}\Big]\\
    &\quad +\ez \Big[(\mathsf{a}_N\Phi(\frac{\sqrt N\sfa_N}\sfb) +\frac{\sfb}{\sqrt{2\pi N}} e^{-\frac{N{\sfa_N}^2}{2\sfb^2}}) \indi\Big\{\frac{1}{N}\sum_{i=1}^{N-1} \frac{Z_i^2}{(\frac{N-1}{N})^{1/2}\la_i-y}> \sfm(y)\Big\}\Big]\\
    &\le \Big\lceil \frac{\sfm(y)}{\eps_1}\Big\rceil \max_{0\le j\eps_1\le \sfm(y)} \Big\{ \pz \Big(\frac{1}{N}\sum_{i=1}^{N-1} f_{y'-2\de,\de}(\la_i)Z_i^2\in[0, \sfm((y'+\de)\wedge -\sqrt2)-j\eps_1 ] \Big) \\
    & \quad \Big[\frac{\sfb}{\sqrt{2\pi N}} e^{-\frac{N(\sfa+\sqrt{-D''(0)}(j\eps_1+\eps_1))^2 }{2\sfb^2}} +[\sfa+\sqrt{-D''(0)}(j\eps_1+\eps_1)]\Phi\Big(\frac{\sqrt N (\sfa + \sqrt{-D''(0)}(j\eps_1+\eps_1))}{\sfb}\Big) \Big]\Big\}\\
    &\quad +\Big[\sfa \Phi\Big(\frac{\sqrt N \sfa }{\sfb}\Big)+\frac{\sfb}{\sqrt{2\pi N}} e^{-\frac{N\sfa^2 }{2\sfb^2}} \Big]\\
    &\le \Big[\sfa_m \Phi\Big(\frac{\sqrt N \sfa_m }{\sfb_m}\Big) +\frac{\sfb_m}{\sqrt{2\pi N}} e^{-\frac{N\sfa_m^2 }{2\sfb^2}} \Big] + \Big\lceil \frac{\sfm(-\sqrt2)}{\eps_1}\Big\rceil \max_{0\le x\le \sfm((y'+\de)\wedge-\sqrt2)} \Big\{ e^{-N(\Lambda^*((x -3\sqrt\de)_+)-\eps_2)}  \\
    & \quad\Big[\frac{\sfb_m}{\sqrt{2\pi N}} e^{-\frac{N(\sfa_m+\sqrt{-D''(0)}(x+\eps_1))^2 }{2\sfb_m^2}} +[\sfa_m+\sqrt{-D''(0)}(x+\eps_1)]\Phi\Big(\frac{\sqrt N (\sfa_m + \sqrt{-D''(0)}(x+\eps_1))}{\sfb_m}\Big) \Big]\Big\}.
  \end{align*}
  Note that
  \begin{align*}
  &\Phi\Big(\frac{\sqrt N (\sfa_m + \sqrt{-D''(0)}(x+\eps_1))}{\sfb_m}\Big) \le \indi\{\sfa_m +\sqrt{-D''(0)}(x+\eps_1)\ge 0\}\\
  &\quad + \frac{\sfb_m}{\sqrt{2\pi N} |\sfa_m + \sqrt{-D''(0)}(x+\eps_1)|} e^{-\frac{N(\sfa_m+\sqrt{-D''(0)}(x+\eps_1))^2 }{2\sfb_m^2}}\indi\{\sfa_m + \sqrt{-D''(0)}(x+\eps_1)<0\}.
  \end{align*}
  Similar estimate holds for $\Phi\Big(\frac{\sqrt N \sfa_m}{\sfb_m}\Big)$. Combining altogether, we deduce that
  \begin{align*}
    &\limsup_{\eps_2\to0+}\limsup_{N\to\8}\frac1N \log \sup_{\rho'-\de < \rho< \rho'+\de, u'-\de< u< u'+\de,\atop y'-\de< y< y'+\de}\\
    &\qquad\ez \Big[\indi\{(\frac{N-1}{N})^{1/2}\la_1> y\}  \Big(\mathsf{a}_N\Phi(\frac{\sqrt N\sfa_N}\sfb) +\frac{\sfb}{\sqrt{2\pi N}} e^{-\frac{N\sfa_N^2}{2\sfb^2}}\Big)\Big]\\
    &\le -\inf_{-\eps_1<x<\sfm((y'+\de)\wedge -\sqrt2)} \Lambda^*((x-3\sqrt\de )_+; f_{y'-2\de,\de})+ \frac1{2\sfb_m^2} [(\sfa_m+\sqrt{-D''(0)}(x+\eps_1))_-]^2.
  \end{align*}
  Observe that the right-hand side is continuous in $\rho,u,y$ and $\eps_1, \de$, and  it tends to $-\ix^-(\rho',u',y')$ as $\eps_1\to0+, \de\to0+$ by \pref{le:ixmin}. We also note that the only possible singular function in the integrand of $I_N$ is $\frac{\rho^2}{\sqrt{D(\rho^2)-\frac{D'(\rho^2)^2\rho^2}{D'(0)}} }$, which continuously extends to $\rho\in [0,R_2)$ and attains its maximum on compact intervals. Since $\Psi(\nu,y)$ is upper semi-continuous for $(\nu,y)\in B_K(\si_{\rm sc}, \de)\times [-K,K]$, we conclude that
  \begin{align*}
    &\limsup_{\de_1\to0+,\atop \eps_1\to0+} \limsup_{\eps_2\to0+} \limsup_{N\to\8} \frac1N\log I_N(\rho',u',y';\de;\de_1) \\
    &\le \limsup_{\de_1\to0+}\sup_{\rho'-\de < \rho< \rho'+\de, u'-\de< u< u'+\de,\atop y'-\de< y< y'+\de,\nu\in B_K(\si_{\rm sc},\de_1)}\psi(\nu,\rho,u,y)\\
    &\quad -\liminf_{\eps_1\to0+} \inf_{-\eps_1<x<\sfm((y'+\de)\wedge -\sqrt2)} \Lambda^*((x-3\sqrt\de )_+)+ \frac1{2\sfb_m^2} [(\sfa_m+\sqrt{-D''(0)}(x+\eps_1))_-]^2\\
    &\le  \sup_{\rho'-\de < \rho< \rho'+\de, u'-\de< u< u'+\de,\atop y'-\de< y< y'+\de}\psi_*(\rho,u,y)\\
    &\quad - \inf_{0<x<\sfm((y'+\de)\wedge -\sqrt2)} \Lambda^*((x-3\sqrt\de )_+; f_{y'-2\de,\de})+ \frac1{2\sfb_m^2} [(\sfa_m+\sqrt{-D''(0)}x)_-]^2.
  \end{align*}
  The second assertion follows from continuity and \pref{le:ixmin} by sending $\de\to0+$.
\end{proof}

Recall $I^0(E,(R_1,R_2))$ as in \pref{eq:ikdef}.
\begin{proposition}\label{pr:ubk0}
 Assume $\bar E$ is compact and $R_2<\8$. Then
 \begin{align*}
   \limsup_{N\to\8}\frac1N\log I^0(E,(R_1,R_2))\le \frac12\log[-4D''(0)] -\frac12\log(2\pi)-\frac12\log D'(0)\\
   +\sup_{(\rho,u,y)\in F} [\psi_*(\rho,u,y)- \ix^-(\rho,u,y)],
 \end{align*}
 where $F=\{(\rho,u,y): \rho\in(R_1,  R_2), u\in\bar E, y\le -\sqrt2\}$.
\end{proposition}
\begin{proof}
Thanks to \pref{le:exptt}, we may assume $R_1>0$. By \pref{eq:psi*lim} and the definition of $\sfa$ as in \pref{eq:abdef1}, for any fixed $\rho>R_1$ and $u\in \bar E$, we have $\lim_{y\to -\8} \sfa =\8$ and
  $$\lim_{y\to-\8} \psi_*(\rho,u,y)- \ix^-(\rho,u,y)=\lim_{y\to-\8} \psi_*(\rho,u,y)=-\8.$$
  For any fixed $u,y$, since $\ix^-(\rho,u,y)\ge0$, by \pref{eq:psi*lim}, we have
\[
\lim_{\rho\to0+} \psi_*(\rho,u,y)- \ix^-(\rho,u,y) =-\8.
\]
We may choose $K$ large enough so that $\sup_{\rho\in[R_1,R_2],u\in\bar E, y\le-\sqrt2} [\psi_*(\rho,u,y)- \ix^-(\rho,u,y)]$ is attained at a point $(\rho^0,u^0,y^0)\in [R_1,R_2]\times \bar E\times [-K,-\sqrt2]$ with $\rho^0>0$. Let $\eps>0$. By \pref{eq:conineq}, we deduce that for any $y\ge -\sqrt2+\eps$,
\[
\pz(\la_1>y)\le e^{-c N^2}
\]
for some $c=c(\eps)>0$. Since
\[
\ez[|\det G|  \indi\{i(G_{**})=0,\zeta>0\}]\le \ez[|\det G|\indi\{(\frac{N-1}{N})^{1/2}\la_1>z_3'\}],
\]
using Lemmas \ref{le:goecpt}, \ref{le:goecpt2}, \ref{le:exptti11}, and \ref{le:exptti12}, by choosing $K>0$ large enough we have for $\de_1>0$,
  \begin{align*}
  &\limsup_{N\to\8} I^0(E,(R_1,R_2))  \\
  &\ \le \frac12\log[-4D''(0)]+ \max\{ \limsup_{N\to\8}I_N^0(\de_1;[-K,-\sqrt2-\eps)), \ \limsup_{N\to\8}I_N^0(\de_1;[-\sqrt2-\eps,-\sqrt2+\eps])\}
  \end{align*}
  where
\begin{align*}
 &I_N^0(\de_1;G) :=\int_{R_1}^{R_2} \int_{E} \int_{G} \ez\Big[ |\det ((\frac{N-1}{N})^{1/2}\GOE_{N-1}-y I_{N-1}) |   \\
  &\ \  \indi\{L((\frac{N-1}{N})^{1/2}\la_1^{N-1})\in B_K(\si_{\rm sc}, \de_1), (\frac{N-1}{N})^{1/2}\la_1> y\}\Big(\mathsf{a}_N\Phi(\frac{\sqrt N\sfa_N}\sfb) +\frac{\sfb}{\sqrt{2\pi N}} e^{-\frac{N\sfa_N^2}{2\sfb^2}}\Big)\Big]  \\
  &\ \ \frac{ e^{-\frac{(u-m_Y)^2}{2\si_Y^2}}}{\sqrt{2\pi}\si_Y} \frac{e^{-\frac{N \mu^2 \rho^2}{2D'(0)}}}{(2\pi)^{N/2} D'(0)^{N/2}} \frac{\sqrt {-4N D''(0)} \exp\{-\frac{N(\sqrt{ -4D''(0)}y+m_2)^2}{2(-2D''(0)-\bt^2)}\}} {\sqrt{2\pi(-2D''(0)-\bt^2)} } \rho^{N-1}  \dd y \dd u  \dd\rho.
\end{align*}
Here we remark that $K$ may depend on $\rho$ through $m_2$; we can choose $K<\8$ since we have assumed $R_1>0$.
Note that $I_N^0(\de_1;[-K,-\sqrt2-\eps))\le I_N^0(\de_1;[-K,-\sqrt2])$.
Consider a cover of the compact set $[R_1,R_2]\times \bar E\times [-K,-\sqrt2]$ with cubes of side length $2\de$ and center $(\rho',u',y')$ so that $(\rho^0,u^0,y^0)$ is one of the centers. 
We deduce from \pref{le:covlet} that
\begin{align*}
  &\limsup_{\de\to0+}\limsup_{\de_1\to0+}\limsup_{N\to\8}\frac1N  \log I_N^0(\de_1;[-K,-\sqrt2]) \le -\frac12\log(2\pi) -\frac12\log D'(0)\\
  & \quad+\limsup_{\de\to0+} \max_{(\rho',u',y') \atop \text{ centers of cubes}} \Big\{ \sup_{\rho'-\de < \rho< \rho'+\de, u'-\de< u< u'+\de,\atop y'-\de< y< y'+\de}\psi_*(\rho,u,y)\\
  &\quad - \inf_{\rho'-\de < \rho< \rho'+\de, u'-\de< u< u'+\de,\atop y'-\de< y< y'+\de,0<x<\sfm((y'+\de)\wedge -\sqrt2)} \Lambda^*((x-3\sqrt\de )_+;f_{y'-2\de,\de})+ \frac1{2\sfb^2} [(\sfa+\sqrt{-D''(0)}x)_-]^2\Big\}\\
  &=-\frac12\log(2\pi) -\frac12\log D'(0)+ \psi_*(\rho^0,u^0,y^0)- \ix^-(\rho^0,u^0,y^0).
\end{align*}
Here we understand that the supremum and infimum were always taken within $(R_1,R_2)\times \bar E\times [-K,-\sqrt2)$. Let us consider $I_N^0(\de_1;[-\sqrt2-\eps,-\sqrt2+\eps])$. From \pref{eq:abdef}, we know for fixed $\rho$, $u$ and all $y\ge-\sqrt2-\eps$,
\begin{align*}
  \sfa_N(\rho,u,y)\indi\{(\frac{N-1}{N})^{1/2}\la_1>y\}&\le \sfa_N(\rho,u,-\sqrt2-\eps)\indi\{(\frac{N-1}{N})^{1/2}\la_1>y\}\\
  &\le \sfa_N(\rho,u,-\sqrt2-\eps)\indi\{(\frac{N-1}{N})^{1/2}\la_1>-\sqrt2-\eps\}.
\end{align*}
Using \pref{eq:gau2} together with continuity of functions in question, we find
\begin{align*}
  &\limsup_{N\to\8}\frac1N  \log I_N^0(\de_1;[-\sqrt2-\eps,-\sqrt2+\eps])\le -\frac12\log(2\pi) -\frac12\log D'(0)\\
  &+ \limsup_{N\to\8}\frac1N  \log   \int_{R_1}^{R_2} \int_{E}  \exp\Big\{(N-3)\sup_{\nu\in B_K(\si_{\rm sc},\de_1),\atop -\sqrt2-\eps<y<-\sqrt2+\eps}[\Psi(\nu,-\sqrt2-\eps)+\psi(\nu,\rho,u,y) -\Psi(\nu,y) ]\Big \}\\
  &\  \ez\Big[ \Big(\mathsf{a}_N(\rho,u,-\sqrt2-\eps)\Phi(\frac{\sqrt N\sfa_N(\rho,u,-\sqrt2-\eps)}\sfb) +\frac{\sfb}{\sqrt{2\pi N}} e^{-\frac{N\sfa_N(\rho,u,-\sqrt2-\eps)^2}{2\sfb^2}}\Big)\\
  &\quad \indi\{(\frac{N-1}{N})^{1/2}\la_1>-\sqrt2-\eps\} \Big] \dd u \dd \rho.
\end{align*}
As in the proof of \pref{le:covlet}, for $N$ large enough since $\frac{1}{(\frac{N-1}{N})^{1/2}\la_i-(-\sqrt2-\eps)}>\frac{1}{\la_i-(-\sqrt2-2\eps)}$ on $\{(\frac{N-1}{N})^{1/2}\la_1>-\sqrt2-\eps\}$, for $\eps_1,\eps_2>0$ we deduce from \pref{pr:ge} that
\begin{align*}
  &\ez\Big[ \indi\{(\frac{N-1}{N})^{1/2}\la_1>-\sqrt2-\eps\} \\
  &\qquad \Big(\mathsf{a}_N(\rho,u,-\sqrt2-\eps)\Phi(\frac{\sqrt N\sfa_N(\rho,u,-\sqrt2-\eps)}\sfb) +\frac{\sfb}{\sqrt{2\pi N}} e^{-\frac{N\sfa_N(\rho,u,-\sqrt2-\eps)^2}{2\sfb^2}}\Big)\Big]\\
  &\le \Big\lceil \frac{\sfm(-\sqrt2-2\eps)}{\eps_1}\Big\rceil \max_{0\le j\eps_1\le \sfm(-\sqrt2-2\eps)} \Big\{ \pz \Big(\frac{1}{N}\sum_{i=1}^{N-1} f_{-\sqrt2-2\eps,\eps}(\la_i)Z_i^2\in[0, \sfm(-\sqrt2-2\eps)-j\eps_1 ] \Big) \\
    & \quad \Big[\frac{\sfb}{\sqrt{2\pi N}} e^{-\frac{N(\sfa+\sqrt{-D''(0)}(j\eps_1+\eps_1))^2 }{2\sfb^2}} +[\sfa+\sqrt{-D''(0)}(j\eps_1+\eps_1)]\Phi\Big(\frac{\sqrt N (\sfa + \sqrt{-D''(0)}(j\eps_1+\eps_1))}{\sfb}\Big) \Big]\Big\}\\
    &\quad +\Big[\sfa \Phi\Big(\frac{\sqrt N \sfa }{\sfb}\Big)+\frac{\sfb}{\sqrt{2\pi N}} e^{-\frac{N\sfa^2 }{2\sfb^2}} \Big]\\
  &\le \Big[\sfa \Phi\Big(\frac{\sqrt N \sfa }{\sfb}\Big) +\frac{\sfb}{\sqrt{2\pi N}} e^{-\frac{N\sfa^2 }{2\sfb^2}} \Big]+ \Big\lceil \frac{\sfm(-\sqrt2)}{\eps_1}\Big\rceil \max_{0\le x\le \sfm(-\sqrt2-2\eps)} \Big\{ e^{-N(\Lambda^*(x)-\eps_2)} \\
    & \quad \Big[\frac{\sfb}{\sqrt{2\pi N}} e^{-\frac{N(\sfa+\sqrt{-D''(0)}(x+\eps_1))^2 }{2\sfb^2}} +[\sfa+\sqrt{-D''(0)}(x+\eps_1)]\Phi\Big(\frac{\sqrt N (\sfa + \sqrt{-D''(0)}(x+\eps_1))}{\sfb}\Big) \Big]\Big\}.
\end{align*}
Here we wrote $\sfa=\sfa(\rho,u,-\sqrt2-2\eps)$. With the same covering argument as for $I_N^0(\de_1;[-K,-\sqrt2])$, we find
\begin{align*}
   &\limsup_{\eps\to0+}\limsup_{\de_1\to0+,\atop \eps_1\to0+,\eps_2\to0+}\limsup_{N\to\8}\frac1N  \log I_N^0(\de_1;[-\sqrt2-\eps,-\sqrt2+\eps])\\
   &\le -\frac12\log(2\pi) -\frac12\log D'(0)+\limsup_{\eps\to0+}\sup_{R_1<\rho<R_2,\atop u\in \bar E} [\psi_*(\rho,u,-\sqrt2-2\eps)- \ix^-(\rho,u,-\sqrt2-2\eps)]\\
   &\le -\frac12\log(2\pi) -\frac12\log D'(0)+ \psi_*(\rho^0,u^0,y^0)- \ix^-(\rho^0,u^0,y^0).
\end{align*}
We have completed the proof.
\end{proof}

\subsection{Lower bound}
Recalling $I^0(E,(R_1,R_2))$ as in \pref{eq:ikdef}, we have the following lower bound.
\begin{proposition}\label{pr:lbk0}
 Suppose $ E$ is open and $R_2<\8$. Then
 \begin{align*}
   \liminf_{N\to\8}\frac1N\log I^0(E,(R_1,R_2))\ge \frac12\log[-4D''(0)] -\frac12\log(2\pi)-\frac12\log D'(0)\\
   +\sup_{(\rho,u,y)\in F} [\psi_*(\rho,u,y)- \ix^-(\rho,u,y)],
 \end{align*}
 where $F=\{(\rho,u,y): \rho\in[R_1,  R_2], u\in\bar E, y\le -\sqrt2\}$.
\end{proposition}
\begin{proof}
   Without loss of generality assume $E\neq\emptyset$. Choose $(\rho^0,u^0,y^0)$  as in the proof of \pref{pr:ubk0} and $\de>0$ small enough so that
$$(\rho^0-\de,\rho^0+\de)\times (u^0-\de,u^0+\de)\times (y^0-\de,y^0)\subset (R_1,R_2)\times E\times (-K,-\sqrt2).$$
Here for simplicity in writing, we assumed $u^0$ is in the interior of $\bar E$. If $u^0$ is on the boundary, we can simply replace the interval $(u^0-\de,u^0+\de)$ with $(u^0,u^0+\de)$ or $(u^0-\de,u^0)$ in an obvious way. By restriction, we have
\begin{align*}
  &I^0(E,(R_1,R_2))  \\
  \ge & \frac{\sqrt{N}[-4D''(0)]^{N/2}}{(2\pi)^{(N+2)/2} D'(0)^{N/2} } \int_{\rho^0-\de}^{\rho^0+\de} \int_{u^0-\de}^{u^0+\de}\int_{y^0-\de}^{y^0-\de/2}  \frac{e^{-\frac{(u-m_Y)^2}{2\si_Y^2}-\frac{N \mu^2 \rho^2}{2D'(0)} -\frac{N(\sqrt{ -4D''(0)}y+m_2)^2}{2(-2D''(0)-\bt^2)} }} {\si_Y\sqrt{-2D''(0)-\bt^2}}
 \\
  & \ez[ e^{(N-1)\Psi(L((\frac{N-1}{N})^{1/2}\la_1^{N-1}),y)} \indi\{L(\la_1^{N-1})\in B(\si_{\rm sc}, \de_1),(\frac{N-1}{N})^{1/2}\la_1>y\} \\
  &(\mathsf{a}_N\Phi(\frac{\sqrt N\sfa_N}\sfb) +\frac{\sfb}{\sqrt{2\pi N}} e^{-\frac{N\sfa_N^2}{2\sfb^2}})]
    \rho^{N-1}  \dd y \dd u  \dd\rho\\
    =: &\frac{\sqrt{N}[-4D''(0)]^{N/2}}{(2\pi)^{(N+2)/2} D'(0)^{N/2} } I_N^0(\de,\de_1).
\end{align*}
Let $\hat x=\hat x(\rho^0,u^0,y^0)$ as defined in \pref{eq:ixmin}. Note that for any $y\in [y^0-\de, y^0-\frac\de2]$ we have $y^0-\frac\de2>(\frac{N}{N-1})^{1/2}y$ and $f_{y^0-\frac\de2,\frac\de4}(x)  \ge \frac1{x-(y^0-\frac\de2)}\indi\{x-(y^0-\frac\de2)\ge \frac\de4\}$. Choose $N$ large enough so that  $\sfm(y^0-2\de)-\hat x\le (\frac{N-1}{N})^{1/2}(\sfm(y^0-\de)-\hat x)$. Since $x\mapsto x\Phi(\frac{\sqrt N x}\sfb) +\frac{\sfb}{\sqrt{2\pi N}} e^{-\frac{N x^2}{2\sfb^2}}$ is positive and strictly increasing,
\begin{align*}
    &\ez\Big[e^{(N-1)\Psi(L((\frac{N-1}{N})^{1/2}\la_1^{N-1}),y)} \indi\{L((\frac{N-1}{N})^{1/2}\la_1^{N-1})\in B(\si_{\rm sc}, \de_1),(\frac{N-1}{N})^{1/2}\la_1>y\} \\
    &\qquad  \Big(\mathsf{a}_N\Phi(\frac{\sqrt N\sfa_N}\sfb) +\frac{\sfb}{\sqrt{2\pi N}} e^{-\frac{N\sfa_N^2}{2\sfb^2}}\Big)\Big] \\
&\ge \ez \Big[e^{(N-1)\Psi(L((\frac{N-1}{N})^{1/2}\la_1^{N-1}),y)}\Big(\mathsf{a}_N\Phi(\frac{\sqrt N\sfa_N}\sfb) +\frac{\sfb}{\sqrt{2\pi N}} e^{-\frac{N{\sfa_N}^2}{2\sfb^2}}\Big) \indi\Big\{ (\frac{N-1}{N})^{1/2}\la_1>y,
    \\
    &\quad L((\frac{N-1}{N})^{1/2} \la_1^{N-1})\in B(\si_{\rm sc}, \de_1), \frac{1}{N}\sum_{i=1}^{N-1} \frac{Z_i^2}{\la_i-(\frac{N}{N-1})^{1/2}y}\le(\frac{N-1}{N})^{1/2}[ \sfm(y)-\hat x] \Big\}\Big]\\
&\ge \Big((\sfa+\sqrt{-D''(0)}\hat x)
    \Phi(\frac{\sqrt N(\sfa+\sqrt{-D''(0)}\hat x)}\sfb) +\frac{\sfb}{\sqrt{2\pi N}} e^{-\frac{N(\sfa+\sqrt{-D''(0)} \hat x)^2}{2\sfb^2}}\Big) \\
    &\ \ \ez\Big[e^{(N-1)\Psi(L((\frac{N-1}{N})^{1/2}\la_1^{N-1}),y)} \indi\Big\{L((\frac{N-1}{N})^{1/2}\la_1^{N-1})\in B(\si_{\rm sc}, \de_1),\\
    &\ \ \la_1>y^0-\frac\de4,\frac{1}{N}\sum_{i=1}^{N-1} f_{y^0-\frac\de2,\frac\de4}(\la_i){Z_i^2}\le \sfm(y^0-2\de)-\hat x \Big\} \Big]\\
&\ge \Big((\sfa+\sqrt{-D''(0)}\hat x)
    \Phi(\frac{\sqrt N(\sfa+\sqrt{-D''(0)} \hat x)}\sfb) +\frac{\sfb}{\sqrt{2\pi N}} e^{-\frac{N(\sfa+\sqrt{-D''(0)} \hat x)^2}{2\sfb^2}}\Big) \\
    & \ \ e^{(N-1)\inf_{\nu_->y^0-\frac\de4, y<y^0-\frac\de2, \nu\in B(\si_{\rm sc}, \de_1)} \Psi(\nu,y)} \Big[ \pz \Big(\frac{1}{N}\sum_{i=1}^{N-1} f_{y^0-\frac\de2,\frac\de4}(\la_i){Z_i^2}\le \sfm(y^0-2\de)-\hat x \Big)\\
    &\ \ - \pz \Big(\la_1\le y^0-\frac\de4,\frac{1}{N}\sum_{i=1}^{N-1} f_{y^0-\frac\de2,\frac\de4}(\la_i){Z_i^2}\le \sfm(y^0-2\de)-\hat x \Big)\\
    &\ \ - \pz(L((\frac{N-1}{N})^{1/2}\la_1^{N-1})\notin B(\si_{\rm sc}, \de_1))\Big].
  \end{align*}
  Here $\nu_-$ denotes the lower edge of the support of $\nu$. Note that $\Psi(\nu,y)$ is a continuous function for $\nu \in B(\si_{\rm sc},\de_1)$ with $\nu_- >y^0-\frac\de4$ and $y<y^0-\frac{\de}{2}$.
  Using \pref{eq:phiblim}, Propositions \ref{pr:ge} and \ref{pr:geup}, since all functions in question are continuous on compact sets, we have
\begin{align*}
  &\liminf_{\de_1\to0+}\liminf_{N\to\8}\frac1N I_N^0(\de,\de_1) \\
  &\ge \inf_{\rho^0-\de<\rho<\rho^0+\de, u^0-\de<u<u^0+\de, \atop y^0-\de<y<y^0-\frac\de2}\psi_*(\rho,u,y)-\Lambda^*[\hat x+\sfm(y^0-\frac\de2)-\sfm(y^0-2\de) ]\\
  & \quad -\sup_{\rho^0-\de<\rho<\rho^0+\de, u^0-\de<u<u^0+\de, \atop y^0-\de<y<y^0-\frac\de2} \frac{[(\sfa+\sqrt{-D''(0)}\hat x)_-]^2}{2\sfb^2}.
\end{align*}
Sending $\de\to0+$, we find by continuity and \pref{re:ylim} if necessary,
\begin{align*}
  &\liminf_{N\to\8} \frac1N \log I^0(E,(R_1,R_2))\\
  & \ge \frac12\log [-4D''(0)]-\frac12\log(2\pi) -\frac12\log D'(0)+\psi_*(\rho^0,u^0,y^0)- \ix^-(\rho^0,u^0,y^0).
\end{align*}
The proof is complete.
\end{proof}
\begin{proof}[Proof of Theorem \ref{th:criticalfix1}]
The case for $\bar E$ to be compact and $R_2<\8$ have been handled by Propositions \ref{pr:ubk0} and \ref{pr:lbk0}. Suppose $\bar E$ is not compact or $R_2=\8$.
Using \pref{eq:psi*lim} and \pref{le:exptt}, we may choose $R<\8$ and $T<\8$ large enough such that
  \begin{align*}
    &\lim_{N\to\8} \frac1N \log\ez \Crt_{N,0}(E, (R_1,R_2)) =	\lim_{N\to\8} \frac1N \log\ez \Crt_{N,0}(E \cap(-T,T), (R_1,R_2)\cap [0,R])\\
&=\frac12 \log[-4D''(0)] -\frac12\log D'(0) +\frac12+\sup_{y\le -\sqrt2, R_1< \rho<R\wedge R_2, u\in \bar E\cap [-T, T], }\psi_*(\rho,u,y)\\
&=\frac12 \log[-4D''(0)] -\frac12\log D'(0) +\frac12+\sup_{(\rho,u,y)\in F}\psi_*(\rho,u,y).\qedhere
  \end{align*}
\end{proof}

\begin{remark}\label{re:63}
Let us justify that $\ix^-(\rho^0,u^0,y^0)$ cannot always vanish. From \pref{eq:abdef1} we compute
\begin{align*}
\partial_y \sfa&=-\frac{(-2D''(0)-\al\bt\rho^2-\bt^2) \sqrt{-4D''(0)}}{-2D''(0)-\bt^2} -\sqrt{-D''(0)}\sfm'(y) <0,\\
\partial_u \sfa &= \frac{-2D''(0)\al\rho^2}{(-2D''(0)-\bt^2)\sqrt{D(\rho^2)-\frac{D'(\rho^2)^2 \rho^2}{D'(0)} }}<0.
\end{align*}
We may choose $\rho$ and (large) $u$ appropriately so that $\sfa<0$ for all $y\in[-2\sqrt2, -\sqrt2]$.
From \pref{le:ixmin}, we see that $\ix^-(\rho,u,y)$ is differentiable with possible exception for $\sfa=0$. In particular, for $\sfa<0$, using $ \partial_x \ix_-(\rho,u,y; \hat x)=0$,
\begin{align*}
\partial_y\ix^-(\rho,u,y)&=\partial_y \ix_-(\rho,u,y;\hat x)+ \partial_x \ix_-(\rho,u,y; \hat x) \partial_y \hat x\\
&= \partial_y \Lambda^*(\hat x;y)+\frac1\sfb (\sfa+\sqrt{-D''(0)}\hat x)_-\partial_y\sfa <0.
\end{align*}
Recall \pref{eq:uvcov} and let $J=\sqrt{-2D''(0)}$. In \cite{AZ20}*{Example 2}, we computed
\begin{align}\label{eq:parypsi}
  \partial_y \psi_* =\frac{-(\bt^2+J^2)y-\sqrt2J (\mu+\bt v)}{J^2-\bt^2}-\sgn(y)\sqrt{y^2-2}\indi\{|y|>\sqrt2\}.
\end{align}
For $y<-\sqrt2$ and $u$ large enough so that $\mu+\bt v<0$,
\begin{align*}
\partial_y [\psi_*(\rho,u,y)-\ix^-(\rho,u,y)]>0.
\end{align*}
Hence, to maximize $\psi_*(\rho,u,y)-\ix^-(\rho,u,y)$ when the critical value $u$ is restricted to large values, we must take $y^0=-\sqrt2$ which ensures $\ix^-(\rho^0,u^0,y^0)>0$.
\end{remark}
\begin{example}\label{ex:3}
  \rm
  Here we use \pref{th:criticalfix1} to recover \pref{th:fixktotal} for the case $k=0$. Let $0\le R_1<R_2\le\8$ and $E=\rz$. This removes restriction on the range of the random field. Recall \pref{eq:uvcov} and \pref{eq:m12cov}. We rewrite
\begin{equation}\label{eq:m12a}
\begin{split}
\sfa &=\frac{\al\bt\rho^2(\frac{J^2v}{\bt}+\mu)-(J^2-\al\bt\rho^2-\bt^2)\sqrt2 Jy-(J^2-\bt^2)\frac1{\sqrt2}J\sfm(y)}{J^2-\bt^2}\\
&= \frac{J}{\sqrt2}(-y+\sqrt{y^2-2}) +\frac{\al \rho^2 (J^2v+\sqrt2 J \bt y+\mu\bt)}{J^2-\bt^2}.
\end{split}
\end{equation}
In \cite{AZ20}*{Example 2}, we calculated
\begin{align*}
\partial_v\psi_*&= \frac{-J^2v -\bt (\sqrt2Jy +\mu )}{J^2-\bt^2},\  \
\end{align*}
and solving $\partial_v\psi_*=0$, we found
\begin{align}\label{eq:musbt}
    v=-\frac{\bt(\sqrt2Jy+\mu)}{J^2}, \ \ \sqrt2Jy+\mu+ \bt v = \frac{(\sqrt2Jy +\mu)(J^2-\bt^2)}{J^2}.
\end{align}
Recalling \pref{eq:phi*}, we can eliminate $v$ and rewrite
\begin{align}\label{eq:psids0}
  \psi_*(\rho,u,y )= -\frac12y^2-\frac12-\frac12\log2 -J_1(-|y|)\indi\{|y|>\sqrt2\}-\frac{\sqrt2 \mu y}{J}-\frac{\mu^2}{2J^2}-\frac{\mu^2\rho^2}{2D'(0)}+\log \rho.
\end{align}
If the first order condition \pref{eq:musbt} holds, we would find $\sfa=\frac{J}{\sqrt2}(-y+\sqrt{y^2-2})>0$ for any $y\le-\sqrt2$.

Recall that $(\rho^0,u^0,y^0)$ is a maximizer of $[\psi_*(\rho,u,y)- \ix^-(\rho,u,y)]$. We claim $\ix^-(\rho^0,u^0,y^0)=0$. Indeed, suppose $\ix^-(\rho^0,u^0,y^0)>0$ and thus $\sfa(\rho^0,u^0,y^0)<0$. Then $(\rho^0,u^0,y^0)$ cannot verify condition \pref{eq:musbt}. Note that we can always find a point $(\rho^0,u_1^0,y^0)$ satisfying \pref{eq:musbt} which differs from $(\rho^0,u^0,y^0)$ in the second coordinate. Since $\ix^-(\rho^0,u_1^0,y^0)=0$, by strict concavity of $u\mapsto \psi_*(\rho^0,u,y^0)$, we have
\[
\psi_*(\rho^0,u^0,y^0)<\psi_*(\rho^0,u_1^0,y^0)\le  \psi_*(\rho^0,u^0,y^0)- \ix^-(\rho^0,u^0,y^0),
\]
which contradicts our assumption.  Thus we find the same complexity function as that for all critical points and the function $\ix^-$ does not play a role in this case. In particular, if $\mu> J$ we have $y^0=-\frac1{\sqrt2}(\frac\mu{J}+\frac{J}\mu)<-\sqrt2$. If $\mu\le J$, since $\psi$ is concave in $y$ and is increasing when $y$ is small enough, we must take $y^0=-\sqrt2$. Observe that this includes the case $\mu\le0$. Plugging these values into \pref{eq:psids0}, we can obtain the conclusion of \pref{th:fixktotal} for $k=0$ in the same way as that in \cite{AZ20}*{Example 2} when $B_N$ is a shell. We summarize the results as follows.

\emph{Case 1}: $\mu\neq0$. We only give results for $R_1<\frac{\sqrt{D'(0)}}{|\mu|}$. Then $\rho^0=\rho_*$ where $\rho_*$ is
\begin{align*}
  \rho_*&=\begin{cases}
    \frac{\sqrt{D'(0)}}{|\mu|},& \text{ if } R_2>\frac{\sqrt{D'(0)}}{|\mu|},\\
    R_2,&\text{ otherwise},
  \end{cases}\\
  u^0& =\begin{cases}
   \Big(\frac\mu{-2D''(0)}-\frac{\sqrt2}{\sqrt{-D''(0)}}\Big) [D'(\rho_*^2)-D'(0)]+ \frac{\mu\rho_*^2}{2}-\frac{\mu D'(\rho_*^2) \rho_*^2}{D'(0) }, & \mu\le \sqrt{-2D''(0)},\\
   \frac{D'(\rho_*^2)-D'(0)}{\mu }+\frac{\mu\rho_*^2}{2}-\frac{\mu D'(\rho_*^2) \rho_*^2}{D'(0) }, & \mu >\sqrt{-2D''(0)},
  \end{cases}
\end{align*}
and
\begin{align*}
&\psi_*(\rho^0,u^0,y^0)\\
&= \begin{cases}
   -2-\frac12\log2+\frac{\sqrt2\mu}{\sqrt{-D''(0)}} +\frac{\mu^2}{4D''(0)} +\log\frac{\sqrt{D'(0)}}{|\mu|}, &\mu\le \sqrt{-2D''(0)}, R_2>\frac{\sqrt{D'(0)}}{|\mu|},\\
   -\frac12\log2-\log\sqrt{-2D''(0)}-\frac12+\frac12\log D'(0), &\mu> \sqrt{-2D''(0)}, R_2>\frac{\sqrt{D'(0)}}{|\mu|},\\
    -\frac32-\frac12\log2+\frac{\sqrt2\mu}{\sqrt{-D''(0)}} +\frac{\mu^2}{4D''(0)}+\log R_2- \frac{\mu^2R_2^2}{2D'(0)}, &\mu\le \sqrt{-2D''(0)}, R_2\le \frac{\sqrt{D'(0)}}{|\mu|},\\
   -\frac12\log2-\log\sqrt{-2D''(0)}+\log|\mu| +\log R_2- \frac{\mu^2R_2^2}{2D'(0)}, &\mu> \sqrt{-2D''(0)}, R_2\le \frac{\sqrt{D'(0)}}{|\mu|}.
  \end{cases}
\end{align*}

\emph{Case 2}: $\mu=0$. We have $\rho^0=R_2<\8$,
\begin{align*}
  u^0&=-\frac{\sqrt2[D'(R_2^2)-D'(0)]}{\sqrt{-D''(0)}},\\
  \psi_*(\rho^0,u^0,y^0)&=-\frac32-\frac12\log2+\log R_2.
\end{align*}
\pref{th:criticalfix1} suggests that the local minima around the value $u^0$ and variable $\rho^0$ given above dominate all other places. Moreover, we observe that $\psi_*(\rho^0,u^0,y^0)=\psi_*(\rho_*,u_*,y_*)$ when $\mu> \sqrt{-2D''(0)}$, where $u_*$ and $y_*$ were given in \cite{AZ20}*{Example 2}. Actually $y_*=y^0$ and $u_*=u^0$ for $\mu> \sqrt{-2D''(0)}$.
Finally, we remark that as in \cite{ABC13}*{Theorem 2.12}, by investigating when the complexity function equals zero, we may obtain a lower bound of the global minimum of the Hamiltonian in the large $N$ limit.
\end{example}

\section{Saddles with fixed index $k\ge1$}\label{se:kge1}

Recall $I^k(E,(R_1,R_2))$ and $II^k(E,(R_1,R_2))$ as in \pref{eq:ikdef}. One would expect similar behavior of $I^0$ and $I^k$. A moment of reflection, however, reveals that the method to prove upper bound for $I^0$ does not work for $I^k$, simply because $(\frac{N-1}{N})^{1/2}\la_i-y<0$ for $i\le k$ on $\{(\frac{N-1}{N})^{1/2}\la_k<y<(\frac{N-1}{N})^{1/2}\la_{k+1}\}$. Fortunately, it turns out that we do not need that precise upper bound, which would always be dominated by $II^k$ in the large $N$ limit. On the other hand, the upper bound for $II^k$ and the lower bound for index $k\ge1$ will be established following ideas similar to that for local minima.

\subsection{Upper bound involving $J_k$ }
For the upper bound of $I^k$, using \pref{eq:schur}, we have
\begin{align*}
  &\ez[|\det G|  \indi\{i(G_{**})=k,\zeta>0\}]\le \ez(|\det G_{**}| | z_1'-\xi^\sfT G_{**}^{-1} \xi |  \indi\{i(G_{**})=k \})\\
&= [-4D''(0)]^{\frac{N-1}{2}} \ez[| z_1'- Q(z_3') | |\det ((\frac{N-1}{N})^{1/2}\GOE_{N-1}-z_3' I_{N-1}) | \\
&\quad \indi\{(\frac{N-1}{N})^{1/2}\la_k<z_3'<(\frac{N-1}{N})^{1/2}\la_{k+1} \}] \\
&\le [-4D''(0)]^{\frac{N-1}{2}}\Big(  \ez\Big[|z_1'| \prod_{i=1}^{N-1} |(\frac{N-1}{N})^{1/2}\la_i-z_3'|\indi\{(\frac{N-1}{N})^{1/2}\la_k<z_3'<(\frac{N-1}{N})^{1/2}\la_{k+1} \}\Big] \\
& \ \ +\frac{\sqrt{-D''(0)}}{N}\sum_{i=1}^{N-1}\ez\Big[Z_i^2\prod_{j\neq i} |(\frac{N-1}{N})^{1/2}\la_j-z_3'|\indi\{(\frac{N-1}{N})^{1/2}\la_k<z_3'<(\frac{N-1}{N})^{1/2}\la_{k+1} \}\Big ] \Big).
\end{align*}
To handle the two terms, let
\begin{align}
  &I_1^k(E,(R_1,R_2)) = [-4D''(0)]^{\frac{N-1}{2}} \int_{R_1}^{R_2} \int_{E} \ez\Big[|z_1'| \prod_{i=1}^{N-1} |(\frac{N-1}{N})^{1/2}\la_i-z_3'| \notag\\
  &\ \ \indi\{(\frac{N-1}{N})^{1/2}\la_k<z_3'<(\frac{N-1}{N})^{1/2}\la_{k+1} \} \Big]\frac{ e^{-\frac{(u-m_Y)^2}{2\si_Y^2}}}{\sqrt{2\pi}\si_Y} \frac{e^{-\frac{N \mu^2 \rho^2}{2D'(0)}}}{(2\pi)^{N/2} D'(0)^{N/2}}  \rho^{N-1}   \dd u  \dd\rho,\notag \\
  &I_2^k(E,(R_1,R_2)) = \frac{[-4D''(0)]^{\frac{N}{2}} }{2N} \sum_{i=1}^{N-1}\int_{R_1}^{R_2} \int_{E} \ez\Big[Z_i^2\prod_{j\neq i} |(\frac{N-1}{N})^{1/2}\la_j-z_3'| \notag \\
  &\ \ \indi\{(\frac{N-1}{N})^{1/2}\la_k<z_3'<(\frac{N-1}{N})^{1/2}\la_{k+1} \}\Big ]\frac{ e^{-\frac{(u-m_Y)^2}{2\si_Y^2}}}{\sqrt{2\pi}\si_Y} \frac{e^{-\frac{N \mu^2 \rho^2}{2D'(0)}}}{(2\pi)^{N/2} D'(0)^{N/2}}  \rho^{N-1}   \dd u  \dd\rho.\label{eq:ik12}
\end{align}
It follows that $I^k\le I_1^k+I_2^k$. We first consider the upper bound for $I_1^k$. With \pref{eq:z13con0} and \pref{eq:absgau}, we find for $y\le -\sqrt2$,
\begin{align*}
  \ez[|z_1'||z_3'=y] & \le \sqrt{\frac2\pi}\frac{\sfb}{\sqrt N} + |\bar\sfa|.
\end{align*}
Since by \pref{eq:conineq}, for any $y\ge -\sqrt2+\eps$, $\pz(\la_{k+1}>y)\le e^{-c N^2}$, using Lemmas \ref{le:exptti11} and \ref{le:exptti12}, for large $K>0$ and small $\de,\eps>0$, we have
\begin{align*}
\limsup_{N\to\8}\frac1N\log I_1^k( E,(R_1,R_2)) \le \max\Big\{ \limsup_{N\to\8}\frac1N\log I_1^k( E,(R_1,R_2), [-K,-\sqrt2);\de),\\
\limsup_{N\to\8}\frac1N\log I_1^k( E,(R_1,R_2), [-\sqrt2,-\sqrt2+\eps];\de)\Big\},
\end{align*}
where
\begin{align*}
  &I_1^k( E,(R_1,R_2),G;\de)= [-4D''(0)]^{\frac{N}{2}} \int_{R_1}^{R_2} \int_{ E}\int_{G} \ez\Big[\Big(\sqrt{\frac2\pi}\frac{\sfb}{\sqrt N} + |\bar\sfa|\Big) \\
  &\ \ \prod_{i=1}^{N-1} |(\frac{N-1}{N})^{1/2}\la_i-y | \indi\{(\frac{N-1}{N})^{1/2}\la_k< y, L((\frac{N-1}{N})^{1/2}\la_{i=1}^{N-1})\in B_K(\si_{\rm sc}, \de) \}\Big] \\
  &\ \ \frac{ e^{-\frac{(u-m_Y)^2}{2\si_Y^2}}}{\sqrt{2\pi}\si_Y} \frac{e^{-\frac{N \mu^2 \rho^2}{2D'(0)}}}{(2\pi)^{N/2} D'(0)^{N/2}} \frac{ \exp\{-\frac{N(\sqrt{ -4D''(0)}y+m_2)^2}{2(-2D''(0)-\bt^2)}\}} {\sqrt{2\pi(-2D''(0)-\bt^2)} } \rho^{N-1}  \dd y \dd u  \dd\rho.
\end{align*}
For any $(\rho',u',y')\in (R_1,R_2]\times \bar E\times (-K,-\sqrt2]$ and $\de>0$, using the LDP of $\la_k$ \pref{eq:jk1x} and continuity of functions in question we have
\begin{align*}
  &\limsup_{N\to\8}\frac1N\log  I_1^k( (u'-\de,u'+\de),(\rho'-\de,\rho'+\de), (y'-\de,y'+\de);\de_1)\le \frac12\log[-4D''(0)]\\
  &\ \ -\frac12\log D'(0)-\frac12\log(2\pi)- kJ_1((y'+\de)\wedge-\sqrt2) +\sup_{\nu\in B_K(\si_{\rm sc}, \de_1),\rho'-\de<\rho<\rho'+\de, \atop u'-\de<u<u'+\de, y'-\de<y<y'+\de}\psi(\nu,\rho,u,y).
\end{align*}
Here as usual we understand all intervals are replaced with shorter intervals if they go out of $(R_1,R_2]\times \bar E\times [-K,-\sqrt2]$. As for the case of local minima in \pref{pr:ubk0}, we may assume $R_1>0$.
Since $\lim_{y\to-\8}\psi_*(\rho, u,y)=-\8$ by \pref{eq:psi*lim}, we may choose $K$ large enough so that
$$\sup_{R_1\le \rho\le R_2, u\in \bar E, y\le -\sqrt2} \psi_*( \rho, u,y)-kJ_1(y)= \max_{R_1\le \rho\le R_2, u\in \bar E, -K\le y\le -\sqrt2}  \psi_*(\rho, u,y)-kJ_1(y).$$
Here $K$ may depend on $R_1>0$. Let $(\rho_1^k,u_1^k,y_1^k)$ be a maximizer. Since  $\lim_{\rho\to 0+}\psi_*(\rho, u,y)=-\8$, we have $\rho_1^k>0$. Consider a cubic cover of $[R_1,R_2]\times \bar E\times [-K,-\sqrt2]$ with side length $2\de$  such that $(\rho_1^k,u_1^k,y_1^k)$ is one of the centers. It follows that
\begin{align*}
  &\limsup_{\de\to0+,\atop\de_1\to0+}\limsup_{N\to\8}\frac1N\log I_1^k( E,(R_1,R_2),[-K,-\sqrt2);\de_1) \\
  & \le \limsup_{\de\to0+,\atop\de_1\to0+}\max_{\text{cubes with} \atop \text{centers} (\rho',u',y')} \limsup_{N\to\8}\frac1N\log  I_1^k((u'-\de,u'+\de),(\rho'-\de,\rho'+\de), (y'-\de,y'+\de);\de_1)\\
  &\le  \frac12\log[-4D''(0)] -\frac12\log D'(0)-\frac12\log(2\pi)+\psi_*( \rho_1^k,u_1^k,y_1^k)-kJ_1(y_1^k).
\end{align*}
For $I_1^k( E,(R_1,R_2), [-\sqrt2,-\sqrt2+\eps];\de)$, 
when $y\in[-\sqrt2,-\sqrt2+\eps]$,
\begin{align*}
  &\prod_{i=1}^{N-1} |(\frac{N-1}{N})^{1/2}\la_i-y | \indi\{(\frac{N-1}{N})^{1/2}\la_k< y, L((\frac{N-1}{N})^{1/2}\la_{i=1}^{N-1})\in B_K(\si_{\rm sc}, \de_1) \}\\
  &\le e^{(N-1)\sup_{\nu\in B_K(\si_{\rm sc}, \de_1) }\Psi(\nu,y)}.
\end{align*}
 Note that
 \begin{align*}
 &I_1^k( E,(R_1,R_2), [-\sqrt2,-\sqrt2+\eps];\de)\le  [-4D''(0)]^{\frac{N}{2}} \int_{R_1}^{R_2} \int_{ E}\int_{-\sqrt2}^{-\sqrt2+\eps} \\
 &\ \   \prod_{i=1}^{N-1} |(\frac{N-1}{N})^{1/2}\la_i-y | \ez\Big[\Big(\sqrt{\frac2\pi}\frac{\sfb}{\sqrt N} + |\bar\sfa|\Big) \indi\{ L((\frac{N-1}{N})^{1/2}\la_{i=1}^{N-1})\in B_K(\si_{\rm sc}, \de) \}\Big] \\
 &\ \ \frac{ e^{-\frac{(u-m_Y)^2}{2\si_Y^2}}}{\sqrt{2\pi}\si_Y} \frac{e^{-\frac{N \mu^2 \rho^2}{2D'(0)}}}{(2\pi)^{N/2} D'(0)^{N/2}} \frac{ \exp\{-\frac{N(\sqrt{ -4D''(0)}y+m_2)^2}{2(-2D''(0)-\bt^2)}\}} {\sqrt{2\pi(-2D''(0)-\bt^2)} } \rho^{N-1}  \dd y \dd u  \dd\rho.
  \end{align*}
 Using the same covering argument as above, sending $\de\to0+,\de_1\to0+$ and $\eps\to0+$ sequentially, we have
\begin{align}\label{eq:i1k2}
  &\limsup_{\eps\to0+}\limsup_{\de\to0+\atop\de_1\to0+}\limsup_{N\to\8}\frac1N\log I_1^k( E,(R_1,R_2),[-\sqrt2,-\sqrt2+\eps];\de_1) \notag \\
  &\le \frac12\log[-4D''(0)] -\frac12\log D'(0)-\frac12\log(2\pi) +\sup_{R_1<\rho<R_2, u\in \bar E} \psi_*(\rho,u,-\sqrt2) \notag \\
  &\le \frac12\log[-4D''(0)] -\frac12\log D'(0)-\frac12\log(2\pi)+\psi_*( \rho_1^k,u_1^k,y_1^k)-kJ_1(y_1^k).
\end{align}
In fact, in this special case the covering argument is not necessary; we can directly bound the integrand from above since $J_1(-\sqrt2)=0$.

Let us consider $I_2^k$ as in \pref{eq:ik12}. Similar to the above, we can assume $R_1>0$ and choose $K$ large and $\de,\eps$ small such that
\begin{align*}
\limsup_{N\to\8}\frac1N\log I_2^k( E,(R_1,R_2)) &\le \max\Big\{ \limsup_{N\to\8}\frac1N\log I_2^k(E,(R_1,R_2), (-K,-\sqrt2);\de),\\
&\ \ \limsup_{N\to\8}\frac1N\log I_2^k(E,(R_1,R_2), [-\sqrt2,-\sqrt2+\eps];\de)
\Big\},
\end{align*}
where
\begin{align*}
  &I_2^k( E,(R_1,R_2),G;\de)= \frac{[-4D''(0)]^{\frac{N+1}{2}} }{2N} \sum_{i=1}^{N-1}\int_{R_1}^{R_2} \int_{ E}\int_{G}\\
  & \ \ \ez\Big[\prod_{j\neq i}^{N-1} |(\frac{N-1}{N})^{1/2}\la_j-y | \indi\{(\frac{N-1}{N})^{1/2}\la_k< y, L((\frac{N-1}{N})^{1/2}\la_{i=1}^{N-1})\in B_K(\si_{\rm sc}, \de) \}\Big] \\
  &\ \ \frac{ e^{-\frac{(u-m_Y)^2}{2\si_Y^2}}}{\sqrt{2\pi}\si_Y} \frac{e^{-\frac{N \mu^2 \rho^2}{2D'(0)}}}{(2\pi)^{N/2} D'(0)^{N/2}} \frac{ \exp\{-\frac{N(\sqrt{ -4D''(0)}y+m_2)^2}{2(-2D''(0)-\bt^2)}\}} {\sqrt{2\pi(-2D''(0)-\bt^2)} } \rho^{N-1}  \dd y \dd u  \dd\rho.
\end{align*}
Note that $ L((\frac{N-1}{N})^{1/2}\la_{i=1}^{N-1})\in B_K(\si_{\rm sc}, \de)$ implies that $ L((\frac{N-1}{N})^{1/2}\la_{i=1,i\neq j}^{N-1})\in B_K(\si_{\rm sc}, 2\de)$ for any $j=1,....,N-1$.
For any $(\rho',u',y')\in (R_1,R_2]\times \bar E\times (-K,-\sqrt2]$ and $\de>0$, the LDP of $\la_k$ and continuity of functions in question we have
\begin{align*}
  &\limsup_{N\to\8}\frac1N\log  I_2^k((u'-\de,u'+\de),(\rho'-\de,\rho'+\de), (y'-\de,y'+\de);\de_1)\le \frac12\log[-4D''(0)]\\
  &\ \ -\frac12\log D'(0)-\frac12\log(2\pi)- kJ_1((y'+\de)\wedge-\sqrt2) +\sup_{\nu\in B_K(\si_{\rm sc}, 2\de_1),\rho'-\de<\rho<\rho'+\de, \atop u'-\de<u<u'+\de, y'-\de<y<y'-\de}\psi(\nu,\rho,u,y).
\end{align*}
Here as usual we understand all intervals are replaced with shorter intervals if they go out of $(R_1,R_2]\times \bar E\times [-K,-\sqrt2]$. We may choose $K$ large enough so that
$$\sup_{R_1\le \rho\le R_2, u\in \bar E, y\le -\sqrt2} \psi_*( \rho, u,y)-kJ_1(y)= \max_{R_1\le \rho\le R_2, u\in \bar E, -K\le y\le -\sqrt2}  \psi_*(\rho, u,y)-kJ_1(y),$$
and $(\rho_1^k,u_1^k,y_1^k)$ is a maximizer. Using the same covering argument as for $I_1^k$, we find
\begin{align*}
  &\limsup_{\de\to0+,\atop\de_1\to0+}\limsup_{N\to\8}\frac1N\log I_2^k( E,(R_1,R_2),(-K,-\sqrt2);\de_1) \\
  & \le \limsup_{\de\to0+,\atop\de_1\to0+}\max_{\text{cubes with} \atop \text{centers} (\rho',u',y')} \limsup_{N\to\8}\frac1N\log  I_2^k((\rho'-\de,\rho'+\de), (u'-\de,u'+\de),(y'-\de,y'+\de);\de_1)\\
  &\le  \frac12\log[-4D''(0)] -\frac12\log D'(0)-\frac12\log(2\pi)+\psi_*( \rho_1^k,u_1^k,y_1^k)-kJ_1(y_1^k).
\end{align*}
Also, we have
\begin{align*}
  &\limsup_{\eps\to0+}\limsup_{\de\to0+\atop\de_1\to0+}\limsup_{N\to\8}\frac1N\log I_2^k( E,(R_1,R_2),[-\sqrt2,-\sqrt2+\eps];\de_1) \\
  &\le \frac12\log[-4D''(0)] -\frac12\log D'(0)-\frac12\log(2\pi)+\psi_*( \rho_1^k,u_1^k,y_1^k)-kJ_1(y_1^k).
\end{align*}
From here we conclude
\begin{proposition}\label{pr:ubk1}
 Assume $\bar E$ is compact and $R_2<\8$. Then
 \begin{align*}
   \limsup_{N\to\8}\frac1N\log I^k(E,(R_1,R_2))\le \frac12\log[-4D''(0)] -\frac12\log(2\pi)-\frac12\log D'(0)\\
   +\sup_{(\rho,u,y)\in F} [\psi_*(\rho,u,y)-k J_1(y)],
 \end{align*}
 where $F=\{(\rho,u,y): \rho\in(R_1,  R_2], u\in\bar E, y\le -\sqrt2\}$.
\end{proposition}

Let us turn to $II^k$.
Note that
\begin{align*}
  &\ez[|\det G| \indi\{i(G_{**})=k-1,\zeta<0\}]=\ez(|\det G_{**}| | z_1'-\xi^\sfT G_{**}^{-1} \xi |  \indi\{i(G_{**})=k-1 , \zeta<0\})\\
&= [-4D''(0)]^{\frac{N-1}{2}} \ez[ |\det ((\frac{N-1}{N})^{1/2}\GOE_{N-1}-z_3' I_{N-1}) | | z_1'- Q(z_3') | \\
&\ \ \indi\{z_1'- Q(z_3') <0\}\indi\{(\frac{N-1}{N})^{1/2}\la_{k-1}<z_3'<(\frac{N-1}{N})^{1/2}\la_{k} \} ]\\
&=[-4D''(0)]^{\frac{N-1}{2}} \ez\Big[ |\det ((\frac{N-1}{N})^{1/2}\GOE_{N-1}-z_3' I_{N-1}) | | z_1'- Q(z_3') |\indi\{z_1'- Q(z_3') <0\}\\
&\ \ \Big(\indi\{ (\frac{N-1}{N})^{1/2}\la_{k-1}<z_3'<(\frac{N-1}{N})^{1/2}\la_k<z_3'+\de'\}\\
&\ \ + \indi\{(\frac{N-1}{N})^{1/2}\la_{k-1}<z_3',z_3'+\de'\le (\frac{N-1}{N})^{1/2}\la_{k} \} \Big)\Big]
\end{align*}
for some small $\de'>0$ to be chosen later. Define
\begin{align*}
  &II_1^k(E,(R_1,R_2);\de')  = [-4D''(0)]^{\frac{N-1}{2}} \int_{R_1}^{R_2} \int_{E} \ez[ |\det ((\frac{N-1}{N})^{1/2}\GOE_{N-1}-z_3' I_{N-1}) | \\
  &\quad  | z_1'- Q(z_3') | \indi\{(\frac{N-1}{N})^{1/2} \la_k<z_3'+\de'\}] \frac{e^{-\frac{(u-m_Y)^2}{2\si_Y^2}}}{\sqrt{2\pi}\si_Y} \frac{e^{-\frac{N \mu^2 \rho^2}{2D'(0)}} }{(2\pi)^{N/2} D'(0)^{N/2}}   \rho^{N-1} \dd u  \dd\rho,\\
  &II_2^k(E,(R_1,R_2);\de')  =[-4D''(0)]^{\frac{N-1}{2}} \int_{R_1}^{R_2} \int_{E} \ez[ |\det ((\frac{N-1}{N})^{1/2}\GOE_{N-1}-z_3' I_{N-1}) |\\
  &\quad | z_1'- Q(z_3') | \indi\{z_1'- Q(z_3') <0\} \indi\{(\frac{N-1}{N})^{1/2}\la_{k-1}<z_3',z_3'+\de'\le (\frac{N-1}{N})^{1/2}\la_{k} \}] \\
  &\quad \frac{e^{-\frac{(u-m_Y)^2}{2\si_Y^2}}}{\sqrt{2\pi}\si_Y} \frac{e^{-\frac{N \mu^2 \rho^2}{2D'(0)}} }{(2\pi)^{N/2} D'(0)^{N/2}}   \rho^{N-1} \dd u  \dd\rho.
\end{align*}
It follows that $II^k( E,(R_1,R_2))\le II_1^k(E,(R_1,R_2);\de') +II_2^k(E,(R_1,R_2);\de')$. For $II_1^k$ we can follow verbatim the argument for $I^k$. The only difference is that we have an extra $\de'$ here, which can be sent to zero in the last step. Thus, we find
\begin{align}\label{eq:ii1k}
  \limsup_{\de'\to0+}\limsup_{N\to\8}\frac1N \log II_1^k(E,(R_1,R_2);\de')\le \frac12\log[-4D''(0)] -\frac12\log(2\pi)-\frac12\log D'(0)\notag\\
   +\sup_{(\rho,u,y)\in F} [\psi_*(\rho,u,y)-k J_1(x)].
\end{align}
It remains to understand $II^k_2$.

Using \pref{eq:z13con0}, \pref{eq:gau2} and by conditioning, we have
\begin{align*}
 &\ez\Big[ |\det ((\frac{N-1}{N})^{1/2}\GOE_{N-1}-z_3' I_{N-1}) | | z_1'- Q(z_3') | \\
 &\quad \indi\{z_1'- Q(z_3') <0,(\frac{N-1}{N})^{1/2}\la_{k-1}<z_3',z_3'+\de'\le(\frac{N-1}{N})^{1/2} \la_{k} \}\Big]  \\
&= \ez[ |\det ((\frac{N-1}{N})^{1/2}\GOE_{N-1}-z_3' I_{N-1}) | \indi\{(\frac{N-1}{N})^{1/2}\la_{k-1}<z_3', z_3'+\de'\le  (\frac{N-1}{N})^{1/2}\la_{k} \} \\
&\quad \ez(  | z_1'- Q(z_3') |\indi\{z_1'- Q(z_3') <0\}  | \la_1^{N-1}, z_3',\xi')]\\
&=\int_{\rz} \ez\Big[ |\det ((\frac{N-1}{N})^{1/2}\GOE_{N-1}- y I_{N-1}) | \indi\{(\frac{N-1}{N})^{1/2}\la_{k-1}< y, y+\de'\le (\frac{N-1}{N})^{1/2}\la_{k} \} \\
& \quad (-\mathsf{a}_N\Phi(-\frac{\sqrt N\sfa_N}\sfb) +\frac{\sfb}{\sqrt{2\pi N}} e^{-\frac{N \sfa_N^2}{2\sfb^2}})\Big] \frac{\sqrt {-4N D''(0)} \exp\{-\frac{N(\sqrt{ -4D''(0)}y+m_2)^2}{2(-2D''(0)-\bt^2)}\}} {\sqrt{2\pi(-2D''(0)-\bt^2)} } \dd y,
\end{align*}
where $\sfa_N, \sfb^2$ are given in \pref{eq:abdef}. By l'Hospital's rule, for $x\in\rz$ and $b>0$ we have
\begin{align}\label{eq:phib2}
  \lim_{N\to\8} \frac1N\log\Big[-x\Phi(-\frac{\sqrt N x}b) +\frac{b}{\sqrt{2\pi N}} e^{-\frac{\sqrt N x^2}{2b^2}} \Big]=-\frac{(x_+)^2}{2b^2},
\end{align}
where $x_+=x\vee 0$. Let $\eps>0$. Using Lemmas \ref{le:goecpt}, \ref{le:goecpt2}, \ref{le:exptti11}, and \ref{le:exptti12}, by choosing $K>0$ large and $\de_1,\de'>0 $ small such that
\begin{align*}
\limsup_{N\to\8}II_2^k( E,(R_1,R_2);\de')& =\max\{ \limsup_{N\to\8} II_2^k( E,(R_1,R_2),(-K,-\sqrt2);\de_1,\de'), \\
&\quad \limsup_{N\to\8} II_2^k(  E,(R_1,R_2),[-\sqrt2,-\sqrt2+\eps);\de_1,\de')\}
\end{align*}
where
\begin{align}\label{eq:ii2k}
  &II_2^k( E,(R_1,R_2),G;\de_1,\de') \notag\\
 & = [-4D''(0)]^{\frac{N}{2}} \int_{R_1}^{R_2} \int_{E}\int_{G} \ez\Big[ (-\mathsf{a}_N\Phi(-\frac{\sqrt N\sfa_N}\sfb) +\frac{\sfb}{\sqrt{2\pi N}} e^{-\frac{N \sfa_N^2}{2\sfb^2}})  e^{(N-1)\Psi(L((\frac{N-1}{N})^{1/2}\la_1^{N-1}),y)} \notag\\
  & \ \  \indi\{L((\frac{N-1}{N})^{1/2}\la_1^{N-1})\in B_{K}(\si_{\rm sc},\de_1),(\frac{N-1}{N})^{1/2}\la_{k-1}< y,y+\de'< (\frac{N-1}{N})^{1/2}\la_{k} \}\Big] \notag\\
&  \ \   \frac{ e^{-\frac{N(\sqrt{ -4D''(0)}y+m_2)^2}{2(-2D''(0)-\bt^2)}}} {\sqrt{2\pi(-2D''(0)-\bt^2)} } \frac{e^{-\frac{(u-m_Y)^2}{2\si_Y^2}}}{\sqrt{2\pi}\si_Y} \frac{e^{-\frac{N \mu^2 \rho^2}{2D'(0)}} }{(2\pi)^{N/2} D'(0)^{N/2}}   \rho^{N-1} \dd y \dd u  \dd\rho.
\end{align}
Note that
\begin{align*}
  &II_2^k(  E,(R_1,R_2),[-\sqrt2,-\sqrt2+\eps);\de_1,\de')\le  [-4D''(0)]^{\frac{N}{2}}\\
  & \ \ \int_{R_1}^{R_2} \int_{E}\int_{-\sqrt2}^{-\sqrt2+\eps}  \ez\Big[ e^{(N-1)\Psi(L((\frac{N-1}{N})^{1/2}\la_1^{N-1}),y)} \indi\{L((\frac{N-1}{N})^{1/2}\la_1^{N-1})\in B_{K}(\si_{\rm sc},\de_1)\} \Big]  \\
  &\ \  \frac{ e^{-\frac{N(\sqrt{ -4D''(0)}y+m_2)^2}{2(-2D''(0)-\bt^2)}}} {\sqrt{2\pi(-2D''(0)-\bt^2)} } \frac{e^{-\frac{(u-m_Y)^2}{2\si_Y^2}}}{\sqrt{2\pi}\si_Y} \frac{e^{-\frac{N \mu^2 \rho^2}{2D'(0)}} }{(2\pi)^{N/2} D'(0)^{N/2}}   \rho^{N-1} \dd y \dd u  \dd\rho.
\end{align*}
Using an argument similar to that of \pref{eq:i1k2}, we find
\begin{align}\label{eq:ii2kub}
  &\limsup_{\eps\to0+} \limsup_{\de_1\to0+\atop \de'\to0+}\limsup_{N\to\8} II_2^k(E,(R_1,R_2),[-\sqrt2,-\sqrt2+\eps);\de_1,\de') \notag\\
 &\le \frac12\log[-4D''(0)] -\frac12\log D'(0)-\frac12\log(2\pi) +\sup_{R_1<\rho<R_2, u\in \bar E} \psi_*(\rho,u,-\sqrt2) \notag \\
  &\le \frac12\log[-4D''(0)] -\frac12\log D'(0)-\frac12\log(2\pi)+\psi_*( \rho_1^k,u_1^k,y_1^k)-kJ_1(y_1^k).
\end{align}
It remains to control $II_2^k( E,(R_1,R_2),(-K,-\sqrt2);\de_1,\de')$ for $R_1>0$, which turns out to be a major obstacle for the case of fixed index.

\subsection{Upper bound involving $J_{k-1}$}\label{se:ex3}

Recall $ \Lambda^*(x;y)=\Lambda^*(x;f_{y,-\sqrt2-y})$ as in \pref{eq:lasy} and $\Lambda'(\frac{\sqrt2+y}{2}+;y)=-\sqrt2+\sfm(y)$. For $y<-\sqrt2$, let us define
\begin{align}
\ix_+(\rho,u,y;x)&= \Lambda^*(x;y)+ \frac1{2\sfb^2} [(\sfa+\sqrt{-D''(0)}x)_+]^2,\notag\\
\ix^+(\rho,u,y) & = \begin{cases}
  \ix_+(\rho,u,y;0)=0, & \mbox{if } \sfa\le0, \\
  \ix_+(\rho,u,y;\tilde x), & \mbox{otherwise},
  \end{cases}\label{eq:ix+}
\end{align}
where
\begin{align*}
\tilde x =\begin{cases}
\frac{\sfb^2(\sqrt2+y)}{2D''(0)} - \frac{\sfa}{\sqrt{-D''(0)}}, &\mbox{if }\frac{\sfb^2(\sqrt2+y)}{2D''(0)}  - \frac{\sfa}{\sqrt{-D''(0)}}\le  \Lambda'(\frac{\sqrt2+y}{2}+;y),\\
\hat x, & \mbox{otherwise},
\end{cases}
\end{align*}
and $\hat x$ is given as in \pref{eq:ixmin}.   When $\sfa\le0$, it is clear that $\ix^+(\rho,u,y)=\inf_{x\le0}  \ix_+(\rho,u,y;x)=0$ and the unique minimizer is $x=0$. We may extend $\ix^+(\rho,u,y)$ to $y=-\sqrt2$ by continuity.


\begin{lemma}\label{le:ix+min}
  For $\sfa>0$, we have
  \[
  \ix^+(\rho,u,y)=\inf_{ 
   x\le0}  \ix_+(\rho,u,y;x)>0
  \]
  and the minimizer is unique, which is in $[-\frac{\sfa}{\sqrt{-D''(0)}} ,0]$.
\end{lemma}
\begin{proof}
Since $x\mapsto \Lambda^*(x;y)$ is strictly decreasing for $x\le 0$,
  \[
  \inf_{ 
   x\le0}  \ix_+(\rho,u,y;x)=\inf_{ -\frac{\sfa}{\sqrt{-D''(0)}}\le
   x\le0}  \ix_+(\rho,u,y;x),
  \]
  and for any fixed $\rho,u$ and $y$, $\ix_+(\rho,u,y;\cdot)$ has a unique minimizer. From \pref{eq:lasy}, we see that $x\mapsto  \Lambda^*(x;y)$ is differentiable.  If $-\frac{\sfa}{\sqrt{-D''(0)}}< x< \Lambda'(\frac{\sqrt2+y}{2}+;y)$,
  \begin{align*}
    \partial_x \ix_+(\rho,u,y;x) = \frac{\sqrt2+y}{2}+\frac{\sqrt{-D''(0)}(\sfa +\sqrt{-D''(0)}x)}{\sfb^2}.
  \end{align*}
  Solving $\partial_x \ix_+(\rho,u,y;x)=0$ for $x$ gives $ x':=\frac{\sfb^2(\sqrt2+y)}{2D''(0)} - \frac{\sfa}{\sqrt{-D''(0)}}>- \frac{\sfa}{\sqrt{-D''(0)}}$. If $ x' \le  \Lambda'(\frac{\sqrt2+y}{2}+;y)$,
  \[
  \inf_{ -\frac{\sfa}{\sqrt{-D''(0)}}\le
   x\le0}  \ix_+(\rho,u,y;x) = \ix_+(\rho,u,y; x').
  \]
Otherwise, $ \partial_x \ix_+(\rho,u,y;x)<0$ for $-\frac{\sfa}{\sqrt{-D''(0)}}\le x\le \Lambda'(\frac{\sqrt2+y}{2}+;y)$.  Now suppose $ x'>\Lambda'(\frac{\sqrt2+y}{2}+;y)$. In this case $\ix_+(\rho,u,y;\cdot)$ can only attain the minimum for $x>\Lambda'(\frac{\sqrt2+y}{2}+;y)$ for which $\partial_x \ix_+(\rho,u,y;x)$ is given as in \pref{eq:ixder1}. We claim the minimizer is given by $x=\hat x$ as in \pref{eq:ixmin}. Indeed, since $C\sfm(y)<1$, we have
  \[
  C+B\sfm(y)<\sqrt{(C-B\sfm(y))^2+4B}, \ \ C+B\sfm(y)+\sqrt{(C-B\sfm(y))^2+4B}>0.
  \]
This means that $\hat x<0$ while the other possible solution to $\partial_x \ix_+(\rho,u,y;x)=0$ is positive. From the expression of $\partial_x \ix_+(\rho,u,y;x)$, we deduce that $- \frac{\sfa}{\sqrt{-D''(0)}}<\hat x<x'$.	  Finally, we clearly have $ \ix_+(\rho,u,y;x)>0$ for $x\le0$ and thus $\ix^+(\rho,u,y)>0$.
\end{proof}

From the proof, we see that $[-\sqrt2+\sfm(y)] \vee -\frac{\sfa}{\sqrt{-D''(0)}}<\tilde x <x'$ if $x'>-\sqrt2+\sfm(y)$, while $-\frac{\sfa}{\sqrt{-D''(0)}}\le \tilde x \le -\sqrt2+\sfm(y) $ if $x'\le -\sqrt2+\sfm(y)$ for the two cases of the minimizer.
Recall $II_2^k$ from \pref{eq:ii2k}.

\begin{lemma}\label{le:covlet2}
  For any $(\rho',u',y')\in [R_1,R_2]\times \bar E\times [-K,-\sqrt2]$, and $0<\de<1$, we have
  \begin{align*}
    &\limsup_{\de_1\to0+}\limsup_{N\to\8} \frac1N\log II_2^k((u'-\de,u'+\de),(\rho'-\de,\rho'+\de), (y'-\de,y'+\de) ;\de_1,\de')\\
    &\le \frac12\log[-4D''(0)] -\frac12\log D'(0) -\frac12\log(2\pi) \\
    &- (k-1)J_1(\tilde y)+\sup_{\rho'-\de < \rho< \rho'+\de,\atop u'-\de< u< u'+\de, y'-\de< y< y'+\de}\psi_*(\rho,u,y)\\
    &-\inf_{x\ge0} \Big\{ \Lambda^*\Big(\Big[-x- 2\de(1+\frac{\tilde y}{\sqrt{\tilde y^2-2}}) \Big]_-;\tilde y\Big) + \frac1{2\sfb_m^2} (\sfa_o-\sqrt{-D''(0)}x)_+^2 \Big\}.
  \end{align*}
  Here $\tilde y = (y'+\de) \wedge (-\sqrt2-\de'/2)$, $\sfb_m=\sup_{\rho\in (\rho'-\de, \rho'+\de)\cap [R_1,R_2]} \sfb(\rho)$ and
  \[
  \sfa_o= \inf_{(\rho,u,y)\in (\rho'-\de, \rho'+\de)\times (u'-\de, u'+\de)\times( y'-\de, y'+\de)\cap [R_1,R_2]\times \bar E\times [-K,-\sqrt2]} \sfa(\rho,u,y).
  \]
  Moreover,
  \begin{align*}
    &\liminf_{\de\to0+, \atop \de'\to 0+}  \inf_{x\ge0} \Big\{ \Lambda^*\Big(\Big[-x- 2\de(1+\frac{\tilde y}{\sqrt{\tilde y^2-2}}) \Big]_-;\tilde y\Big) + \frac1{2\sfb_m^2} (\sfa_o-\sqrt{-D''(0)}x)_+^2\Big\} =\ix^+(\rho',u',y').
  \end{align*}
\end{lemma}
\begin{proof}
We follow the proof of \pref{le:covlet} with the same notation. Let $y<y'+\de$.  If $y+\frac{\de'}2>-\sqrt2$, by the LDP of $L(\la_{i=1}^{N-1})$ \pref{eq:conineq}, there exists $c=c(\de')$ such that for $N$ large,
$$
\pz((\frac{N-1}{N})^{1/2}\la_k>y+\de')\le \pz(\la_k>-\sqrt2+\de'/2)\le e^{-c N^2}
$$
uniformly for $y>-\sqrt2-\frac{\de'}{2}$.
Using Lemmas \ref{le:goecpt2} and  \ref{le:exptti12}, the contribution from this is exponentially negligible. We only need to consider $y\le-\sqrt2-\frac{\de'}{2}$. For simplicity, let us write $\de_2=\frac{\de'}{2}$. Observe that for $(\frac{N-1}{N})^{1/2}\la_{k-1}<y< (\frac{N-1}{N})^{1/2}\la_{k}-\de'$,
  \begin{align*}
    \sfa_N &\ge \bar \sfa- \frac{\sqrt{-D''(0)}}{N}\sum_{i=k}^{N-1} \frac{Z_i^2}{(\frac{N-1}{N})^{1/2}\la_i-y}\ge \bar\sfa- \frac{\sqrt{-D''(0)}}{N-1}\sum_{i=k}^{N-1} f_{y,\de_2}(\la_i){Z_i^2} =:\sfa_N',
  \end{align*}
which implies that
  \begin{align}\label{eq:aan'}
    -\mathsf{a}_N\Phi(-\frac{\sqrt N\sfa_N}\sfb) +\frac{\sfb}{\sqrt{2\pi N}} e^{-\frac{N {\sfa_N}^2}{2\sfb^2}}\le -\mathsf{a}'_N\Phi(-\frac{\sqrt N\sfa'_N}\sfb) +\frac{\sfb}{\sqrt{2\pi N}} e^{-\frac{N {\sfa'_N}^2}{2\sfb^2}},
  \end{align}
since the function $x\mapsto -x\Phi(-\frac{\sqrt N x}{\sfb })+\frac{\sfb}{\sqrt{2\pi N}}e^{-\frac{N x^2}{2\sfb^2}}$ is strictly decreasing. 

  Using the concentration inequality for chi-square distribution \cite{ML00} and conditioning, for any $t>0$,
  \[
  \pz\Big(\sum_{i=k}^{N-1} f_{y,\de_2}(\la_i) Z_i^2 \ge \sum_{i=k}^{N-1} f_{y,\de_2}(\la_i)+ 2\sqrt{t}\Big(\sum_{i=k}^{N-1} f_{y,\de_2}(\la_i)^2\Big)^{1/2} +2\max_{k\le i\le N-1}f_{y,\de_2}(\la_i)t \Big)\le e^{-t}.
  \]
  Since $\|f_{y,\de_2}\|_L=\de_2^{-2}$ and $\|f_{y,\de_2}^2\|_L=2\de_2^{-3}$, with the concentration inequality for empirical measures of GOE eigenvalues \pref{eq:mms}, we know there exists $c=c(\de_2)>0$ independent of $y$ such that
  \begin{align*}
    &\pz\Big(\Big|\frac1N\sum_{i=k}^{N-1} f_{y,\de_2}(\la_i)-\si_{\rm sc}(f_{y,\de_2})\Big|>N^{-1/4} \text{ or } \Big|\frac1N\sum_{i=k}^{N-1} f_{y,\de_2}(\la_i)^2-\si_{\rm sc}(f_{y,\de_2}^2)\Big|>N^{-1/4} \Big)\\
    &  \le e^{-cN^{3/2}}
  \end{align*}
  uniformly in $y$, where $\si_{\rm sc}(f)=\int f(x)\si_{\rm sc}(\dd x)$. Note that $\si_{\rm sc}(f_{y,\de_2})\le \sfm(-\sqrt2)=\sqrt2$ and $\si_{\rm sc}(f_{y,\de_2}^2)\le \frac1{\de_2^2}$.
  For $N$ large, we have
  \begin{align*}
    \pz\Big(\frac1N\sum_{i=k}^{N-1} f_{y,\de_2}(\la_i) Z_i^2 \ge 2+\frac{2\sqrt{2T}}{\de_2}+\frac{2T}{\de_2} \Big) & \le (1-e^{-cN^{3/2}}) e^{-NT}+e^{-cN^{3/2}}.
  \end{align*}
  It follows that
  \[
  \limsup_{T\to+\8}\limsup_{N\to+\8} \frac1N\log\pz\Big(\frac1{N-1}\sum_{i=k}^{N-1} f_{y,\de_2}(\la_i) Z_i^2 \ge \frac{3T}{\de_2} \Big)=-\8.
  \]
  We choose $T$ large enough so that $\pz(\frac1{ N-1}\sum_{i=k}^{N-1} f_{y,\de_2}(\la_i) Z_i^2 \ge \frac{3T}{\de_2} )$ is exponentially negligible uniformly in $y\in [y'-\de,(y'+\de)\wedge (-\sqrt2-\de_2)]$.
  Let $\eps>0$ be a small number that will be sent to zero later. Since $y'-\de\le y\le \tilde y$, we deduce from \pref{pr:ge2} that for $N$ large enough and $j=0,1,...,[\frac{\sfa_{m+}}{\sqrt{-D''(0)}\eps }]$,
\begin{align*}
&\frac1N\log\sup_{y\le \tilde y} \pz\Big(\la_{k-1}<y, \frac1{ N-1 }\sum_{i=k}^{N-1}f_{y,\de_2}(\la_i){Z_i^2}\ge \sfm(y)+ j\eps \Big) \\
&\le \frac1N\log\sup_{y\le \tilde y} \pz\Big(\la_{k-1}<\tilde y, \frac1{N-1}\sum_{i=k}^{N-1}f_{y,\de_2}(\la_i){Z_i^2}\ge \sfm(\tilde y)+ j\eps +2\de(1+\frac{\tilde y}{\sqrt{\tilde y^2-2}})\Big)\\
&\le -(k-1)J_1(\tilde y) -\Lambda^*\Big(\Big[-j\eps- 2\de(1+\frac{\tilde y}{\sqrt{\tilde y^2-2}}) \Big]_-;\tilde y\Big)  +\eps.
\end{align*}
Here $\sfa_m= \sup_{(\rho,u,y)\in (\rho'-\de, \rho'+\de)\times (u'-\de, u'+\de)\times( y'-\de, y'+\de)\cap [R_1,R_2]\times \bar E\times [-K,-\sqrt2]} \sfa(\rho,u,y).$
From \pref{eq:supa} we know $\sfa_o>-\8$.
Using the above facts, for a large $T$ and $y'-\de\le y\le \tilde y$ we deduce that
  \begin{align*}
    &\ez\Big[  \indi\Big\{(\frac{N-1}{N})^{1/2}\la_{k-1}< y, y+\de'< (\frac{N-1}{N})^{1/2}\la_{k},\frac1{N-1}\sum_{i=k}^{N-1} f_{y,\de_2}(\la_i) Z_i^2 < \frac{3T}{\de_2} \Big\}  \\
    & \quad (-\mathsf{a}'_N\Phi(-\frac{\sqrt N\sfa'_N}\sfb) +\frac{\sfb}{\sqrt{2\pi N}} e^{-\frac{N {\sfa'_N}^2}{2\sfb^2}})\Big]\\
  \le &~\indi\{\sfa_m>0\}\sum_{j=0}^{[\frac{\sfa_m}{\sqrt{-D''(0)}\eps}]} \ez\Big[  \indi\{(\frac{N-1}{N})^{1/2}\la_{k-1}< y, y+\de'< (\frac{N-1}{N})^{1/2}\la_{k} \}   \\
  &\indi\Big\{\frac1{N-1}\sum_{i=k}^{N-1} f_{y,\de_2}(\la_i)Z_i^2\in [\sfm(y)+j\eps, \sfm(y)+(j+1)\eps] \Big\}(-\mathsf{a}_N'\Phi(-\frac{\sqrt N\sfa_N'}\sfb) +\frac{\sfb}{\sqrt{2\pi N}} e^{-\frac{N {\sfa_N'}^2}{2\sfb^2}})\Big] \\
  &+\ez \Big[(-\mathsf{a}'_N\Phi(-\frac{\sqrt N\sfa'_N}\sfb) +\frac{\sfb}{\sqrt{2\pi N}} e^{-\frac{N{\sfa'_N}^2}{2\sfb^2}}) \\
  &\Big(\indi\Big\{(\frac{N-1}{N})^{1/2}\la_{k-1}< y, \sfm(y) +\frac{\sfa_{m+}}{\sqrt{-D''(0)}} < \frac{1}{N-1}\sum_{i=k}^{N-1} f_{y,\de_2}(\la_i){Z_i^2}\le \frac{3T}{\de_2} \Big\}\\
  &+\indi\Big\{(\frac{N-1}{N})^{1/2}\la_{k-1}< y,\frac{1}{N-1}\sum_{i=k}^{N-1} f_{y,\de_2}(\la_i){Z_i^2}< \sfm(y)\Big\}\Big)\Big]\\
\le &~ \Big\lceil \frac{\sfa_{m+}}{\sqrt{-D''(0)}\eps}
    \Big\rceil
    \max_{0\le j\eps \le \frac{\sfa_{m+}}{\sqrt{-D''(0)}} }\Big\{ \Big[-(\sfa-\sqrt{-D''(0)}(j+1)\eps)\Phi(-\frac{\sqrt N (\sfa-\sqrt{-D''(0)}(j+1)\eps)}{\sfb}) \\
    &+\frac{\sfb}{\sqrt{2\pi N}}e^{-\frac{N(\sfa-\sqrt{-D''(0)}(j+1)\eps)^2}{2\sfb^2}} \Big] \pz\Big(\la_{k-1}< y, \frac1{N-1}\sum_{i=k}^{N-1}f_{y,\de_2}(\la_i){Z_i^2}\ge \sfm(y)+ j\eps\Big)\Big\}\\
    &+ \Big(\frac{3T\sqrt{-D''(0)}}{\de_2}+|\bar \sfa| +\frac{\sfb}{\sqrt{2\pi N}} \Big) \pz\Big(\la_{k-1}< y, \frac{1}{N-1}\sum_{i=k}^{N-1} f_{y,\de_2}(\la_i){Z_i^2} >\sfm(y)+ \frac{\sfa_{m+}}{\sqrt{-D''(0)}} \Big)\\
    &+\Big[-\sfa \Phi\Big(-\frac{\sqrt N \sfa }{\sfb}\Big)+\frac{\sfb}{\sqrt{2\pi N}} e^{-\frac{N\sfa^2 }{2\sfb^2}} \Big]\pz(\la_{k-1}< y)\\
\le&~ \Big(\frac{3T\sqrt{-D''(0)}}{\de_2}+|\bar\sfa|+\sfb_m+\Big\lceil \frac{\sfa_m}{\sqrt{-D''(0)}\eps}
    \Big\rceil\Big)\max_{-\eps\le x \le  \frac{\sfa_{m+}}{\sqrt{-D''(0)}} } \Big\{\Big[\frac{\sfb_m}{\sqrt{2\pi N}}e^{-\frac{N(\sfa_o-\sqrt{-D''(0)}(x+\eps))^2}{2\sfb_m^2}}\\
    &-(\sfa_o-\sqrt{-D''(0)}(x+\eps))\Phi(-\frac{\sqrt N (\sfa_o-\sqrt{-D''(0)}(x+\eps))}{\sfb_m}) \Big]\\
    & \exp\Big[-N\Big((k-1)J_1(\tilde y) +\Lambda^*\Big(\Big[-x- 2\de(1+\frac{\tilde y}{\sqrt{\tilde y^2-2}}) \Big]_-;\tilde y\Big) -\eps\Big)\Big]\Big\}.
  \end{align*}
  Since
  \begin{align*}
    &\Phi(-\frac{\sqrt N (\sfa_o-\sqrt{-D''(0)}(x+\eps))}{\sfb_m})\le \indi\{\sfa_o-\sqrt{-D''(0)}(x+\eps)\le 0\} \\
    &\quad + \frac{\sfb_m}{\sqrt{2\pi N}[\sfa_o-\sqrt{-D''(0)}(x+\eps)]} e^{-\frac{N(\sfa_o-\sqrt{-D''(0)}(x+\eps))^2}{2\sfb_m^2}} \indi\{\sfa_o-\sqrt{-D''(0)}(x+\eps)> 0\},
  \end{align*}
using \pref{eq:aan'}, sending $N\to\8$ and then $\eps\to0+,\de_1\to0+$ we find
\begin{align*}
  &\limsup_{\de_1\to0+,\atop \eps\to0+}\limsup_{N\to\8} \frac1N\log II_2^k((u'-\de,u'+\de),(\rho'-\de,\rho'+\de), (y'-\de,y'+\de) ;\de_1,\de')\\
    &\le \frac12\log[-4D''(0)] -\frac12\log D'(0) -\frac12\log(2\pi) \\
    &- (k-1)J_1(\tilde y)+\sup_{\rho'-\de < \rho< \rho'+\de,\atop u'-\de< u< u'+\de, y'-\de< y< y'+\de}\psi_*(\rho,u,y)\\
    &- \min_{0\le x\le \frac{\sfa_{m+}}{\sqrt{-D''(0)}}} \Big\{ \Lambda^*\Big(\Big[-x- 2\de(1+\frac{\tilde y}{\sqrt{\tilde y^2-2}}) \Big]_-;\tilde y\Big) + \frac1{2\sfb_m^2} (\sfa_o-\sqrt{-D''(0)}x)_+^2\Big\}.
\end{align*}
The second assertion follows from continuity.
\end{proof}

\begin{proposition}\label{pr:ubk2}
 Assume $\bar E$ is compact, $k\ge1$ and $R_2<\8$. Then
 \begin{align*}
   &\limsup_{\de'\to0+}\limsup_{\de_1\to0+} \limsup_{N\to\8}\frac1N\log II_2^k( E,(R_1,R_2),(-K,-\sqrt2);\de_1,\de')\le -\frac12\log(2\pi)\\
   & \ +\frac12\log[-4D''(0)] -\frac12\log D'(0)
   +\sup_{(\rho,u,y)\in F} [\psi_*(\rho,u,y)- \ix^+(\rho,u,y)- (k-1) J_1(y)],
 \end{align*}
 where $F=\{(\rho,u,y): \rho\in(R_1,  R_2], u\in\bar E, y\le -\sqrt2\}$.
\end{proposition}
\begin{proof}
  We follow the argument of \pref{pr:ubk0} and assume $R_1>0$. 
  By \pref{eq:psi*lim}, 
for any fixed $\rho>R_1$ and $u\in \bar E$,
$$\limsup_{y\to-\8} \psi_*(\rho,u,y)- \ix^+(\rho,u,y)- (k-1) J_1(y)\le \limsup_{y\to-\8} \psi_*(\rho,u,y)- (k-1) J_1(y)=-\8.
  $$
We may choose $K$ large enough so that $$\sup_{\rho\in(R_1,R_2],u\in\bar E, y\le-\sqrt2} [\psi_*(\rho,u,y)- \ix^+(\rho,u,y)- (k-1) J_1(y)]$$ is attained at a point $(\rho_2^k,u_2^k,y_2^k)\in (R_1,R_2]\times \bar E\times [-K,-\sqrt2]$.

Consider a cover of the compact set $[R_1,R_2]\times \bar E\times [-K,-\sqrt2]$ with cubes of side length $2\de$ and center $(\rho',u',y')$ so that $(\rho_2^k,u_2^k,y_2^k)$ is one of the centers. 
We deduce from \pref{le:covlet2} that
\begin{align*}
  &\limsup_{\de\to0+,\atop \de'\to0+}\limsup_{\de_1\to0+}\limsup_{N\to\8}\frac1N  \log II_2^k(\bar E,(R_1,R_2),(-K,-\sqrt2);\de_1,\de') \\
  & \le \frac12\log[-4D''(0)] -\frac12\log D'(0) -\frac12\log(2\pi)\\
   &+\limsup_{\de\to0+,\atop \de'\to0+} \max_{(\rho',u',y') \atop \text{ centers of cubes}} \Big\{ - (k-1)J_1(\tilde y)+\sup_{\rho'-\de < \rho< \rho'+\de,\atop u'-\de< u< u'+\de, y'-\de< y< y'+\de}\psi_*(\rho,u,y)\\
    &-\inf_{x\ge0} \Big[  \Big(\Big[-x- 2\de(1+\frac{\tilde y}{\sqrt{\tilde y^2-2}}) \Big]_-;\tilde y\Big) + \frac1{2\sfb_m^2} (\sfa_o-\sqrt{-D''(0)}x)_+^2 \Big]\Big\} \\
  &=\frac12\log[-4D''(0)] -\frac12\log D'(0) -\frac12\log(2\pi)+ \psi_*(\rho_2^k,u_2^k,y_2^k)- \ix^+(\rho_2^k,u_2^k,y_2^k)-(k-1)J_1(y_2^k).
\end{align*}
Here we understand that the supremum and infimum were taken within $[R_1,R_2]\times \bar E\times [-K,-\sqrt2)$. We have completed the proof.
\end{proof}

Let
\begin{align}\label{eq:psik}
 \jx_k(\bar E,(R_1,R_2))
  =\max\{ \jx_k^1(\bar E,(R_1,R_2)),  \jx_k^2(\bar E,(R_1,R_2))
   \},
\end{align}
where
\begin{align*}
  \jx_k^1(\bar E,(R_1,R_2))
  &=\sup_{(\rho,u,y)\in F} [\psi_*(\rho,u,y)-k J_1(y)]=\psi_*(\rho_1^k,u_1^k,y_1^k)-k J_1(y_1^k), \\
  \jx_k^2(\bar E,(R_1,R_2))
  &=\sup_{(\rho,u,y)\in F} [\psi_*(\rho,u,y)- \ix^+(\rho,u,y)- (k-1) J_1(y)]\\
&=\psi_*(\rho_2^k,u_2^k,y_2^k)- \ix^+(\rho_2^k,u_2^k,y_2^k)- (k-1) J_1(y_2^k).
\end{align*}
From \pref{eq:snlim}, \pref{eq:ii1k}, \pref{eq:ii2kub}, Propositions \ref{pr:ubk1} and \ref{pr:ubk2}, we conclude
\begin{align}\label{eq:crkub}
  \limsup_{N\to\8}\frac1N\log \ez\Crt_{N,k}(E,(R_1,R_2) )\le  \frac12 \log[-4D''(0)] -\frac12\log D'(0) +\frac12+\jx_k(\bar E,(R_1,R_2)) .
\end{align}

\begin{remark}
  It is crucial to note that $\jx_k^1\le\jx_k^2$ always holds if $\sfa(\rho_1^k,u_1^k,y_1^k)\le0$. This explains why we do not need $\ix^-$ like the case of local minima in the expression of $\jx_k^1$ (which we cannot get from our analysis anyway). However, it could be $J_1(y_1^k)=0$ so that $\jx_k^1=\jx_k^2$. Let us justify the necessity of $\jx_k^2$. In other words, we have to show that $\jx_k^1(\rho_1^k,u_1^k,y_1^k) <\jx_k^2(\rho_1^k,u_1^k,y_1^k)$ is possible.

To this end, it suffices to cook up an example where $J_1(y_1^k)>0$ while $\sfa(\rho_1^k,u_1^k,y_1^k)<0$.  Let us fix the structure function $D$ and leave $\mu,E,R_1,R_2$ free to specify later. Recall the notation from \pref{ex:3} and \pref{eq:m12a}. Using \pref{le:albtd} and continuity, for $\rho>0$ small enough,
\[
2(J^2-\al\bt\rho^2-\bt^2)-(J^2-\bt^2)<0.
\]
Recall from \pref{re:63} that $\partial_y \sfa<0$. By continuity we may find $\de>0$ and $\rho_0>0$ such that
\[
-(J^2-\al\bt\rho^2-\bt^2)\sqrt2 Jy-(J^2-\bt^2)\frac1{\sqrt2}J\sfm(y)<0
\]
for $y\in[-\sqrt2-\de, -\sqrt2]$ and  $\rho\in[\frac12\rho_0, \frac32\rho_0]$. It remains to designate $\frac{J^2 v}{\bt}+\mu<0$ to ensure $\sfa<0$. On the other hand,  from \pref{eq:parypsi} we have for $y<-\sqrt2$,
\[
\partial_y[\psi_*(\rho,u,y)-kJ_1(y)] =  \frac{-(\bt^2+J^2)y-\sqrt2J (\mu+\bt v)}{J^2-\bt^2} +(k+1)\sqrt{y^2-2}.
\]
Let $C(\rho,y)=\frac{1}{\sqrt2 J}[\frac{-(J^2+\bt^2)y}{J^2-\bt^2}+(k+1)\sqrt{y^2-2}]$.  Given any $y\in[-\sqrt2-\de, -\sqrt2)$ and  $\rho\in[\frac12\rho_0, \frac32\rho_0]$, we have $C(\rho,y)>0$ and since $J^2>\bt^2$, the following (in)equalities
\[
\frac{\mu+\bt v}{J^2-\bt^2}=C(\rho,y), \ \ \frac{J^2v}{\bt}+\mu<0
\]
have solutions when $\mu>C(\rho,y)J^2$.

The discussion above allows us to complete our job. Indeed, take any
$$\mu>\max_{\rho\in[\frac12\rho_0,  \frac32\rho_0], y\in[-\sqrt2-\de,-\sqrt2-\frac\de2]} C(\rho,y)J^2 $$
and let
\begin{align*}
T_-(\rho)&=\min_{ y\in[-\sqrt2-\de,-\sqrt2-\frac\de2]}\frac{C(\rho,y)(J^2-\bt^2)-\mu}{\bt(\rho)}\sqrt{D(\rho^2)-\frac{D'(\rho^2)^2 \rho^2}{D'(0)} }+\frac{\mu\rho^2}{2}-\frac{\mu D'(\rho^2) \rho^2}{D'(0) }, \\
T_+(\rho)&= \max_{ y\in[-\sqrt2-\de,-\sqrt2-\frac\de2]}\frac{C(\rho,y)(J^2-\bt^2)-\mu}{\bt(\rho)}\sqrt{D(\rho^2)-\frac{D'(\rho^2)^2 \rho^2}{D'(0)} }+\frac{\mu\rho^2}{2}-\frac{\mu D'(\rho^2) \rho^2}{D'(0) }.
\end{align*}
Since $\partial_y C(\rho,y)<0$, we have $T_-(\rho)<T_+(\rho)$ for any $\rho>0$. By continuity, we may choose $\eps>0$ small enough such that $\rho_0-\eps>\frac12\rho_0, \rho_0+\eps<\frac32\rho_0$ and
$$\max_{\rho_0-\eps\le \rho\le \rho_0+\eps}T_-(\rho)< \min_{\rho_0-\eps\le \rho\le \rho_0+\eps}T_+(\rho).$$
Let $R_1=\rho_0-\eps, R_2=\rho_0+\eps$ and
$E=(\max_{\rho_0-\eps\le \rho\le \rho_0+\eps}T_-(\rho), \min_{\rho_0-\eps\le \rho\le \rho_0+\eps}T_+(\rho))$. Now suppose $(\rho_1^k,u_1^k,y_1^k)$ is a maximizer of $\psi_*(\rho,u,y)-kJ_1(y)$ in $F$ as in \pref{pr:ubk2}. Since
\[
 u_1^k\in \bar E = \bigcap_{\rho_0-\eps\le \rho\le \rho_0+\eps} [T_-(\rho), T_+(\rho)],\quad \rho_0-\eps\le \rho_1^k\le \rho_0+\eps,
\]
we conclude that $y_1^k\in [-\sqrt2-\de,-\sqrt2-\frac\de2]$ and thus $J_1(y_1^k)>0, \sfa(\rho_1^k,u_1^k,y_1^k)<0$.


Using similar ideas, one can argue that $\jx_k^1$ is also needed.  In fact, later on we will  see that when there is no restriction on the critical value, $\jx_k^1$ is larger than $\jx_k^2$.
\end{remark}

\subsection{Lower bound}\label{se:snowe}

Recall $I^k$ and $II^k$ as in \pref{eq:ikdef}.
\begin{proposition}\label{pr:lbk1}
 Suppose $ E$ is open and $R_2<\8$. If $\jx_k^1> \jx_k^2$, then
 \begin{align*}
   \liminf_{N\to\8}\frac1N\log I^k(E,(R_1,R_2))\ge \frac12\log[-4D''(0)] -\frac12\log(2\pi)-\frac12\log D'(0)\\
   +\sup_{(\rho,u,y)\in F} [\psi_*(\rho,u,y)- kJ_1(y)],
 \end{align*}
 where $F=\{(\rho,u,y): \rho\in [R_1,  R_2], u\in\bar E, y\le -\sqrt2\}$.
\end{proposition}

\begin{proof}
   Choose $\de>0$ small enough so that
$$(\rho_1^k-\de,\rho_1^k+\de)\times (u^k_1-\de,u^k_1+\de)\times (y_1^k-\de,y_1^k)\subset (R_1,R_2)\times E\times (-K,-\sqrt2).$$
Here as usual we understand if $\rho_1^k$ (resp.~$u_1^k$) is on the boundary, we can simply replace the interval $(\rho^k_1-\de,\rho^k_1+\de)$ [resp.~$(u^k_1-\de,u^k_1+\de)$] with $(\rho_1^k,\rho_1^k+\de)$ or $(\rho_1^k-\de,\rho_1^k)$ [resp.~$(u_1^k,u_1^k+\de)$ or $(u_1^k-\de,u_1^k)$] in an obvious way. Using \pref{eq:detgk}, by restriction we have
\begin{align*}
  &I^k(E,(R_1,R_2))  \\
\ge & \frac{\sqrt{N}[-4D''(0)]^{N/2}}{(2\pi)^{(N+2)/2} D'(0)^{N/2} } \int_{\rho_1^k-\de}^{\rho_1^k+\de} \int_{u_1^k-\de}^{u_1^k+\de}\int_{y_1^k-\de}^{y_1^k-\de/2}  \frac{e^{-\frac{(u-m_Y)^2}{2\si_Y^2}-\frac{N \mu^2 \rho^2}{2D'(0)} -\frac{N(\sqrt{ -4D''(0)}y+m_2)^2}{2(-2D''(0)-\bt^2)} }} {\si_Y\sqrt{-2D''(0)-\bt^2}}
 \\
  & \ez[ e^{(N-1)\Psi(L((\frac{N-1}{N})^{1/2}\la_1^{N-1}),y)} \indi\{L((\frac{N-1}{N})^{1/2}\la_1^{N-1})\in B(\si_{\rm sc}, \de_1), \la_k<(\frac{N}{N-1})^{1/2}y<\la_{k+1}\} \\
  &(\mathsf{a}_N\Phi(\frac{\sqrt N\sfa_N}\sfb) +\frac{\sfb}{\sqrt{2\pi N}} e^{-\frac{N\sfa_N^2}{2\sfb^2}})]
    \rho^{N-1}  \dd y \dd u  \dd\rho\\
    =: &\frac{\sqrt{N}[-4D''(0)]^{N/2}}{(2\pi)^{(N+2)/2} D'(0)^{N/2} } I_N^k(\de,\de_1).
\end{align*}
Note that for any $y\in [y_1^k-\de, y_1^k-\frac\de2]$ we have $f_{y,\frac\de4}(x)  \ge \frac1{x-y}\indi\{x-y\ge \frac\de4\}$. Let $K>0$ be a large constant such that $K^2>10[kJ_1(y_1^k-\de) +\Lambda^*(\sfm(y_1^k-\frac\de2) -\sfm(y_1^k-2\de);f_{y_1^k-\frac\de2,\frac{\de}{4}})]$. 
Since $x\mapsto x\Phi(\frac{\sqrt N x}\sfb) +\frac{\sfb}{\sqrt{2\pi N}} e^{-\frac{N x^2}{2\sfb^2}}$ is positive and strictly increasing, for $y_1^k-\de< y<y_1^k-\frac\de2\le -\sqrt2-\frac{\de}{2}$,
\begin{align*}
    &\ez\Big[e^{(N-1)\Psi(L((\frac{N-1}{N})^{1/2}\la_1^{N-1}),y)} \indi\{L((\frac{N-1}{N})^{1/2}\la_1^{N-1})\in B(\si_{\rm sc}, \de_1),\la_k< (\frac{N}{N-1})^{1/2} y<\la_{k+1} \}  \\
    &\quad \Big(\mathsf{a}_N\Phi(\frac{\sqrt N\sfa_N}\sfb) +\frac{\sfb}{\sqrt{2\pi N}} e^{-\frac{N\sfa_N^2}{2\sfb^2}}\Big)\Big] \\
\ge~& \ez \Big[e^{(N-1)\Psi(L((\frac{N-1}{N})^{1/2}\la_1^{N-1}),y)}
    \indi\Big\{L((\frac{N-1}{N})^{1/2}\la_1^{N-1})\in B_K(\si_{\rm sc}, \de_1),\\
    &(\frac{N-1}{N})^{1/2}\la_k+2\de < y_1^k<(\frac{N-1}{N})^{1/2}\la_{k+1}+\frac{\de}{4},\frac{1}{N}\sum_{i=1}^{N-1} \frac{Z_i^2}{(\frac{N-1}{N})^{1/2}\la_i-y}\le \sfm(y) \Big\}\\
    &\Big(\mathsf{a}_N\Phi(\frac{\sqrt N\sfa_N}\sfb) +\frac{\sfb}{\sqrt{2\pi N}} e^{-\frac{N{\sfa_N}^2}{2\sfb^2}}\Big) \Big]\\
\ge~& \Big(\sfa
    \Phi(\frac{\sqrt N\sfa}\sfb) +\frac{\sfb}{\sqrt{2\pi N}} e^{-\frac{N\sfa^2}{2\sfb^2}}\Big) \exp\Big\{(N-1)\inf_{\nu\in B(\si_{\rm sc}, \de_1)\cap \px([-K,K]\setminus (y_1^k-2\de, y_1^k-\frac\de4))} \Psi(\nu,y)\Big\} \\
    &\ez\Big[ \indi\Big\{L((\frac{N-1}{N})^{1/2}\la_1^{N-1})\in B_K(\si_{\rm sc}, \de_1),(\frac{N-1}{N})^{1/2}\la_k+2\de < y_1^k<\la_{k+1}+\frac{\de}{4},\\
    &\frac{1}{N} \sum_{i=k+1}^{N-1}  \frac{Z_i^2}{\la_i-(y_1^k-\frac\de2)}\le(\frac{N-1}{N})^{1/2} \sfm(y_1^k-\de)\Big\} \Big]\\
\ge~& \Big(\sfa
    \Phi(\frac{\sqrt N\sfa}\sfb) +\frac{\sfb}{\sqrt{2\pi N}} e^{-\frac{N\sfa^2}{2\sfb^2}}\Big) \exp\Big\{(N-1)\inf_{\nu\in B(\si_{\rm sc}, \de_1)\cap \px([-K,K]\setminus (y_1^k-2\de, y_1^k-\frac\de4))} \Psi(\nu,y)\Big\}\\
    &\Big[ \pz \Big((\frac{N-1}{N})^{1/2}\la_k<y_1^k-2\de, \frac{1}{N} \sum_{i=k+1}^{N-1}  f_{y_1^k-\frac\de2,\frac{\de}{4}}(\la_i){Z_i^2} \le \sfm(y_1^k-2\de) \Big)\\
    &-\pz \Big( \la_{k+1}\le y_1^k-\frac\de4,\frac{1}{N} \sum_{i=k+1}^{N-1}  f_{y_1^k-\frac\de2,\frac{\de}{4}}(\la_i){Z_i^2}\le \sfm(y_1^k-2\de) \Big)\\
    &- \pz(L((\frac{N-1}{N})^{1/2}\la_1^{N-1})\notin B(\si_{\rm sc}, \de_1))-\pz(\la_{N-1}^*>K) \Big].
  \end{align*}
  Note that $\Psi(\nu,y)$ is continuous for $y_1^k-\de< y<y_1^k-\frac\de2$ and $\nu\in B(\si_{\rm sc}, \de_1)\cap \px([-K,K]\setminus (y_1^k-2\de, y_1^k-\frac\de4))$. Using \pref{pr:geup}, we have
  \begin{align*}
    \liminf_{N\to\8} &~\frac1N \log \pz \Big((\frac{N-1}{N})^{1/2}\la_k<y_1^k-2\de, \frac{1}{N} \sum_{i=k+1}^{N-1}  f_{y_1^k-\frac\de2,\frac{\de}{4}}(\la_i){Z_i^2}
    \le \sfm(y_1^k-2\de) \Big) \\
    &\ge -kJ_1(y_1^k-2\de) -\Lambda^*\Big(\sfm(y_1^k-\frac\de2) -\sfm(y_1^k-2\de); f_{y_1^k-\frac\de2,\frac{\de}{4}}\Big),\\
    \limsup_{N\to\8} &~ \frac1N\log \pz \Big(\la_{k+1}\le y_1^k-\frac\de4,\frac{1}{N} \sum_{i=k+1}^{N-1}  f_{y_1^k-\frac\de2,\frac{\de}{4}}(\la_i){Z_i^2}\le \sfm(y_1^k-2\de) \Big)\\
    &\le -(k+1)J_1(y_1^k-\frac\de4) -\Lambda^*\Big(\sfm(y_1^k-\frac\de2) -\sfm(y_1^k-2\de); f_{y_1^k-\frac\de2,\frac{\de}{4}}\Big).
  \end{align*}
  We can always choose $\de$ small enough so that
  \[
  kJ_1(y_1^k-2\de)<(k+1)J_1(y_1^k-\frac\de4).
  \]
Since $\jx_k^1> \jx_k^2$, $\sfa(\rho_1^k,u_1^k,y_1^k)>0$.
Using \pref{eq:phiblim}, \pref{eq:ladein} and the LDP of the empirical measures of GOE eigenvalues, since all functions in question are continuous, we deduce   that
\begin{align*}
  &\liminf_{\de_1\to0+}\liminf_{N\to\8}\frac1N I_N^k(\de,\de_1) \\
  &\ge \inf_{\rho_1^k-\de<\rho<\rho_1^k+\de, u_1^k-\de<u<u_1^k+\de, \atop y_1^k-\de<y<y_1^k-\frac\de2}\psi_*(\rho,u,y) -kJ_1(y_1^k-2\de)-\Lambda^*\Big(\sfm(y_1^k-\frac\de2) -\sfm(y_1^k-2\de); f_{y_1^k-\frac\de2,\frac{\de}{4}}\Big).
\end{align*}
Sending $\de\to0+$, we find by continuity and \pref{re:ylim} if necessary
\begin{align*}
  &\liminf_{N\to\8} \frac1N \log I_1^k(E,(R_1,R_2))\\
  & \ge \frac12\log [-4D''(0)]-\frac12\log(2\pi) -\frac12\log D'(0)+\psi_*(\rho_1^k,u_1^k,y_1^k)- kJ_1(y_1^k).
\end{align*}
The proof is complete.
\end{proof}

\begin{proposition}\label{pr:lbk2}
 Suppose $ E$ is open and $R_2<\8$. Then
 \begin{align*}
   \liminf_{N\to\8}\frac1N\log II^k(E,(R_1,R_2))\ge \frac12\log[-4D''(0)] -\frac12\log(2\pi)-\frac12\log D'(0)\\
   +\sup_{(\rho,u,y)\in F} [\psi_*(\rho,u,y)- \ix^+(\rho,u,y)- (k-1) J_1(y)],
 \end{align*}
 where $F=\{(\rho,u,y): \rho\in (R_1,  R_2], u\in\bar E, y\le -\sqrt2\}$.
\end{proposition}

\begin{proof}
  Choose $\de>0$ small enough so that
$$(\rho_2^k-\de,\rho_2^k+\de)\times (u^k_2-\de,u^k_2+\de)\times (y_2^k-\de,y_2^k)\subset (R_1,R_2)\times E\times (-K,-\sqrt2).$$
Here as usual we understand if $u_2^k$ or $\rho_2^k$ is on the boundary, we can simply replace the corresponding interval by half of itself in an obvious way. Note that by conditioning
\begin{align}
  &\ez[|\det G| \indi\{i(G_{**})=k-1,\zeta<0\}]=\ez(|\det G_{**}| | z_1'-\xi^\sfT G_{**}^{-1} \xi |  \indi\{i(G_{**})=k-1 , \zeta<0\})\notag\\
&= [-4D''(0)]^{\frac{N-1}{2}} \int_{\rz} \ez\Big[ |\det ((\frac{N-1}{N})^{1/2}\GOE_{N-1}- y I_{N-1}) | \indi\{\la_{k-1}<(\frac{N}{N-1})^{1/2} y< \la_{k} \} \notag\\
& \quad (-\mathsf{a}_N\Phi(-\frac{\sqrt N\sfa_N}\sfb) +\frac{\sfb}{\sqrt{2\pi N}} e^{-\frac{N \sfa_N^2}{2\sfb^2}})\Big] \frac{\sqrt {-4N D''(0)} \exp\{-\frac{N(\sqrt{ -4D''(0)}y+m_2)^2}{2(-2D''(0)-\bt^2)}\}} {\sqrt{2\pi(-2D''(0)-\bt^2)} } \dd y. \label{eq:detgk1}
\end{align}
It follows from restriction that
\begin{align*}
  &II^k(E,(R_1,R_2))  \\
  \ge & \frac{\sqrt{N}[-4D''(0)]^{N/2}}{(2\pi)^{(N+2)/2} D'(0)^{N/2} } \int_{\rho_2^k-\de}^{\rho_2^k+\de} \int_{u_2^k-\de}^{u_2^k+\de}\int_{y_2^k-\de}^{y_2^k-\de/2}  \frac{e^{-\frac{(u-m_Y)^2}{2\si_Y^2}-\frac{N \mu^2 \rho^2}{2D'(0)} -\frac{N(\sqrt{ -4D''(0)}y+m_2)^2}{2(-2D''(0)-\bt^2)} }} {\si_Y\sqrt{-2D''(0)-\bt^2}}
 \\
  & \ez[ e^{(N-1)\Psi(L((\frac{N-1}{N})^{1/2}\la_1^{N-1}),y)} \indi\{L((\frac{N-1}{N})^{1/2}\la_1^{N-1})\in B(\si_{\rm sc}, \de_1), \la_{k-1} <(\frac{N}{N-1})^{1/2} y<\la_{k}\} \\
  &(-\mathsf{a}_N\Phi(-\frac{\sqrt N\sfa_N}\sfb) +\frac{\sfb}{\sqrt{2\pi N}} e^{-\frac{N\sfa_N^2}{2\sfb^2}})]
    \rho^{N-1}  \dd y \dd u  \dd\rho\\
    =: &\frac{\sqrt{N}[-4D''(0)]^{N/2}}{(2\pi)^{(N+2)/2} D'(0)^{N/2} } II_N^k(\de,\de_1).
\end{align*}
Recalling $\tilde x$ as in \pref{eq:ix+}, let
\[
s_k=s_k(\rho_2^k,u_2^k,y_2^k)=\begin{cases}
      0, & \mbox{if } \sfa(\rho_2^k,u_2^k,y_2^k)\le0, \\
      \tilde x,  & \mbox{otherwise}.
    \end{cases}
\]
Let $K>0$ be a large constant such that
\[
K^2>10 \Big[(k-1)J_1(y_2^k-2\de) + \Lambda^*\Big([s_k+\sfm(y_2^k-\frac32\de) -\sfm(y_2^k-\frac\de2) ]_-;y_2^k-\frac32\de\Big)\Big].
\]
Since $x\mapsto -x\Phi(-\frac{\sqrt N x}\sfb) +\frac{\sfb}{\sqrt{2\pi N}} e^{-\frac{N x^2}{2\sfb^2}}$ is positive and strictly decreasing, for $y_2^k-\de< y<y_2^k-\frac\de2$,
\begin{align*}
    \ez&\Big[e^{(N-1)\Psi(L((\frac{N-1}{N})^{1/2}\la_1^{N-1}),y)} \indi\{L((\frac{N-1}{N})^{1/2}\la_1^{N-1})\in B(\si_{\rm sc}, \de_1),\la_{k-1}<(\frac{N}{N-1})^{1/2}y<\la_{k} \} \\
    & \Big(-\mathsf{a}_N\Phi(-\frac{\sqrt N\sfa_N}\sfb) +\frac{\sfb}{\sqrt{2\pi N}} e^{-\frac{N\sfa_N^2}{2\sfb^2}}\Big)\Big] \\
\ge~& \ez \Big[e^{(N-1)\Psi(L((\frac{N-1}{N})^{1/2}\la_1^{N-1}),y)}
    \indi\Big\{L((\frac{N-1}{N})^{1/2}\la_1^{N-1})\in B_K(\si_{\rm sc}, \de_1),\\
    &(\frac{N-1}{N})^{1/2}\la_{k-1}+2\de< y_2^k<(\frac{N-1}{N})^{1/2}\la_{k}+\frac{\de}{4},\frac{1}{N}\sum_{i=1}^{N-1} \frac{Z_i^2}{(\frac{N-1}{N})^{1/2}\la_i-y}\ge \sfm(y)- s_k\Big\}\\
    & \Big(-(\sfa+\sqrt{-D''(0)} s_k)\Phi(-\frac{\sqrt N(\sfa+\sqrt{-D''(0)} s_k)}\sfb)+\frac{\sfb}{\sqrt{2\pi N}} e^{-\frac{N(\sfa+\sqrt{-D''(0)} s_k)^2}{2\sfb^2}}\Big) \Big]\\
\ge~& \Big(-(\sfa +\sqrt{-D''(0)}s_k)\Phi(-\frac{\sqrt N(\sfa+\sqrt{-D''(0)}s_k)}\sfb)
    +\frac{\sfb}{\sqrt{2\pi N}} e^{-\frac{N(\sfa+\sqrt{-D''(0)}s_k)^2}{2\sfb^2}}\Big) \\
    &e^{(N-1)\inf_{\nu \in B(\si_{\rm sc}, \de_1)\cap \px([-K,K]\setminus (y_2^k-2\de, y_2^k-\frac\de4))}\Psi(\nu,y)} \\
    & \ez\Big[ \indi\Big\{L((\frac{N-1}{N})^{1/2}\la_1^{N-1})\in B_K(\si_{\rm sc}, \de_1),(\frac{N-1}{N})^{1/2}\la_{k-1}<y_2^k-2\de, \la_{k}>y_2^k-\frac\de4,\\
    &\frac{1}{N} \sum_{i=k}^{N-1}  \frac{Z_i^2}{\la_i-(y_2^k-\frac32\de)}\ge \sfm(y_2^k-\frac\de2)-s_k+\frac{k-1}{N\de}\Big\} \Big]\\
\ge~& \Big(-(\sfa+\sqrt{-D''(0)} s_k)\Phi(-\frac{\sqrt N(\sfa+\sqrt{-D''(0)}s_k)}\sfb)
    +\frac{\sfb}{\sqrt{2\pi N}} e^{-\frac{N(\sfa+ \sqrt{-D''(0)}s_k)^2}{2\sfb^2}}\Big) \\ &e^{(N-1)\inf_{\nu \in B(\si_{\rm sc}, \de_1)\cap \px([-K,K]\setminus (y_2^k-2\de, y_2^k-\frac\de4))}\Psi(\nu,y)} \Big[ \pz \Big((\frac{N-1}{N})^{1/2}\la_{k-1}<y_2^k-2\de, \\
    &\la_{k}>y_2^k-\frac\de4,
    \frac{1}{N} \sum_{i=k}^{N-1}  \frac{Z_i^2}{\la_i-(y_2^k-\frac32\de)}\ge \sfm(y_2^k-\frac\de2)-s_k+\frac{k-1}{N\de} \Big) \\
    &- \pz(L((\frac{N-1}{N})^{1/2}\la_1^{N-1})\notin B(\si_{\rm sc}, \de_1))-\pz(\la_{N-1}^*>K)\Big].
  \end{align*}
We can choose in the beginning $\de$ small enough so that
$$
(k-1)J_1(y_2^k-2\de) < k J_1(y_2^k-\frac\de4).
$$
It follows from \pref{pr:ge2} that
\begin{align*}
  &\liminf_{N\to\8} \frac1N \log\pz \Big(\la_{k-1}<y_2^k-2\de, \la_{k}>y_2^k-\frac\de4,\\
&\quad    \frac{1}{N} \sum_{i=k}^{N-1}  \frac{Z_i^2}{\la_i-(y_2^k-\frac32\de)}\ge \sfm(y_2^k-\frac\de2)-s_k+\frac{k-1}{N\de} \Big)\\
    &\ge -(k-1)J_1(y_2^k-2\de)- \Lambda^*\Big([s_k+\sfm(y_2^k-\frac32\de) -\sfm(y_2^k-\frac\de2) ]_-;y_2^k-\frac32\de\Big).
\end{align*}
Sending $\de\to0+$, using \pref{eq:phib2}, \pref{eq:ix+}, the LDP of $L(\la_1^{N-1})$, continuity of functions in question and \pref{re:ylim} if necessary,
\begin{align*}
  \liminf_{N\to\8} \frac1N \log II^k (E,(R_1,R_2))\ge \frac12\log[-4D''(0)] -\frac12\log(2\pi)-\frac12\log D'(0)\\ +\psi_*(\rho_2^k,u_2^k,y_2^k)- \ix^+(\rho_2^k,u_2^k,y_2^k)- (k-1) J_1(y_2^k).
\end{align*}
The proof is complete.
\end{proof}

\begin{proof}[Proof of Theorem \ref{th:fixk}]
With \pref{eq:crkub}, Propositions \ref{pr:lbk1} and \ref{pr:lbk2}, using the same argument as for \pref{th:criticalfix1}, we have established Theorem \ref{th:fixk}.
\end{proof}

\begin{example}\label{ex:31}
  \rm
  Here we use \pref{th:fixk} to recover \pref{th:fixktotal} for $k\ge1$ when $B_N$ is a shell. Let us recall the setting in \pref{ex:3}. we first consider $ \psi_*(\rho,u,y)-k J_1(y)$ for $y\le -\sqrt2$. In \cite{AZ20}*{Example 2}, we computed
  \begin{align*}
    \partial_{yy}\psi_*&=-\frac{J^2+\bt^2}{J^2-\bt^2} -\frac{|y|}{\sqrt{y^2-2}}\indi\{|y|>\sqrt2\}, \\
    \partial_{yv}\psi_* &=-\frac{\sqrt2 J\bt}{J^2-\bt^2},\  \
    \partial_{vv}\psi_* = -\frac{J^2}{J^2-\bt^2}.
    \end{align*}
Thus for any fixed $\rho$, this is a strictly concave function of $(y,u)$ and there is at most one global maximum. Suppose the global maximum is attained at $(y_1^k,u_1^k)$. Solving $\partial_y \psi_*(\rho,u,y)-k J_1'(y)=0$ with \pref{eq:musbt} yields
  \[
  -y-\frac{\sqrt2\mu}{J}+(k+1)\sqrt{y^2-2}=0, \ \ y_1^k=\frac{\sqrt2\mu-\sqrt2(k+1)\sqrt{\mu^2+k(k+2)J^2}} {Jk(k+2)}.
  \]
  We see that $y_1^k<-\sqrt2$ if and only if $\mu>J$. If $ \mu\le J$, since the global maximum of $\psi_*$ is attained at $y\ge -\sqrt2$ as shown in \cite{AZ20}*{Example 2}, the maximum of $(u,y)\mapsto \psi_*(\rho,u,y)-k J_1(y)$ must be attained at $y=-\sqrt2$, for otherwise this maximum must be a local maximum which would lead to a contradiction.  
All these calculations plugging in \pref{eq:psids0} recover \pref{th:fixktotal} provided $\sup_{(\rho,u,y)\in F} [\psi_*(\rho,u,y)-k J_1(y)]$ always dominates the other term in \pref{th:fixk} for the current non-restriction situation. This is indeed true due to the following result.

\begin{proposition}\label{pr:compj12}
For any fixed $\rho>0$, we have
\[
\sup_{u\in \rz} \psi_*(\rho,u,y)-k J_1(y) > \sup_{u\in \rz} \psi_*(\rho,u,y) -\ix^+(\rho,u,y) -(k-1) J_1(y)
\]
for any $ y<-\sqrt2$.
\end{proposition}
\begin{proof}
Let us write in this proof
\begin{align*}
 L_\rho(y)&= \sup_{u\in \rz} \psi_*(\rho,u,y)-k J_1(y),\\
R_\rho(y)&=\sup_{u\in \rz} \psi_*(\rho,u,y) -\ix^+(\rho,u,y) -(k-1) J_1(y).
\end{align*}
 From \pref{eq:psids0}, we find
\begin{align*}
L_\rho(y)= -\frac12y^2-(k+1)J_1(y)- \frac{\sqrt2 \mu y}{J}  -\frac{\mu^2}{2J^2}-\frac{\mu^2\rho^2}{2D'(0)}+\log \rho -\frac12-\frac12\log2.
\end{align*}
To figure out $R_\rho(y)$, we start with investigating $\ix^+(\rho,u,y) $. Given $\rho$ and $y$, note that the maximizer $u(\rho,y)$ for $R_\rho(y)$ must satisfy $\sfa(\rho,u(\rho,y),y)>0$, for otherwise $\ix^+(\rho,u(\rho,y),y)=0$ and the first order condition \pref{eq:musbt} implies $\sfa(\rho,u(\rho,y),y)>0$ as shown in \pref{ex:3}. By definition, $\ix^+(\rho,u,y) = \ix_+(\rho,u,y;\tilde x) $. Since $\partial_x \ix_+(\rho,u,y; x)|_{x=\tilde x}=0$, using the chain rule with \pref{eq:m12a} to differentiate in $v$, we have
\begin{align}\label{eq:pars1}
&\partial_v[\psi_*(\rho,u,y) -\ix^+(\rho,u,y) -(k-1) J_1(y)]= \frac{-J^2v-\bt(\sqrt2Jy+\mu)}{J^2-\bt^2}-\frac{(\sfa+\frac{J \tilde x}{\sqrt2})J^2\al \rho^2}{\sfb^2(J^2-\bt^2)}.
\end{align}
As in \pref{le:ix+min}, to find out which minimizer we should use, we work with the branch
\[
  \Lambda^*(x;y)=\frac{(\sqrt2+y)x}{2}  - \Lambda(\frac{\sqrt2+y}{2}+;y)
\]
for $x\le \Lambda'(\frac{\sqrt2+y}{2}+;y)= -\sqrt2+\sfm(y)$ and consider $x':=-\frac{\sfb^2(\sqrt2+y)}{J^2} - \frac{\sqrt2\sfa}{J}\le -\sqrt2+\sfm(y)$ as an ansatz. Plugging in \pref{eq:abdef} and  \pref{eq:abdef1} we obtain 
\[
x'=\frac1{J^2-\bt^2}[-\sqrt2(2J^2-2\bt^2-\al^2\rho^4) +(\al^2\rho^4-2\al\bt\rho^2) y-\frac{\sqrt2 \al\rho^2}{J}(J^2v+\mu\bt) ]+\sfm(y).
\]
Setting \pref{eq:pars1} equal to zero and using the ansatz $\sfa+\frac{J\tilde x}{\sqrt2}=-\frac{(\sqrt2+y)\sfb^2}{\sqrt2 J}$, we find
\begin{align}\label{eq:sj}
v=\frac1{J^2}\Big(\frac{(\sqrt2+y)J\al\rho^2}{\sqrt2}-\sqrt2 J\bt y -\mu\bt \Big).
\end{align}
Plugging in the last display for $x'$, after simplification we obtain $x'=-2\sqrt2+\sfm(y)$, which indeed verifies the ansatz $x'\le -\sqrt2+\sfm(y)$. Since the minimizer is unique, we must have $\tilde x= x'$.

Using \pref{eq:sj} to eliminate $u$ in $\psi_*(\rho,u,y) -\ix^+(\rho,u,y)$, we find
\begin{align*}
\psi_*(\rho,u,y)&= -\frac12 y^2-\frac{\al^2\rho^4(\sqrt2+y)^2}{4(J^2-\bt^2)}-J_1(y)-\frac{\sqrt2\mu y}{J}-\frac{\mu^2}{2J^2}-\frac{\mu^2\rho^2}{2D'(0)}+\log \rho -\frac12-\frac12\log2,\\
\ix^+(\rho,u,y)&=\frac14y^2-\frac{\al^2\rho^4(\sqrt2+y)^2}{4(J^2-\bt^2)}-\frac14y\sqrt{y^2-2}+\frac12\log\sfm(y)-\frac12-\frac14\log2.
\end{align*}
From \pref{eq:rf1ev}, we see that $J_1(y)=-\frac12\log2-\frac12y\sqrt{y^2-2}+\log\sfm(y)$ and then
\begin{align*}
R_\rho(y)=-\frac34 y^2-(k+\frac12)J_1(y)-\frac{\sqrt2 \mu y}{J} -\frac12\log2-\frac{\mu^2}{2J^2}-\frac{\mu^2\rho^2}{2D'(0)}+\log \rho.
\end{align*}
It follows that
\begin{align*}
L_\rho(y)-R_\rho(y)=\frac14y^2-\frac12 J_1(y)-\frac12.
\end{align*}
Finally, observing that $y\mapsto L_\rho(y)-R_\rho(y)$ is strictly decreasing for $y<-\sqrt2$ with $L_\rho(-\sqrt2)-R_\rho(-\sqrt2)=0$, we have completed the proof.
\end{proof}

Like what we did in \pref{ex:31}, we can also solve $\partial_y R_\rho(y)=0$ to find the optimizer $y_2^k$, which would further yield $u_2^k$ and $\rho_2^k$. From here we see that the sup in \pref{pr:compj12} can be taken over a compact set instead of $\rz$.  Moreover, the above discussion suggests that the critical points with index $k$ around the value $u_1^k$ and variable $\rho_1^k$ given above dominate all other places.
\end{example}

\begin{remark}
  Observe that in the setting of \pref{re:63}, since $y^0=-\sqrt2$,
  \begin{align*}
    \psi_*(\rho^0,u^0,y^0)-\ix^-(\rho^0,u^0,y^0)< \psi_*(\rho^0,u^0,y^0)\le \psi_*(\rho_1^k,u_1^k,y_1^k) -k J_1(y_1^k).
  \end{align*}
  This shows that the complexity function of index $k$ may be larger than that of the local minima even though we have \pref{re:25} for the total number of critical points with index $k$. This seems to be a new phenomenon, compared to the spherical $p$-spin glass model where the latter always dominates \cite{ABC13}.
\end{remark}

\section{Saddles with diverging index}\label{se:div}
Let $\ga\in(0,1)$ and $k=k_N\in(1,N)$ be a sequence with $\lim_{N\to\8} \frac{k_N}{N}=\gamma$.
Define
\begin{align}
  I^{k_N}(E,(R_1,R_2)) &= \int_{R_1}^{R_2} \int_{E} \ez[|\det G|  \indi\{i(G_{**})=k_N \}]\frac{ e^{-\frac{(u-m_Y)^2}{2\si_Y^2}}}{\sqrt{2\pi}\si_Y} \frac{e^{-\frac{N \mu^2 \rho^2}{2D'(0)}}}{(2\pi)^{N/2} D'(0)^{N/2}}  \rho^{N-1}   \dd u  \dd\rho.
\end{align}

\begin{proposition}\label{pr:divup}
  Let $R_2<\8$ and $E$ be an open set. We have
  \begin{align*}
      &\lim_{N\to\8}\frac1N\log I^{k_N}(E,(R_1,R_2))\\
      &= \frac12\log[-4D''(0)] -\frac12\log D'(0)-\frac12\log(2\pi) +\sup_{R_1<\rho<R_2, u\in  E  }\psi_*(\rho,u,s_\ga).
  \end{align*}
\end{proposition}

\begin{proof}
As before, we may assume $R_1>0$. Note that denoting the probability space as $\Omega$,
\begin{align*}
  I^{k_N}(E,(R_1,R_2))&\le I_1^{k_N}(E,(R_1,R_2))+I_2^{k_N}(E,(R_1,R_2))\\
  &\le I_1(E,(R_1,R_2), \Omega)+I_2(E,(R_1,R_2),\Omega)
\end{align*}
where $I_1^{k_N}$ and $I_2^{k_N}$ are defined as in \pref{eq:ik12}, and $I_1$ and $I_2$ are defined as in \pref{se:exptt}. As in \pref{eq:lakn2} and \pref{eq:sgasc}, for any $\eps>0$, we may find $\de=\de(\eps,\ga)$ and $c=c(\eps,\ga)>0$ such that
\[
\{|\la_{k} -s_\ga|>\eps\}\subset \{ L(\la_{i=1}^{N-1})\notin B(\si_{\rm sc},\de)\}.
\]
It follows from Lemmas \ref{le:exptti11} and \ref{le:exptti12} that for large $K$ and small $\de>0$,
\begin{align*}
  \limsup_{N\to\8} \frac1N\log I_1^{k_N}(E,(R_1,R_2)) & \le \limsup_{N\to\8} \frac1N\log I_1^{k_N}(\bar E,(R_1,R_2),\eps;K,\de),
\end{align*}
where
\begin{align*}
 & I_1^{k}(\bar E,(R_1,R_2),\eps;K,\de)  = [-4D''(0)]^{\frac{N}{2}} \int_{R_1}^{R_2} \int_{\bar E}\int_{s_\ga-\eps}^{s_\ga+\eps} \ez\Big[\Big(\sqrt{\frac2\pi}\frac{\sfb}{\sqrt N} + |\bar\sfa|\Big) \\
  &\prod_{i=1}^{N-1} |(\frac{N-1}{N})^{1/2}\la_i-y | \indi\{\la_k<(\frac{N}{N-1})^{1/2} y<\la_{k+1}, L((\frac{N-1}{N})^{1/2}\la_{i=1}^{N-1})\in B_K(\si_{\rm sc}, \de) \}\Big] \\
  &\frac{ e^{-\frac{(u-m_Y)^2}{2\si_Y^2}}}{\sqrt{2\pi}\si_Y} \frac{e^{-\frac{N \mu^2 \rho^2}{2D'(0)}}}{(2\pi)^{N/2} D'(0)^{N/2}} \frac{ \exp\{-\frac{N(\sqrt{ -4D''(0)}y+m_2)^2}{2(-2D''(0)-\bt^2)}\}} {\sqrt{2\pi(-2D''(0)-\bt^2)} } \rho^{N-1}  \dd y \dd u  \dd\rho.
\end{align*}
Sending $N\to\8$, $\de\to0+$, $\eps\to0+$ sequentially, and using the upper semi-continuity and continuity as proving upper bounds for fixed indices, we find
\begin{align*}
  &\limsup_{\eps\to0+}\limsup_{\de\to0+}\limsup_{N\to\8} \frac1N\log I_1^{k_N}(\bar E,(R_1,R_2),\eps;K,\de)\\
  &\le \frac12\log[-4D''(0)] -\frac12\log D'(0)-\frac12\log(2\pi) +\sup_{R_1<\rho<R_2, u\in\bar E  }\psi_*(\rho,u,s_\ga).
\end{align*}
For $I_2^{k_N}$, we deduce from Lemmas \ref{le:goecpt} and \ref{le:goecpt2} that for large $K$ and small $\de>0$,
\begin{align*}
  \limsup_{N\to\8} \frac1N\log I_2^{k_N}(E,(R_1,R_2)) & \le \limsup_{N\to\8} \frac1N\log I_2^{k_N}(\bar E,(R_1,R_2),\eps;K,\de),
\end{align*}
where
\begin{align*}
  &I_2^k(\bar E,(R_1,R_2); \eps,K,\de) = \frac{[-4D''(0)]^{\frac{N+1}{2}} }{2N} \sum_{i=1}^{N-1}\int_{R_1}^{R_2} \int_{\bar E}\int_{s_\ga-\eps}^{ s_\ga+\eps}  \\
  &\ez\Big[\prod_{j\neq i}^{N-1} |(\frac{N-1}{N})^{1/2}\la_j-y | \indi\{\la_k<(\frac{N}{N-1})^{1/2} y <\la_{k+1}, L((\frac{N-1}{N})^{1/2}\la_{i=1}^{N-1})\in B_K(\si_{\rm sc}, \de) \}\Big] \\
  &\frac{ e^{-\frac{(u-m_Y)^2}{2\si_Y^2}}}{\sqrt{2\pi}\si_Y} \frac{e^{-\frac{N \mu^2 \rho^2}{2D'(0)}}}{(2\pi)^{N/2} D'(0)^{N/2}} \frac{ \exp\{-\frac{N(\sqrt{ -4D''(0)}y+m_2)^2}{2(-2D''(0)-\bt^2)}\}} {\sqrt{2\pi(-2D''(0)-\bt^2)} } \rho^{N-1}  \dd y \dd u  \dd\rho.
\end{align*}
Note that $L((\frac{N-1}{N})^{1/2}\la_{i=1}^{N-1})\in B_K(\si_{\rm sc}, \de) $ implies $L((\frac{N-1}{N})^{1/2}\la_{i=1,i\neq j}^{N-1})\in  B_K(\si_{\rm sc}, 2\de)$ for any $j=1,...,N-1$. Sending $N\to\8$, $\de\to0+$, $\eps\to0+$ sequentially, and using the upper semi-continuity and continuity again, we find
\begin{align*}
  &\limsup_{\eps\to0+}\limsup_{\de\to0+}\limsup_{N\to\8} \frac1N\log I_2^{k_N}(\bar E,(R_1,R_2),\eps;K,\de)\\
  &\le \frac12\log[-4D''(0)] -\frac12\log D'(0)-\frac12\log(2\pi) +\sup_{R_1<\rho<R_2, u\in\bar E  }\psi_*(\rho,u,s_\ga).
\end{align*}
Let $(\rho_\ga,u_\ga)$ be a maximizer of $\psi_*(\rho,u,s_\ga)$ in $[R_1,R_2]\times \bar E$ which exists clearly.
For lower bound, using \pref{eq:absgau} with conditional distribution \pref{eq:z13con0},
\begin{align*}
  \ez[|z_1'-\xi^\sfT G_{**}^{-1} \xi||z_3'=y] & \ge \sqrt{\frac2\pi}\frac{\sfb}{\sqrt N}.
\end{align*}
It follows that
\begin{align*}
  \ez[|\det G| \indi\{i(G_{**})=k \}] &=\ez[\ez(|\det G_{**}| | z_1'-\xi^\sfT G_{**}^{-1} \xi |  \indi\{i(G_{**})=k \})|z_3', \la_1^{N-1}]\\
  &\ge \sqrt{\frac2\pi}\frac{\sfb}{\sqrt N} \ez(|\det G_{**}|\indi\{i(G_{**})=k \})
\end{align*}
Let $\de>0$ be a small constant. Since $G_{**}=\sqrt{\frac{-4D''(0)(N-1)}{N}}(\GOE_{N-1}-(\frac{N}{N-1})^{1/2}z_3' I_{N-1})$, using \cite{ABC13}*{Lemma 3.3} with $m=-\frac{m_2 N^{1/2}}{\sqrt{-4D''(0)(N-1)}}$ and $t^2=\frac{-2D''(0)-\bt^2}{-4(N-1) D''(0)}$,
\begin{align*}
  &I^{k_N}(E,(R_1,R_2)) \\
   &\ge \int_{\rho_\ga-\de}^{\rho_\ga+\de} \int_{u_\ga-\de}^{u_\ga+\de} \sqrt{\frac2\pi}\frac{\sfb}{\sqrt N} \ez(|\det G_{**}|\indi\{i(G_{**})=k \}) \frac{ e^{-\frac{(u-m_Y)^2}{2\si_Y^2}}}{\sqrt{2\pi}\si_Y} \frac{e^{-\frac{N \mu^2 \rho^2}{2D'(0)}}}{(2\pi)^{N/2} D'(0)^{N/2}}  \rho^{N-1}   \dd u  \dd\rho\\
  &= \frac{\sqrt2 \Gamma(\frac{N}2) N^{-\frac{N}2} [-4D''(0)]^{\frac{N}{2}}}{\pi   } \int_{\rho_\ga-\de}^{\rho_\ga+\de} \int_{u_\ga-\de}^{u_\ga+\de}  \frac{\sfb}{ \sqrt{-2D''(0)-\bt^2}} \frac{ e^{-\frac{(u-m_Y)^2}{2\si_Y^2}}}{\sqrt{2\pi}\si_Y} \\
  &\quad \ez_{\GOE(N)}\Big [\exp\Big\{\frac{N(\la_{k+1}^N)^2}{2} -\frac{N (\la_{k+1}^N+ \frac{m_2}{\sqrt{-4D''(0)}})^2 }{\frac{-2D''(0)-\bt^2}{-2 D''(0)}}\Big\}\Big]  \frac{e^{-\frac{N \mu^2 \rho^2}{2D'(0)}}}{(2\pi)^{N/2} D'(0)^{N/2}}  \rho^{N-1}   \dd u  \dd\rho\\
  &=:I^{k_N}((\rho_\ga-\de,\rho_\ga+\de),(u_\ga-\de, u_\ga+\de),\la_{k+1}^N),
\end{align*}
where $\la_{k+1}^N$ is the $(k+1)$th smallest eigenvalue of $\GOE_N$. Since $\frac{k_N+1}{N}\to\ga$, we still have $\pz(|\la_{k+1}^N-s_\ga|>\eps)<e^{-cN^2}$. We can thus restrict $\la_{k+1}^N$ to $(s_\ga-\eps,s_\ga+\eps)$. Sending $N\to\8$, $\de\to0+$, $\eps\to0+$ sequentially, and using the Stirling formula and continuity, we deduce from \pref{eq:phi*} that
\begin{align*}
  &\liminf_{\de\to0+,\atop \eps\to0+}\liminf_{N\to\8} \frac1N\log I^{k_N}((\rho_\ga-\de,\rho_\ga+\de),(u_\ga-\de, u_\ga+\de),\la_{k+1}^N) \\
  & \ge \frac12\log[-4D''(0)] -\frac12\log D'(0)-\frac12\log(2\pi) + \psi_*(\rho_\ga,u_\ga,s_\ga).
\end{align*}
The proof is complete.
\end{proof}

\begin{proof}[Proof of \pref{th:kdiv}]
  As before, it suffices to consider $0<R_1<R_2<\8$. Since $\lim_{N\to\8}\frac{k_N-1}{N}=\ga$, we have
\[
\lim_{N\to\8}\frac1N\log I^{k_N}(E,(R_1,R_2)) = \lim_{N\to\8}\frac1N\log I^{k_N-1}(E,(R_1,R_2)),
\]
where by \pref{pr:divup} the limits exist. Noting that
\begin{align*}
  \{i(G)=k\}&\subset \{i(G_{**})=k\} \cup \{i(G_{**})=k-1\},
\end{align*}
 the upper bound immediately follows from \pref{pr:divup}. For the lower bound, using \pref{eq:krk}, \pref{eq:ikdef}, \pref{eq:detgk} and \pref{eq:detgk1}, we have
 \begin{align*}
   &\ez\Crt_{N,k}(E,(R_1,R_2) ) =S_{N-1} N^{(N-1)/2} [-4D''(0)]^{\frac{N-1}{2}} \int_{R_1}^{R_2}\int_E \\
   &\quad \ez\Big[ \Big( \indi\{\la_k< (\frac{N}{N-1})^{1/2}z_3'<\la_{k+1} \}(\mathsf{a}_N\Phi(\frac{\sqrt N \sfa_N}\sfb) +\frac{\sfb}{\sqrt{2\pi N}} e^{-\frac{N\sfa_N^2}{2\sfb^2}})\\
   &\quad +\indi\{\la_{k-1}<(\frac{N}{N-1})^{1/2} z_3'< \la_{k} \}
 (-\mathsf{a}_N\Phi(-\frac{\sqrt N\sfa_N}\sfb) +\frac{\sfb}{\sqrt{2\pi N}} e^{-\frac{N \sfa_N^2}{2\sfb^2}})\Big)\\
 &\quad |\det ((\frac{N-1}{N})^{1/2}\GOE_{N-1}- z_3' I_{N-1}) | \Big]\frac1{\sqrt{2\pi}\si_Y} e^{-\frac{(u-m_Y)^2}{2\si_Y^2}}  \frac1{(2\pi)^{N/2} D'(0)^{N/2}} e^{-\frac{N \mu^2 \rho^2}{2D'(0)}} \rho^{N-1}   \dd u  \dd\rho.
 \end{align*}
 Here $\sfa_N=\sfa_N(\rho,u,z_3')$ as in \pref{eq:abdef}. Let us define $f_N(x)=x\Phi(\frac{\sqrt N x}{\sfb})+\frac{\sfb}{\sqrt{2\pi N}} e^{-\frac{Nx^2}{2\sfb^2}}$. Using the same argument as that of \cite{ABC13}*{Lemma 3.3} with $m=-\frac{m_2 N^{1/2}}{\sqrt{-4D''(0)(N-1)}}$ and $t^2=\frac{-2D''(0)-\bt^2}{-4(N-1) D''(0)}$, we can rewrite
 \begin{align*}
    &\ez\Big[ \Big( \indi\{\la_k< (\frac{N}{N-1})^{1/2}z_3'<\la_{k+1} \}(\mathsf{a}_N\Phi(\frac{\sqrt N \sfa_N}\sfb) +\frac{\sfb}{\sqrt{2\pi N}} e^{-\frac{N\sfa_N^2}{2\sfb^2}})\\
    &\quad +\indi\{\la_{k-1}<(\frac{N}{N-1})^{1/2} z_3'< \la_{k} \}
  (-\mathsf{a}_N\Phi(-\frac{\sqrt N\sfa_N}\sfb) +\frac{\sfb}{\sqrt{2\pi N}} e^{-\frac{N \sfa_N^2}{2\sfb^2}})\Big)\\
  &\quad |\det ((\frac{N-1}{N})^{1/2}\GOE_{N-1}- z_3' I_{N-1}) | \Big]\\
    &= \frac{\Gamma(\frac{ N}2)N^{-\frac{N-1}{2}}\sqrt{-4D''(0)}}{\sqrt{\pi(-2D''(0)-\bt^2)}} \ez\Big[ \Big(f_N(\hat \sfa_N(\rho,u,\la_{k+1}^N) ) +f_N(-\hat \sfa_N(\rho,u, \la_{k}^N))\Big) \\
    &\quad \exp\Big\{\frac{N(\la_{k+1}^N)^2}{2} -\frac{N (\la_{k+1}^N+ \frac{m_2}{\sqrt{-4D''(0)}})^2 }{\frac{-2D''(0)-\bt^2}{-2 D''(0)}}\Big\} \Big]
 \end{align*}
 where $\la^N_{k+1}$ is the $(k+1)$th smallest eigenvalue of $\GOE_N$, and
 \begin{align*}
   &\hat\sfa_N (\rho,u,\la_{k+1}^N) :=\frac{-2D''(0)\al\rho^2(u-\frac{\mu\rho^2}{2}+\frac{\mu D'(\rho^2) \rho^2}{D'(0) } )}{(-2D''(0)-\bt^2) \sqrt{D(\rho^2)-\frac{D'(\rho^2)^2 \rho^2}{D'(0)} }}
   +\frac{\al\bt\rho^2 \mu }{-2D''(0)-\bt^2}\notag \\
   &\quad  -\frac{(-2D''(0)-\bt^2-\al\bt\rho^2) \sqrt{-4D''(0)} \la_{k+1}^N}{-2D''(0)-\bt^2} -\frac{\sqrt{-D''(0)}}{N}\sum_{i=1,i\neq k+1}^{N}\frac{Z_i^2}{\la^N_i-\la^N_{k+1}}.
 \end{align*}
 Let $(\rho_\ga,u_\ga)$ be a maximizer of $\psi_*(\rho,u,s_\ga)$ in $[R_1,R_2]\times \bar E$. Since $f_N(x)$ is a convex function in $x$, for any $\de>0$, by Jensen's inequality and restriction,
 \begin{align}\label{eq:crtkn}
   &\ez\Crt_{N,k_N}(E,(R_1,R_2) ) \ge  \frac{S_{N-1} \Gamma(\frac{ N}2)  [-4D''(0)]^{{N}/{2}} }{\sqrt{\pi(-2D''(0)-\bt^2)} } \int_{\rho_\ga-\de}^{\rho_\ga+\de} \int_{u_\ga-\de}^{u_\ga+\de} \notag\\
   &\  \Big[f_N(\ez[\hat \sfa_N(\rho,u,\la_{k+1}^N)])+ f_N(-\ez[\hat \sfa_N(\rho,u,\la_{k}^N)]) \notag\\
   &\ -\ez \Big([f_N(\hat \sfa_N(\rho,u,\la_{k+1}^N) )\indi\{|\la_{k+1}^N-s_\ga|>\eps\} +f_N(-\hat \sfa_N(\rho,u, \la_{k}^N))\indi\{|\la_{k}^N-s_\ga|>\eps \}]\Big) \Big]\notag \\
   &\ \exp\Big\{N \inf_{y\in [s_\ga-\eps,s_\ga+\eps]}\Big(\frac{y^2}{2} -\frac{(y+\frac{m_2}{\sqrt{-4D''(0)}})^2} {\frac{-2D''(0)-\bt^2}{-2D''(0)}} \Big) \Big\}\notag\\
   &\  \frac1{\sqrt{2\pi}\si_Y} e^{-\frac{(u-m_Y)^2}{2\si_Y^2}}  \frac1{(2\pi)^{N/2} D'(0)^{N/2}} e^{-\frac{N \mu^2 \rho^2}{2D'(0)}} \rho^{N-1}   \dd u  \dd\rho.
 \end{align}
Since $\pz(\la_k<s_\ga-\eps \text{ or } \la_{k+1}>s_\ga+\eps)< e^{-c N^2}$, using \pref{eq:detgk}, \pref{eq:absgau}, the Cauchy--Schwarz inequality, and the argument of \cite{ABC13}*{Lemma 3.3} again,
\begin{align*}
  &\frac{\Gamma(\frac{ N}2)N^{-\frac{N-1}{2}}\sqrt{-4D''(0)}}{\sqrt{\pi(-2D''(0)-\bt^2)}}\ez [f_N(\hat \sfa_N(\rho,u,\la_{k+1}^N) ) \indi\{|\la_{k+1}^N-s_\ga|>\eps \}]\\
  &=\ez\Big[ \Big( \indi\{\la_k< (\frac{N}{N-1})^{1/2}z_3'<\la_{k+1}, |z_3'-s_\ga|>\eps \}(\mathsf{a}_N\Phi(\frac{\sqrt N \sfa_N}\sfb) +\frac{\sfb}{\sqrt{2\pi N}} e^{-\frac{N\sfa_N^2}{2\sfb^2}})\Big)\\
&\quad |\det ((\frac{N-1}{N})^{1/2}\GOE_{N-1}- z_3' I_{N-1}) | \exp\Big\{-\frac{N(z_3')^2}{2} +\frac{N (z_3' + \frac{m_2}{\sqrt{-4D''(0)}})^2 }{\frac{-2D''(0)-\bt^2}{-2 D''(0)}}\Big\}\Big] \\
&=\ez\Big(|\det ((\frac{N-1}{N})^{1/2}\GOE_{N-1}- z_3' I_{N-1}) | | \zeta | \exp\Big\{-\frac{N(z_3')^2}{2} +\frac{N (z_3' + \frac{m_2}{\sqrt{-4D''(0)}})^2 }{\frac{-2D''(0)-\bt^2}{-2 D''(0)}}\Big\} \\
&\ \ \indi\{\la_k< (\frac{N}{N-1})^{1/2}z_3'<\la_{k+1} , \zeta>0,|z_3'-s_\ga|>\eps\}\Big) \\
&\le \frac{\sqrt{-4N D''(0)}}{\sqrt{\pi(-2D''(0)-\bt^2) }}  \int_{|y-s_\ga|>\eps} \ez\Big[\Big(\ez(|z_1'| | z_3'=y) \prod_{i=1}^{N-1} |(\frac{N-1}{N})^{1/2}\la_i-y|\\
&\ \ +\frac{-2D''(0)}{N} \sum_{i=1}^{N-1} Z_i^2 \prod_{j\neq i} |(\frac{N-1}{N})^{1/2}\la_j-y|\Big)\indi\{\la_k< (\frac{N}{N-1})^{1/2}y <\la_{k+1} \}\Big] e^{-\frac{Ny^2}{2}} \dd y\\
&\le C_D^N (|\bar \sfa|+\sfb ) \Big[\int_{s_\ga+\eps}^\8\ez[(1+(\la_{N-1}^*)^{2N}+|y|^{2N}) ]^{1/2} \pz(\la_{k+1}>s_\ga+\frac\eps2)^{1/2}  e^{-\frac{Ny^2}{2}}  \dd y\\
&\ \ +\int_{-\8}^{s_\ga-\eps}\ez[(1+(\la_{N-1}^*)^{2N}+|y|^{2N}) ]^{1/2} \pz(\la_{k}<s_\ga-\frac\eps2)^{1/2}  e^{-\frac{Ny^2}{2}}  \dd y\Big]\\
&\le  C_D^N (|\bar \sfa|+\sfb ) e^{-cN^2}.
\end{align*}
 From here we deduce that the term involving
 $$\ez \Big([f_N(\hat \sfa_N(\rho,u,\la_{k+1}^N) )\indi\{|\la_{k+1}^N-s_\ga|>\eps\} +f_N(-\hat \sfa_N(\rho,u, \la_{k}^N))\indi\{|\la_{k}^N-s_\ga|>\eps \}]\Big) $$
 in \pref{eq:crtkn} is bounded above by $e^{-cN^2}$ and thus is exponentially negligible as other terms are bounded below by $e^{\Omega(N)}$ as we will show below.
 To compute the expectation $\ez[\hat \sfa_N(\rho,u,\la_{k}^N)]$, we observe that
 \begin{align*}
   0&=\int_{\la_1\le\la_2\le \cdots \le \la_N} \partial_{\la_k} p_{\GOE}(\la_1,...,\la_N) \prod_{i=1}^{N} \dd \la_i \\
    & = \int_{\la_1\le\la_2\le \cdots \le \la_N} \frac1{Z_N} \partial_{\la_k} \prod_{1\le i<j\le N} |\la_i-\la_j| e^{-\frac{N}{2}\sum_{\ell=1}^{N}\la_\ell^2} \prod_{i=1}^{N} \dd \la_i \\
   &= \int_{\la_1\le\la_2\le \cdots \le \la_N} \Big(\sum_{i=1,i\neq k}^{N} \frac1{\la_k-\la_i}- N \la_k \Big) p_{\GOE}(\la_1,...,\la_N)\prod_{i=1}^{N} \dd \la_i .
 \end{align*}
 It follows that
 \begin{align}\label{eq:elak}
   \ez(\la_k^N) & =\frac1N \ez \sum_{i=1,i\neq k}^{N} \frac1{\la_k^N-\la_i^N}.
 \end{align}
 From here we compute
 \begin{align*}
   \ez[\hat \sfa_N(\rho,u,\la_{k+1}^N)]-\ez[\hat \sfa_N(\rho,u,\la_{k}^N)] & = \frac{(-2D''(0)-\bt^2-\al\bt\rho^2) \sqrt{-4D''(0)} \ez(\la_k^N-\la_{k+1}^N)}{-2D''(0)-\bt^2} \\
   & \ +\sqrt{-D''(0)} \ez(\la_{k+1}^N-\la_{k}^N).
 \end{align*}
 It is by far well known that $\ez(\la_{k+1}^N-\la_{k}^N)=o(1)$. In fact, using rigidity of eigenvalues for Wigner matrices \cite{EYY12}, for any given $\eps_0>0$, we have $\ez(\la_{k+1}^N-\la_{k}^N)=O(N^{-1+\eps_0})$ for all $N$ large enough. 
 By convexity and continuity, we have
 $$f_N(\ez[\hat \sfa_N(\rho,u,\la_{k+1}^N)])+f_N(-\ez[\hat \sfa_N(\rho,u,\la_{k}^N)])\ge 2f_N( o(1))$$
 where $o(1)\to 0$ uniformly in $\rho\in[\rho_\ga-\de, \rho_\ga+\de]$ and $u\in[u_\ga-\de,u_\ga+\de]$ as $N\to\8$.
 Using \pref{eq:phi*}, \pref{eq:snlim}, the Stirling formula and continuity again, we find
 \begin{align*}
  & \liminf_{\de\to0+\atop \eps\to 0+}\liminf_{N\to\8} \frac1N\log \ez\Crt_{N,k_N}(E,(R_1,R_2) ) \\
  &\ge \frac12\log[-4D''(0)]-\frac12\log D'(0)+\frac12 + \psi_*(\rho_\ga,u_\ga,s_\ga).
 \end{align*}
 The proof is complete.
\end{proof}

\begin{example} \label{ex:4}\rm
  In the setting of \pref{ex:3}, it is easy to check that the complexity function in \pref{th:kdiv} recovers \pref{th:divtotal}. Indeed, from \pref{eq:psids0} we have for $\mu\neq 0$,
  \begin{align*}
    \psi_*(\rho_\ga,u_\ga,s_\ga) & =
    \begin{cases}
      -\frac12s_\ga^2-1-\frac12\log2-\frac{ \mu s_\ga}{\sqrt{-D''(0)}}-\frac{\mu^2}{-4D''(0)}  +\log\frac{\sqrt{D'(0)}}{|\mu|}, & \mbox{if } R_2>\frac{\sqrt{D'(0)}}{|\mu|}, \\
      -\frac12s_\ga^2-\frac12-\frac12\log2-\frac{ \mu s_\ga}{\sqrt{-D''(0)}}-\frac{\mu^2}{-4D''(0)}-\frac{\mu^2 R_2^2}{2D'(0)}+ \log R_2, & \mbox{otherwise},
    \end{cases}\\
    u_\ga&=-\frac{[D'(\rho^2)-D'(0)](\sqrt{-4D''(0)}s_\ga+\mu)}{-2D''(0)}+ \frac{\mu\rho^2_\ga}{2}-\frac{\mu D'(\rho_\ga^2)\rho_\ga^2}{D'(0)},\\
    \rho_\ga&=\begin{cases}
                \frac{\sqrt{D'(0)}}{|\mu|}, & \mbox{if } R_2>\frac{\sqrt{D'(0)}}{|\mu|}, \\
                R_2, & \mbox{otherwise}.
              \end{cases}
  \end{align*}
  For $\mu=0$,  we have
  \begin{align*}
    \psi_*(\rho_\ga,u_\ga,s_\ga) &=-\frac12s_\ga^2 -\frac12-\frac12\log2 +\log R_2, \\
    u_\ga & = -\frac{[D'(\rho^2)-D'(0)](\sqrt{-4D''(0)}s_\ga)}{-2D''(0)},\\
    \rho_\ga&=R_2.
  \end{align*}
  These expressions match those in \pref{th:divtotal} when $B_N$ is a shell between radii $\sqrt N R_1$ and $\sqrt N R_2$.
\end{example}

\begin{appendix}
\section{Covariances and Hessian distribution}
For the reader's convenience, in this part we list some results proved in \cite{AZ20} that are needed in the text. 
\begin{lemma}[\cite{AZ20}*{Lemma A.1}]\label{le:cov}
  Assume Assumptions I and II. Then for $x\in \rz^N$,
  \begin{align*}
    \Cov[H_N(x), \partial_i H_N(x)]&= D'\left(\frac{\|x\|^2}N\right)x_i,\\
  \Cov[\partial_i H_N(x),\partial_j H_N(x)] &= D'(0)\de_{ij},\\
  \Cov[H_N(x),\partial_{ij} H_N(x)]&= 2D''\left(\frac{\|x\|^2}N\right)\frac{x_ix_j}N +\left[D'\left(\frac{\|x\|^2}N\right)-D'(0)\right]\de_{ij}\\
  \Cov[\partial_k H_N(x), \partial_{ij} H_N(x)]&= 0,\\
  \Cov[\partial_{lk} H_N(x), \partial_{ij} H_N(x)]&= -2D''(0)[\de_{jl}\de_{ik}+\de_{il}\de_{kj} +\de_{kl}\de_{ij}]/N,
  \end{align*}
  where $\de_{ij}$ are the Kronecker delta function.
\end{lemma}

Let $Y= \frac{H_N(x)}N- \frac{D'(\frac{\|x\|^2}N)\sum_{i=1}^{N} x_i \partial_i H_N(x)}{N D'(0)}$ and define
\begin{align}\label{eq:albt0}
  \alpha=\al(\|x\|^2/N) &= \frac{2D''(\|x\|^2/N)}{ \sqrt{ D(\frac{\|x\|^2}N)-\frac{D'({\|x\|^2}/N)^2}{D'(0)}\frac{\|x\|^2}N}}, \notag\\
  \beta=\bt(\|x\|^2/N) & =\frac{D'(\|x\|^2/N)-D'(0)}{\sqrt{ D(\frac{\|x\|^2}N)-\frac{D'({\|x\|^2}/N)^2}{D'(0)}\frac{\|x\|^2}N}}.
\end{align}
Writing $\rho=\frac{\|x\|}{\sqrt N}$, we introduce the following notations
\begin{align}
  m_1 & =m_1(\rho,u)= \mu + \frac{(u-\frac{\mu\rho^2}{2} +\frac{\mu D'(\rho^2)\rho^2}{D'(0)}) ( 2D''(\rho^2)\rho^2+D'(\rho^2) -D'(0) )}{D(\rho^2) -\frac{D'(\rho^2)^2 \rho^2}{D'(0) }}, \notag \\
  m_2&=m_2(\rho,u)= \mu + \frac{(u-\frac{\mu \rho^2}{2} +\frac{\mu D'(\rho^2) \rho^2}{D'(0)}) (D'(\rho^2) -D'(0) )}{D( \rho^2) -\frac{D'(\rho^2)^2 \rho^2}{D'(0) }}, \notag\\
\si_1 & =\si_1(\rho)= \sqrt{\frac{-4D''(0)-(\al \rho^2 +\bt)\al\rho^2}{N}}, \ \ \
  \si_2  =\si_2(\rho)= \sqrt{\frac{-2D''(0)-(\al\rho^2 +\bt)\bt}{N}}, \notag \\
  m_Y&=m_Y(\rho)= \frac{\mu\rho^2}{2}-\frac{\mu D'(\rho^2) \rho^2}{D'(0) }, \ \ \ \si_Y =\si_Y(\rho) =\sqrt{\frac1N\Big(D(\rho^2)-\frac{D'(\rho^2)^2\rho^2}{D'(0)} \Big)},\notag \\
  \alpha &=\alpha(\rho^2)= \frac{2D''(\rho^2)}{ \sqrt{ D(\rho^2)-\frac{D'(\rho^2)^2 \rho^2}{D'(0)}}},  \ \ \
   \beta =\beta(\rho^2)=\frac{D'(\rho^2 )-D'(0)}{\sqrt{ D(\rho^2 )-\frac{D'(\rho^2)^2 \rho^2}{D'(0)}}}, \label{eq:msialbt}
\end{align}
Sometimes, we also use the following change of variable
\begin{align}
v&=\frac{u-\frac{\mu\rho^2}{2}+\frac{\mu D'(\rho^2)\rho^2}{D'(0)}}{\sqrt{D(\rho^2)-\frac{D'(\rho^2)^2\rho^2}{D'(0)}}} = \frac{u-m_Y}{\sqrt{N}\si_Y},\label{eq:uvcov}\\
m_1&=\mu +v(\al \rho^2+\bt), \ \
m_2=\mu+ v\bt.\label{eq:m12cov}
\end{align}
In many situations, we need to take care of the singularity as $\rho \to 0+$. Thus the following result is helpful.
\begin{lemma}[\cite{AZ20}*{Lemma 3.1}]\label{le:albtd}
  We have $\lim_{\rho \to0+} \frac{D(\rho^2)}{\rho^4}-\frac{D'(\rho^2)^2}{D'(0)\rho^2} =-\frac32D''(0)$ and
  \begin{align*}
    \lim_{\rho\to0+} \bt(\rho^2)^2&=-\frac23 D''(0),\quad
    \lim_{\rho\to 0+} \al(\rho^2) \bt(\rho^2)\rho^2  = -\frac43 D''(0), \quad
    \lim_{\rho\to 0+} [\al(\rho^2) \rho^2]^2  = -\frac83 D''(0).
  \end{align*}
\end{lemma}
We have the following conditional distribution which is our basis for complexity analysis.
\begin{proposition}[\cite{AZ20}*{Proposition 3.3}]
Under Assumptions I, II and IV,  we can find deterministic orthogonal matrices $U$ such that
  \begin{align}\label{eq:ayg}
    (U\nabla^2 H_N U^\sfT | Y=u)  & \stackrel{d}{=}
    \begin{pmatrix}
      z_1'& \xi^\mathsf T \\
       \xi & \sqrt{-4D''(0)} (\sqrt{\frac{N-1}{N}}\GOE_{N-1}-z_3'I_{N-1})
    \end{pmatrix}=:G ,
  \end{align}
  where with $z_1,z_2,z_3$ being independent standard Gaussian random variables,
  \begin{align*}
    z_1'&=\si_1 z_1 - \si_2  z_2 + m_1, \quad
     z_3'=\frac1{\sqrt{-4D''(0)}}\Big(\si_2 z_2+ \frac{  \sqrt{\al\bt}\rho }{\sqrt N} z_3 - m_2\Big),
  \end{align*}
  and $\xi$ is a centered column Gaussian vector with covariance matrix $\frac{-2D''(0)}{N}I_{N-1}$ which is independent from $z_1,z_2,z_3$ and the GOE matrix $\GOE_{N-1}$.
\end{proposition}
We write frequently
$$G_{**}=\sqrt{-4D''(0)} \Big(\sqrt{\frac{N-1}{N}}\GOE_{N-1}-z_3'I_{N-1}\Big).$$
As a property of GOE matrices, we may find a random orthogonal matrix $V$ which is independent of the unordered eigenvalues $\tilde \la_j, j=1,...,N-1$ and $z_3'$, such that
\begin{align}\label{eq:goedc}
  G_{**} = \sqrt{-4D''(0)} V^\mathsf{T} \begin{pmatrix}
    (\frac{N-1}{N})^{1/2}\tilde \la_1-z_3' &\cdots  &0  \\
  \vdots& \ddots& \vdots  \\
  0&  \cdots& (\frac{N-1}{N})^{1/2}\tilde\la_{N-1}-z_3'
  \end{pmatrix} V.
  \end{align}
To connect with the Kac--Rice formula \pref{eq:startingpoint}, we have
\begin{align}
&\ez(|\det \nabla^2 H_N|\indi \{i(\nabla^2 H_N)=k\} |Y=u)\notag \\ 
&=\ez(|U\det \nabla^2 H_N U^\sfT|\indi \{i(U\nabla^2 H_N U^\sfT)=k\} |Y=u) = \ez(|\det G|\indi\{i(G)=k\})\label{eq:martin}
\end{align}
where $G$ depends on $u$ implicitly.

We need the following technical lemma for proving Theorems \ref{th:fixktotal} and \ref{th:divtotal}.
\begin{lemma}[\cite{AZ20}*{Lemma B.1}]\label{le:repl}
  Let $\nu_N$ be probability measures on $\rz$ and $\mu\neq0$. Suppose
  $$\lim_{N\to\8} \frac1N\log \int_\rz e^{-\frac12(N+1)x^2 -\frac{(N+1)\mu x}{\sqrt{-D''(0)}} }\nu_{N+1}(\dd x) >-\8.$$
  Then we have
  \begin{align*}
  \lim_{N\to \8} \frac1N \Big(\log \int_\rz e^{-\frac12(N+1)x^2 -\frac{\sqrt{N(N+1)}\mu x}{\sqrt{-D''(0)}} } \nu_{N+1}(\dd x)
  - \log \int_\rz e^{-\frac12(N+1)x^2 -\frac{(N+1)\mu x}{\sqrt{-D''(0)}} }\nu_{N+1}(\dd x)\Big)=0.
  \end{align*}
  \end{lemma}

\section{Exponential tightness}\label{se:exptt}
The following results allow us to reduce our analysis to the compact setting. They were proven via hard analysis in \cite{AZ20} for the sake of all critical points. In this paper, they are applicable because $\Crt_{N,k}(E,(R_1,R_2))\le \Crt_{N}(E,(R_1,R_2))$.
\begin{lemma}[\cite{AZ20}*{Lemma 4.2}]\label{le:exptt}
  Suppose $\mu\neq 0$. 
  Then
  \begin{align*}
  \limsup_{ T\to \8} \limsup_{N\to \8} \frac1N \log  \ez \Crt_{N} ([-T, T]^c, (0,\8)) &= -\8,\\
  \limsup_{ R \to \8} \limsup_{N\to \8} \frac1N \log  \ez \Crt_{N} (\rz, (R,\8)) &= -\8,\\
  \limsup_{ \eps \to 0+} \limsup_{N\to \8} \frac1N \log  \ez \Crt_{N} (\rz, (0,\eps)) &= -\8.
  \end{align*}
  \end{lemma}
\begin{lemma}[\cite{AZ20}*{Lemma 4.3}]\label{le:exptt2}
    Let $\mu=0$ and $R<\8$. Then
    \begin{align*}
    \limsup_{ T\to \8} \limsup_{N\to \8} \frac1N \log  \ez \Crt_{N} ([-T, T]^c, [0,R)) &= -\8.
    \end{align*}
\end{lemma}

For an event $\Delta$ that may depend on the eigenvalues of GOE and other Gaussian random variables in question,  let us define
\begin{align*}
  I_1(E,(R_1,R_2),\Delta) &= [-4D''(0)]^{\frac{N-1}{2}} \int_{R_1}^{R_2} \int_{E} \ez\Big[|z_1'| \prod_{i=1}^{N-1} |(\frac{N-1}{N})^{1/2}\la_i-z_3'| \indi_\Delta\Big]\notag\\
  &\frac{ e^{-\frac{(u-m_Y)^2}{2\si_Y^2}}}{\sqrt{2\pi}\si_Y} \frac{e^{-\frac{N \mu^2 \rho^2}{2D'(0)}}}{(2\pi)^{N/2} D'(0)^{N/2}}  \rho^{N-1}   \dd u  \dd\rho,\\
I_2(E,(R_1,R_2),\Delta) &= \frac{[-4D''(0)]^{\frac{N}{2}} }{2N} \sum_{i=1}^{N-1}\int_{R_1}^{R_2} \int_{E} \ez\Big[Z_i^2 \prod_{j\neq i} |(\frac{N-1}{N})^{1/2}\la_j-z_3'| \indi_\Delta \Big]  \\
  &\frac{ e^{-\frac{(u-m_Y)^2}{2\si_Y^2}}}{\sqrt{2\pi}\si_Y} \frac{e^{-\frac{N \mu^2 \rho^2}{2D'(0)}}}{(2\pi)^{N/2} D'(0)^{N/2}}  \rho^{N-1}   \dd u  \dd\rho.
\end{align*}
\begin{lemma}[\cite{AZ20}*{Lemma 4.5}]\label{le:goecpt}
Suppose $|\mu| +\frac1{R_2}>0$.  Then
\begin{align*}
&\limsup_{K\to\8} \limsup_{N \to\8} \frac1N\log I_2(E, (R_1,R_2), \{\la_{N-1}^*>K\} )=-\8,\\
&\limsup_{K\to\8} \limsup_{N \to\8} \frac1N\log I_2(E, (R_1,R_2), \{|z_3'-\ez(z_3')|>K\} )=-\8.
\end{align*}
\end{lemma}

\begin{lemma}[\cite{AZ20}*{Lemma 4.6}]\label{le:goecpt2}
  Suppose $|\mu| +\frac1{R_2}>0$.  Then for any $\de>0$,
  \[
  \limsup_{N \to\8} \frac1N\log I_2(E, (R_1,R_2), \{L(\la_{1}^{N-1})\notin B(\si_{{\rm sc}},\delta) \} )=-\8.
  \]
  \end{lemma}

\begin{lemma}[\cite{AZ20}*{Lemma 4.7}]\label{le:exptti11}
    Suppose $|\mu|+\frac1{R_2}>0$. Then we have
  \begin{align*}
  &\limsup_{K\to\8} \limsup_{N \to\8} \frac1N\log I_1(E, (R_1,R_2), \{\la_{N-1}^*>K\} )=-\8,\\
  &\limsup_{K\to\8} \limsup_{N \to\8} \frac1N\log I_1(E, (R_1,R_2), \{|z_3'-\ez(z_3')|>K\} )=-\8.
  \end{align*}
\end{lemma}

\begin{lemma}[\cite{AZ20}*{Lemma 4.8}]\label{le:exptti12}
    Let $\de>0$.  Suppose $|\mu|+\frac1{R_2}>0$. Then we have
   \begin{align*}
   \limsup_{N \to\8} \frac1N\log I_1(E, (R_1,R_2), \{L(\la_1^{N-1})\notin B(\si_{\rm sc},\de)\} )=-\8.
   \end{align*}
\end{lemma}

\end{appendix}

\bibliographystyle{plain}
\bibliography{gficrev}

\end{document}